\documentclass[12pt, paper=a4, abstract=true, headings=standardclasses]{scrartcl}

\usepackage{import}
\usepackage[utf8]{inputenc}
\usepackage[T1]{fontenc}
\usepackage[british]{babel}
\usepackage{csquotes}

\usepackage[normalem]{ulem}  

\usepackage[
    pdfpagelabels,
    hypertexnames=false,  
    bookmarksopen=true,
    bookmarksopenlevel=1,  
    colorlinks=true,
    pdfstartview=FitV,
    linkcolor=blue,
    citecolor=blue,
    urlcolor=blue,
]{hyperref} 

\usepackage{xspace}

\usepackage{pifont}

\usepackage[dvipsnames]{xcolor}
\definecolor{forestgreen}{rgb}{0.13, 0.55, 0.13}
\definecolor{amber}{rgb}{1.0, 0.75, 0.0}
\definecolor{bananayellow}{rgb}{.8, 0.6, 0}
\definecolor{uqp}{RGB}{152,24,147}
\definecolor{oxb}{RGB}{0,33,71}

\usepackage{mathtools}
\usepackage{amsfonts}
\usepackage{amssymb}
\usepackage{latexsym}
\usepackage{dsfont} 
\usepackage{bm}  
\usepackage{mleftright} \mleftright{}  

\newcommand{\real}{\mathbb{R}}

\newcommand{\reals}{\mathbb{R}}

\newcommand{\N}{\mathbb{N}}
\newcommand{\R}{\real}

\newcommand{\T}{\intercal}

\newcommand{\refleq}[1]{\stackrel{\mathclap{\mathrm{\cref{#1}}}}{\leq}}
\newcommand{\refl}[1]{\stackrel{\mathclap{\mathrm{\cref{#1}}}}{<}}

\let\brack\undefined
\DeclarePairedDelimiter{\floor}{\lfloor}{\rfloor}

\DeclarePairedDelimiter{\abs}{\lvert}{\rvert}
\DeclarePairedDelimiter{\norm}{\lVert}{\rVert}
\DeclarePairedDelimiter{\paren}{\lparen}{\rparen}
\DeclarePairedDelimiter{\brack}{\lbrack}{\rbrack}
\DeclarePairedDelimiter{\dotprod}{\langle}{\rangle}
\DeclarePairedDelimiterXPP{\bigO}[1]{\mathcal{O}}{\lparen}{\rparen}{}{#1}
\DeclarePairedDelimiterXPP{\bigOmega}[1]{\Omega}{\lparen}{\rparen}{}{#1}

\newcommand{\vek}[1]{\bm{#1}}
\newcommand{\mat}[1]{\bm{#1}}
\newcommand{\tns}[1]{\textsf{\textbf{#1}}}

\renewcommand{\aa}{\mathbf{a}}

\newcommand{\bb}{\mathbf{b}}

\renewcommand{\gg}{\mathbf{g}}

\newcommand{\HH}{\mathbf{H}}

\newcommand{\eye}{\mathbf{I}}

\renewcommand{\ss}{\mathbf{s}}

\newcommand{\UU}{\mathbf{U}}
\newcommand{\uu}{\mathbf{u}}

\newcommand{\VV}{\mathbf{V}}
\newcommand{\vv}{\mathbf{v}}

\newcommand{\ww}{\mathbf{w}}

\newcommand{\xx}{\mathbf{x}}

\newcommand{\yy}{\mathbf{y}}

\newcommand{\deltapred}{\delta^{\operatorname{pred}}}
\newcommand{\deltaact}{\delta^{\operatorname{act}}}

\newcommand{\eps}{\varepsilon}

\usepackage[linesnumbered, lined, ruled, algosection]{algorithm2e}
\DontPrintSemicolon
\SetAlgoHangIndent{0pt}
\SetKwInput{KwParameters}{Parameters}
\SetKwInput{KwDefaults}{Defaults}
\SetKwInput{KwNote}{Note}

\renewenvironment{procedure}{\algorithm}{\endalgorithm}
\AtBeginEnvironment{procedure}{
    \SetAlgorithmName{Procedure}{Procedure}{}
}

\usepackage{aliascnt}
\usepackage[
    capitalise,
    nameinlink,
    noabbrev,
]{cleveref}

\usepackage{tikz}
\usetikzlibrary{shapes, arrows, positioning}

\usepackage{graphicx}
\usepackage{caption}
\usepackage{subcaption}
\usepackage{float}

\usepackage{booktabs}  
\usepackage{multirow}
\usepackage{array}  
\usepackage{makecell} 
\usepackage{siunitx} 
\usepackage{tabularx} 
\usepackage{diagbox} 
\usepackage{rotating}

\usepackage{arydshln}
\usepackage[shortlabels]{enumitem}

\usepackage{authblk}

\usepackage{amsthm}
\usepackage[framemethod=TikZ]{mdframed}

\mdfsetup{
    linewidth=1pt,
    roundcorner=10pt,
    leftmargin=0pt,
    rightmargin=0pt,
    innertopmargin=0pt,
    innerbottommargin=10pt,
    outerlinecolor=black, 
}

\newmdtheoremenv[
    backgroundcolor = green!3,
]{theorem}{Theorem}[section]

\newaliascnt{corollary}{theorem}
\newmdtheoremenv[
    backgroundcolor = green!3,
]{corollary}[corollary]{Corollary}
\aliascntresetthe{corollary}

\newaliascnt{lemma}{theorem}
\newmdtheoremenv[
    backgroundcolor = green!3,
]{lemma}[lemma]{Lemma}
\aliascntresetthe{lemma}

\newaliascnt{proposition}{theorem}
\newmdtheoremenv[
    backgroundcolor = green!3,
]{proposition}[proposition]{Proposition}
\aliascntresetthe{proposition}

\newaliascnt{definition}{theorem}
\newmdtheoremenv[
    backgroundcolor = blue!3,
]{definition}[definition]{Definition}
\aliascntresetthe{definition}

\newaliascnt{condition}{theorem}
\newmdtheoremenv[
    backgroundcolor = yellow!3,
]{condition}[condition]{Condition}
\aliascntresetthe{condition}

\newaliascnt{assumption}{theorem}
\newmdtheoremenv[
    backgroundcolor = yellow!3,
]{assumption}[theorem]{Assumption}
\aliascntresetthe{assumption}

\theoremstyle{definition}
\newaliascnt{example}{theorem}
\newmdtheoremenv[
    backgroundcolor = cyan!3,
]{example}[example]{Example}
\aliascntresetthe{example}

\theoremstyle{definition}
\newaliascnt{remark}{theorem}
\newmdtheoremenv[
    backgroundcolor = red!3,
]{remark}[remark]{Remark}
\aliascntresetthe{remark}

\crefname{equation}{}{}
\crefname{figure}{Figure}{Figures}
\creflabelformat{equation}{\textup{(#2#1#3)}}
\crefname{assumption}{Assumption}{Assumptions}
\crefname{condition}{Condition}{Conditions}
\crefname{procedure}{Procedure}{Procedures}

\usepackage[
  backend=biber,
  style=numeric,
  eprint=false,
  maxbibnames=10,
]{biblatex}

\AtEveryBibitem{
  \iffieldundef{doi}{
  }{
    \clearfield{url}\clearfield{urlyear}
  }
}

\usepackage{pgf}

\usepackage[normalem]{ulem}

\newcommand{\ignore}[1]{}

\title{Efficient Implementation of Third-Order Tensor Methods with Adaptive Regularization for Unconstrained Optimization}

\author[1]{Coralia Cartis}
\author[1]{Raphael Hauser}
\author[1]{Yang Liu}
\author[1]{Karl Welzel}
\author[1]{Wenqi Zhu}
\affil[1]{Mathematical Institute, University of Oxford\footnote{The order of the authors is alphabetical; Yang Liu and Karl Welzel are the primary contributors.}\thanks{This work was supported by the Hong Kong Innovation and Technology Commission (InnoHK Project CIMDA).}}
\affil[ ]{\scriptsize{\textit{\{coralia.cartis, raphael.hauser, yang.liu, karl.welzel, wenqi.zhu\}@maths.ox.ac.uk}}}

\date{\today}

\begin{document}

\maketitle

\begin{abstract}
	High-order tensor methods that employ local Taylor models of degree $p$ within adaptive regularization frameworks (AR$p$) have recently received significant attention, due to their improved/optimal global and local rates of convergence, for both convex and nonconvex optimization problems.
    The numerical performance of tensor methods for general unconstrained optimization problems remains insufficiently explored/understood, which we address in this paper, by showcasing the numerical performance of standard second- and third-order variants ($p=2,3$) and proposing novel techniques for key algorithmic aspects when $p\geq 3$, to improve the numerical efficiency of tensor variants.
    In particular, to improve the adaptive choice of the regularization parameter, we extend the interpolation-based updating strategy introduced in [Gould, Porcelli and Toint, \textit{Comput Optim Appl} (2012) 53:1--22] for $p=2$, to the case when $p \geq 3$.
    We identify fundamental differences between the different local minima of the regularised subproblems for $p=2$ and $p \geq 3$ and their effect on algorithm performance.
    Then, when $p\geq 3$, we introduce a novel pre-rejection technique that rejects poor/unsuccessful subproblem minimizers (that we refer to as `transient') prior to any function evaluation, thereby reducing cost and selecting useful (`persistent') ones.
    Numerical studies showcase the efficiency improvements generated by our proposed modifications of the AR$3$ algorithm.
    We also assess numerically, the effect of different subproblem termination conditions and choice of initial regularization parameter on the overall algorithm performance.
    Finally, we benchmark our best-performing AR$3$ variants, as well as those in [Birgin et al., \textit{Optim Lett} (2020) 14:815--838], against second-order ones (AR$2$).
    Encouraging results on standard test problems are obtained, confirming that AR$3$ variants can be made to outperform second-order variants in terms of objective evaluations, derivative evaluations, and number of subproblem solves.
    We provide an efficient, extensive and modular software package in MATLAB that includes many AR$2$ and AR$3$ variants, including Hessian- and tensor-free ones, allowing ease of use and experimentation for interested users.
\end{abstract}

\section{Introduction}

In this paper, we consider the unconstrained nonconvex optimization problem,
\begin{equation*}
	\min_{\xx \in \real^d} f(\xx)
\end{equation*}
where $f \colon \real^d \rightarrow \real $ is $p$-times continuously differentiable ($p \ge 1$) and bounded below.
Recently, adaptive $p$th-order regularization (AR$p$) methods have been proposed \cite{birgin2017worst,cartis2022evaluation} for such objectives, that can naturally incorporate higher derivatives  into their construction.
In these methods, a local model $m_k(\ss)$ is constructed to approximate $f(\xx + \ss)$ based on a regularized $p$th-order Taylor expansion of $f$ around the current iterate $x_k$,
\begin{equation}
	\label{eqn:Taylor}
	t_k (\ss) = t^p_{\xx_k}(\ss) = f(\xx_k) + \sum_{j=1}^p \frac{1}{j!} \nabla^j f(\xx_k)[\ss]^j,
\end{equation}
where $k$ is the iteration counter, $\nabla^j f(\xx_k) \in \real^{\otimes^j d}$ is a $j$th-order tensor, $\real^{\otimes^j d}$ denotes the $j$-fold tensor product of $\real^d$ and $\nabla^j f(\xx_k)[\ss]^j$ is the $j$th derivative of $f$ at $\xx_k$ along direction $\ss \in \real^d$.\footnote{This notation is equivalent to the multi-index notation for the multidimensional Taylor expansion $t_{\xx_k}^p(\ss) = \sum_{\abs{\alpha} \leq p} \frac{1}{\alpha!} D^{\alpha} f(\xx_k) \ss^{\alpha}$ since $\frac{1}{j!} \nabla^j f(\xx_k)[\ss]^j = \sum_{\abs{\alpha} = j} \frac{1}{\alpha!} D^{\alpha} f(\xx_k) \ss^{\alpha}$.}
To ensure that the local model $m_k$ is bounded below, a $(p+1)$st-order regularization term scaled by $\sigma_k > 0$ is added to $t_k$. The model $m_k$ is given by\footnote{Unless otherwise stated, $\|\cdot\|$ denotes the Euclidean norm in this paper.}
\begin{equation}
	m_k(\ss) = m^p_{\xx_k, \sigma_k}(\ss) = t_k (\ss)+ \frac{\sigma_k}{p+1}\|\ss\|^{p+1}.
	\label{eqn:subproblem}
\end{equation}
The iterate $\xx_k$ is then updated by approximately minimizing this model to find a step \(\ss_k\), and setting $\xx_{k+1} = \xx_k + \ss_k$, provided that sufficient decrease in the objective function is achieved; the regularization parameter $\sigma_k$ is adjusted adaptively to ensure progress towards optimality over the iterations.
This process continues until an approximate local minimizer of $f$ is found.

Provided Lipschitz continuity assumptions on the $p$th-order derivative hold,  the AR$p$ algorithm requires no more than $\bigO{\eps^{-(p+1)/p}}$ evaluations of $f$ and its derivatives to compute an approximate first-order stationary point that satisfies $\norm{\nabla f(\xx_k)} \leq \eps$, which is an optimal bound for this function class \cite{carmon2020lower,cartis2022evaluation} as a function of the accuracy $\epsilon$.
Thus, as we increase the order $p$, the global evaluation complexity bound improves.
The same holds for the rate of local convergence:
the AR$p$ method achieves a $p$th-order local rate if the objective function is strongly convex and $\sigma_k$ is chosen large enough, depending on the Lipschitz constant of $\nabla^p f$ \cite{doikov2022local}.
These results, showing superior performance of higher-order methods in a worst-case sense, motivate us to develop an efficient algorithmic implementation of the AR$p$ method, in order to investigate the typical performance of these methods.


When it comes to first- and second-order methods, the practical performance of the latter is typically superior to the former, even with inexact second-derivatives, and particularly when a high accuracy solution is sought for an ill-conditioned problem
\cite{martens2010deep,nocedal1999numerical,erdogdu2015convergence,liu2021convergence}. These numerical findings are backed by theoretical analyses and guarantees of performance in terms of problem evaluations, as well as the development/use of efficient subproblem solvers so that the overall computational cost of second-order methods is improved.

In a similar vein, it is natural to ask whether third- or higher-order regularization methods, such as AR$p$ with $p\geq 3$, can showcase superior practical performance by comparison to Newton-type algorithms. This is an ongoing investigation even for performance measures that ignore the computational cost of the subproblem solution (that now involves the (local) minimization of higher-degree multivariate polynomials). Here, we also focus on this question in the context of AR$p$ methods for $p\geq 3$, with particular numerical focus on AR$3$.
Before giving further technical details, we illustrate the potential benefits of AR$3$ in \cref{exp:harpin_slalom}.


\begin{example}
	\label{exp:harpin_slalom}
	\begin{figure}[H]
		\centering
		\includegraphics[width=.50\textwidth]{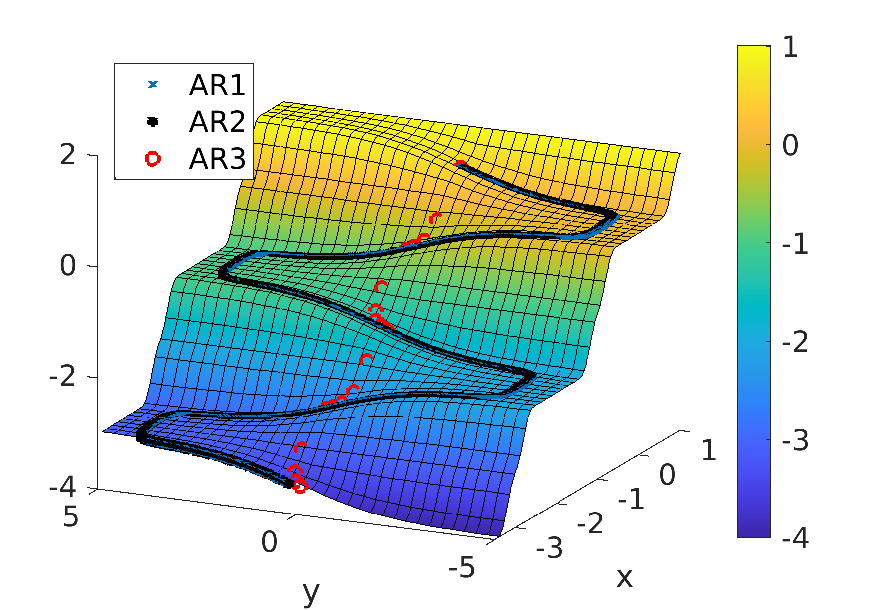}
		\includegraphics[width=.487\textwidth]{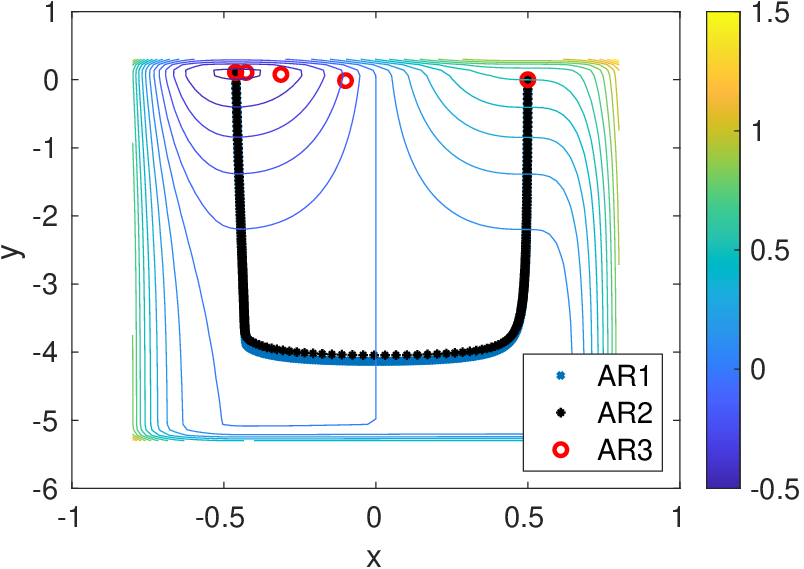}
        \caption{Two examples, namely \enquote*{slalom} (left) and \enquote*{hairpin turn} (right), illustrate the potential benefits of increasing $p$.  By leveraging curvature information, AR$2$ requires fewer iterations and can take a turn earlier than AR$1$. However, AR$3$ is capable of employing  higher-order curvature information, resulting in a dramatic decrease in the number of iterations. The total number of iterations for the \enquote*{slalom} experiment using AR$1$, AR$2$, and AR$3$ are $29\,529$, $935$, and $17$ with corresponding running times of $131.99$, $2.91$, and $0.08$ seconds, respectively, while the total number of iterations for the \enquote*{hairpin turn} experiment are $7\,551$, $230$, and $5$ with corresponding running times of $10.80$, $0.62$, and $0.02$ seconds, respectively. For more details about the construction of these functions, see \cref{sec:slalom_hairpin}.}
		\label{fig:slalom_hairpin}
	\end{figure}
    The following examples illustrate the potential for superior performance of third-order methods compared to first- and second-order methods. We constructed two functions (`slalom' and `hairpin turn') that incorporate key information about the local shape of the function inside the third-order derivative, which we plot in \cref{fig:slalom_hairpin}, together with the path of iterates generated by applying AR$1$, AR$2$, and AR$3$ to these functions. The left image in \cref{fig:slalom_hairpin} shows the periodic and unbounded slalom function \cref{eqn:slalom_function}. AR$1$ and AR$2$ both take small steps along a curving path. Only in the AR$3$ method does the local Taylor expansion include crucial information about the curvature in the negative $x$ direction, enabling the iterates to follow the steep \enquote{downhill} path. The image on right-hand side of \cref{fig:slalom_hairpin} is using the hairpin turn function \cref{eqn:hairpin_function}, which has the same structure as one turn of the slalom function but is modified such that it has a global minimum close to $(-0.5, 0)^\T$. Again, knowledge of the third derivative allows AR$3$ to find the minimizer using significantly fewer iterations than AR$1$ and AR$2$.

    In terms of details, these iterates were generated using the AR$p$ method (\cref{alg:arp_framework}), with $p=1,2,3$, where the starting point was set to $\xx_0 = (0.5, 0)^\T$ for all experiments,  $\sigma_0=50$ and $\sigma_k$ is non-decreasing, namely, $\sigma_k$ is increased when an iteration is unsuccessful (in the sense of \cref{alg:simple_sigma_update}) but never decreased.
	Note that $\sigma_k$ remained constant during both \enquote*{slalom} and \enquote*{hairpin turn} experiments, i.e., all iterations in \cref{fig:slalom_hairpin} are successful. Each subproblem is solved by the Matlab steepest descent algorithm starting at $\ss_0 = (-0.1, 0)^\T$ and terminating when $\norm{\nabla m_k(\ss)} < 10^{-9}$. The subproblem starting point is slightly shifted toward the direction of rapid decrease to provide additional opportunities for the subproblem solver to find minimizers with smaller objective function values in each iteration.
    Please see \cref{sec:slalom_hairpin} for more details on the construction of these functions.
\end{example}

\paragraph{Existing literature}

In the case of $p=1$, AR$1$ is a first-order method, and the minimizer of the model in \cref{eqn:subproblem} is a steepest descent step. The case of $p=2$ (AR$2$) recovers the well researched (adaptive) cubic regularization framework (ARC) and its associated optimal global rate of convergence for second order methods in the worst-case \cite{nesterov2006cubic,cartis2011adaptiveI,cartis2011adaptiveII,dussault2018arcq,kohler2017sub,martinez2017cubic}.
Numerically, in the ARC framework, efficient iterative algorithms are available for finding the global minimizer of the subproblem \cite{cartis2022evaluation,cartis2011adaptiveI}.
Furthermore, specific strategies for efficiently choosing the regularization parameter adaptively can be found in \cite{cartis2011adaptiveI,cartis2009trust,gould2012updating,griewank1981modification,nesterov2006cubic}.

The first paper to propose the AR$p$ algorithmic framework (with $p\geq 3$) and its associated first-order global rate of convergence for general nonconvex objectives is \cite{birgin2017worst}, with a similarly early attempt to present high-order trust region methods and higher order criticality measures  given in \cite{cartis2018FoCM}. Further works followed, showing second-order criticality results for AR$p$ \cite{cartis2020concise,cartis2020sharp} and lower complexity bounds in \cite{carmon2020lower,cartis2020sharp,carmon2021lower} that show the AR$p$ complexity bounds are optimal within a respective class of sufficiently smooth objectives. Tensor methods for convex optimization problems have been pioneered by
Nesterov starting with \cite{nesterov2021implementable} and further extended in \cite{nesterov2023inexact,nesterov2021inexact,nesterov2021superfast,nesterov2022quartic} and several other works. In terms of local convergence, see \cite{doikov2022local,welzel2025local}. Further details can be found in Chapter 4 of \cite{cartis2022evaluation}.

However, all these developments have remained generally theoretical when $p\geq 3$, so that the efficient implementation of higher-order methods (for $p \geq 3$) remains an active area of investigation in which the community has not yet settled on a consensus. Earlier works by Schnabel et al.~\cite{chow1989derivative,schnabel1991tensor,schnabel1984tensor} resulted in a practical tensor algorithm for solving unconstrained optimization problems. Schnabel et al construct a fourth-order local model based on exact first- and second-order information and choose the third- and fourth-order term such that the model interpolates previous function values and gradients. This model is used within a trust-region framework to ensure consistent progress, with successful and promising numerical results.

Recently, the AR\(p\) algorithm with \(p = 3\) was implemented (in Fortran) by Birgin et al~\cite{birgin2020use}, and an increasing regularization parameter update was developed for better efficiency, as well as a step rejection mechanism based on step length, in order to safeguard against possible pitfalls of third-order models. Their numerical experiments compared the performance of AR$2$/ARC and AR$3$ for a standard test set of unconstrained optimization problems, with promising results in terms of function and derivative evaluations. We discuss at length their implementation in later sections and benchmark ours against theirs as well.


Other recent developments for tensor methods include
 \cite{ahmadi2024higher}, that introduces a convex higher-order Newton model with polynomial work per iteration, with adaptive regularization parameter updates and extensions to the nonconvex case  developed in \cite{zhu2024global}. However, these tensor variants rely on sums-of-squares
  semidefinite programming reformulations, which do not scale well with problem dimension. Purposely-designed algorithms for solving the (non-convex) AR$3$ subproblem have been proposed in \cite{cartis2023second,zhu2022quartic,zhu2023cubic}, the first of which  generalizing an idea by Nesterov \cite{nesterov2022quartic} from the convex to the nonconvex case.

\paragraph{Contributions} The numerical performance of tensor methods for general unconstrained optimization problems remains insufficiently explored/understood. Thus the first aim of our paper is to present an in-depth numerical study of their key features on standard test problems, in the case of third-order methods with fourth-order regularization (AR$3$), contrasting them to second-order methods' behavior. This investigation highlights the challenges that tensor methods bring in order to make them numerically efficient, which we then aim to improve by proposing some novel procedures.
In particular, we analyze how certain algorithmic components affect the number of function and derivative evaluations when moving from $p = 2$ to $p = 3$. Our viewpoint is primarily algorithmic, and independent of the particular choice of subproblem solver (which we keep fixed throughout our main experiments).

A first crucial aspect is the choice of the regularization parameter (initial value and adaptive update scheme) and associated iterate update, which is also similarly important to the performance of second-order regularization methods. For the latter, an interpolation-based update procedure was designed in \cite{gould2012updating}, that uses information from previous iterations to estimate an appropriate regularization parameter for the current iteration, thus typically improving the overall evaluation counts of AR$2$. Here, we extend this scheme to AR$p$ methods with $p\geq 3$, and extensively test it numerically for AR$3$; see \cref{sec:interpolation-sigma-update}. Practical proposals for the initial choice of $\sigma_k$, $\sigma_0$, are also  introduced; see \cref{sec:sigma0}.

An essential difference between AR$2$ and AR$3$ is that the latter's subproblem can have multiple local minima even when the Hessian is positive definite, as long as the regularization parameter is sufficiently small (which may be desirable as it allows long steps). Although choosing any of these local minimizers is sufficient to ensure the optimal worst-case complexity bound, the practical efficiency of the ensuing AR$3$ variant is highly dependent on this choice, as well as the subproblem termination condition (whether absolute or relative error conditions are used). Thus, we investigate
the fundamental differences in AR$p$ subproblem solutions when $p \geq 3$ (compared to second-order variants), namely, the discontinuous changes in subproblem solutions as a function of changing regularization parameter values, that may lead to unsuccessful steps and slow progress. We propose a practical technique that detects
so-called  (directionally) `transient' minimizers and pre-rejects them without evaluating the objective, since they would not improve its value; aiming to select as trial step a `persistent' subproblem minimizer, that would be successful; \cref{sec:prerejection} addresses AR$p$ subproblem solutions and the pre-rejection technique which we then show improves AR$3$ performance significantly.

Afterwards, our best-performing AR$3$ variants are benchmarked against corresponding AR$2$ variants, as well as the third-order tensor methods in \cite{birgin2020use}. Encouraging results are obtained that confirm that efficient AR$3$ variants can outperform second-order variants in terms of objective and derivative evaluations and subproblem solves counts.

To summarize, our contributions are as follows.
\begin{itemize}
	\item We introduce a Taylor-based selection strategy which adaptively chooses a suitable initial regularization parameter $\sigma_0$ with low computational cost.  We extend the interpolation-based updating strategy for $\sigma_k$ and $x_{k+1}$ from $p=2$ in \cite{gould2012updating}, to any positive integer $p \geq 3$. This updating strategy allows $\sigma_k$ to change dramatically to a suitable level based on interpolation models.
	\item We identify fundamental differences between the local minima of the subproblems for $p=2$ and $p \geq 3$. Based on this, we introduce a novel pre-rejection module capable of \enquote*{predicting} unsuccessful steps before any function evaluations are carried out, for general $p \geq 3$. We also assess, numerically, the effect of different subproblem termination conditions on the overall algorithm performance.
	\item Numerical studies of the new versus existing techniques, provided in each section, showcase the efficiency of the proposed modifications of the basic AR$3$ algorithm, as illustrated by convergence dot plots and performance profiles. Furthermore, we provide benchmarks of our best performing variants against second- and third-order methods on standard test problems.
    \item We provide an efficient, extensive and modular software package\footnote{\url{https://github.com/karlwelzel/ar3-matlab}} in MATLAB that includes many AR$2$ and AR$3$ variants (including options for Hessian-free and tensor-free operations), allowing easy use and experimentation for interested users.
\end{itemize}

\subsection{\texorpdfstring{The AR$p$ algorithmic framework}{The ARp algorithmic framework}}
Since we address several variants of the AR$p$ algorithm and different updating strategies for the regularization parameter, we describe a generic AR$p$ algorithmic framework in \cref{alg:arp_framework} into which we will fit the techniques we propose and implement.
\begin{algorithm}[ht]
	\KwIn{%
		A starting point $\xx_0 \in \R^d$,
		a tolerance $\eps > 0$, and
		an initial regularization parameter $\sigma_0 > 0$
	}
	\For{$k = 0, 1, \dots$}{
		\If{$k = 0$ or $\xx_k \neq \xx_{k-1}$}{
			Compute $f(\xx_k)$ and $\nabla f(\xx_k)$ \label{lin:arp_compute_derivatives}\;
			\uIf{$\norm{\nabla f(\xx_k)} < \eps$}{
				\Return{approximate local minimizer $\xx_k$}
			}
                \Else{Compute $\nabla^2 f(\xx_k), \ \dots, \ \nabla^p f(\xx_k)$}
		}
        Construct the local model
		$m_k(\ss) = t^p_{\xx_k}(\ss) + \frac{\sigma_k}{p + 1} \norm{\ss}^{p+1}$ and find an approximate local minimizer $\ss_k$ that satisfies $m_k(\ss_k) < m_k(\vek{0})$ and a subproblem termination condition
		\label{lin:arp_solve_subproblem}\;
		Compute $\xx_{k+1}$ and $\sigma_{k+1}$ via one of the updating strategies of \cref{tbl:sigma_updates} \label{lin:arp_sigma_update}
	}
	\caption{The generic AR$p$ framework}\label{alg:arp_framework}
\end{algorithm}

\begin{procedure}[ht]
	\KwParameters{$0 < \eta_1 \leq \eta_2 < 1$, $0 < \gamma_1 < 1 < \gamma_2$, $\sigma_{\min} > 0$}
	\KwDefaults{$\eta_1=0.01$, $\eta_2=0.95$, $\gamma_1=0.5$, $\gamma_2=3$, $\sigma_{\min}=10^{-8}$}
	Compute the predicted function decrease $\deltapred_k = t_k(\vek{0}) - t_k(\ss_k)$\;
	Compute the actual function decrease $\deltaact_k = f(\xx_k) - f(\xx_k + \ss_k)$\;
	Compute the decrease ratio $\rho_k = \deltaact_k / \deltapred_k$\;
	\uIf{$\rho_k \geq \eta_2$}{
		Set $\xx_{k+1} = \xx_k + \ss_k$ and $\sigma_{k+1} = \max\{\gamma_1 \sigma_k, \sigma_{\min}\}$ \tcp*[r]{very successful step}
	}
	\uElseIf{$\rho_k \geq \eta_1$}{
		Set $\xx_{k+1} = \xx_k + \ss_k$ and $\sigma_{k+1} = \sigma_k$ \tcp*[r]{successful step}
	}
	\Else{
		Set $\xx_{k+1} = \xx_k$ and $\sigma_{k+1} = \gamma_2 \sigma_k$ \tcp*[r]{unsuccessful step}
	}
	\caption{Simple update}
	\label[procedure]{alg:simple_sigma_update}
\end{procedure}

Two key parts are unspecified in Algorithm \ref{alg:arp_framework}: the subproblem termination condition in \cref{lin:arp_solve_subproblem} and the updating strategy in \cref{lin:arp_sigma_update}.
In terms of termination condition (TC), we  consider two different options: limiting the norm of the gradient by an \textit{absolute} constant $\eps_{\textrm{sub}} > 0$ as in
\begin{align}
	\label{eqn:subproblem_absolute_tc}
	\norm{\nabla m_k(\ss_k)} \leq \eps_{\textrm{sub}},
	\tag{TC.a}
\end{align}
or bounding the same norm  \textit{relative} to the size of the step, namely,
\begin{align}
	\label{eqn:subproblem_relative_tc}
	\norm{\nabla m_k(\ss_k)} \leq \theta \norm{\ss_k}^p
	\tag{TC.r}
\end{align}
for some (algorithm parameter) $\theta > 0$.
In the latter case, the smaller the step becomes, the higher the precision to which the subproblem is solved. \Cref{eqn:subproblem_relative_tc} provides the required subproblem accuracy so that the  $O(\eps^{-(p+1)/p})$ global iteration/evaluation complexity holds.
We will refer to \cref{eqn:subproblem_absolute_tc} and  \cref{eqn:subproblem_relative_tc} as the absolute and the relative subproblem termination condition, respectively.

The role of the \textit{updating strategy} is to determine whether the current step should be accepted ($\xx_{k+1} = \xx_k + \ss_k$) or rejected ($\xx_{k+1} = \xx_k$) and to update $\sigma_k$ in such a way to ensure progress, in particular by increasing $\sigma_k$ whenever the step is rejected.
A simple and standard updating strategy (with provably optimal worst-case complexity \cite{birgin2017worst}) is given in \cref{alg:simple_sigma_update}.
This approach measures the predicted decrease $\deltapred_k$ using the decrease in the $p$th-order Taylor expansion, and evaluates the success of the step by calculating $\rho_k = \deltaact_k / \deltapred_k$.
Depending on the value of $\rho_k$, the parameter $\sigma_k$ is either decreased by a constant, kept unchanged, or increased by a constant.
We refer to \cref{alg:arp_framework} using \cref{alg:simple_sigma_update} in \cref{lin:arp_sigma_update} as \textsf{AR$p$-Simple}, where $p$, as always, describes the order of the highest derivative.
Throughout the paper we introduce four additional updating strategies that can be used inside \cref{alg:arp_framework} in \cref{lin:arp_sigma_update}.
\Cref{tbl:sigma_updates} gives an overview of all of these variants.

\begin{table}
	\caption{Overview of the named AR$p$ variants in this paper.}\label{tbl:sigma_updates}
	\centering
	\begin{tabular}{ll}
		\toprule
		Name & Updating strategy used in \cref{lin:arp_sigma_update} of \cref{alg:arp_framework} \\
		\midrule
		\textsf{AR$p$-Simple} & \cref{alg:simple_sigma_update} (Simple update \cite{birgin2017worst}) \\
		\textsf{AR$p$-Simple\textsuperscript{+}} & \cref{alg:simple_sigma_update} \& \cref{alg:prerejection} (pre-rejection) \\
		\textsf{AR$p$-Interp} & \cref{alg:interpolation_sigma_update_full} (Interpolation update) \\
		\textsf{AR$p$-Interp\textsuperscript{+}} & \cref{alg:interpolation_sigma_update_full}
		\& \cref{alg:prerejection} \& \cref{rem:interp_plus} \\
		\textsf{AR$p$-BGMS} & \cref{alg:bgms_sigma_update} (BGMS update \cite{birgin2020use}) \\
		\bottomrule
	\end{tabular}
\end{table}

\subsection{Implementation details and preliminaries}
\label{sec:implementation_details}
In this paper, we aim to explore ways to enhance the efficiency of AR$p$ methods for $p \geq 3$, with particular focus on AR$3$. Each existing or new feature that we explore is investigated numerically,  with numerical studies included at the end of each section, and culminating with benchmarking experiments. Since there are running numerical examples and tests throughout the paper, we describe here  the general setup for our experiments.\footnote{The reader may prefer to skip these technical details at present and return to them when encountering the first numerical experiments of the paper.}
To this end, we implemented all AR$p$ variants discussed in this paper for $p=2$ and $p=3$ in a joint MATLAB code base, which can be found at \url{https://github.com/karlwelzel/ar3-matlab}.
Moreover, the data from all experiments relevant for the following graphs are accessible at \url{https://wandb.ai/ar3-project/all_experiments}.
All experiments were conducted on a workstation with an Intel Core i5-13500T CPU (14 cores / 20 threads, up to 4.6\,GHz) running Ubuntu~22.04.

\paragraph{Parameter settings}
Throughout the paper, we set the first-order tolerance for  \cref{alg:arp_framework} to $\eps = 10^{-8}$.
By default, the subproblem solver uses \cref{eqn:subproblem_absolute_tc} with $\eps_{\textrm{sub}} = 10^{-9}$, i.e., it terminates when $\norm{\nabla m_k(\ss)} \leq 10^{-9}$.
Such strict requirement for the subproblem is chosen to make the algorithm's performance less dependent on any specific subproblem solver used.
In our implementation, by default AR$2$ subproblems are solved using \textsf{MCM} \cite[Algorithm 6.1]{cartis2011adaptiveI}, which is a factorization-based solver that computes a global minimizer of the AR$2$ subproblem to high accuracy.
The AR$3$ subproblems in turn are solved using our AR$2$ implementation throughout the experiments in this paper.
In the following, we will refer to the algorithm solving the AR$3$ subproblems as the inner solver and, in the case where AR$2$ is used to solve these problems, the algorithm used inside AR$2$ as the inner-inner or subsubproblem solver.
For our experiments, we use \textsf{AR2-Simple} with default parameters, starting values $\ss_0 = \vek{0}$ and $\sigma_0 = 10^{-8}$ and termination condition \cref{eqn:subproblem_absolute_tc} with $\eps_{\textrm{sub}} = 10^{-9}$ as the inner solver and \textsf{MCM} with termination condition \cref{eqn:subproblem_absolute_tc} with $\eps_{\textrm{sub}} = 10^{-10}$ as the inner-inner solver, unless otherwise specified.
In this high-accuracy regime \textsf{AR$2$-Simple} produces essentially the same AR$3$ trial steps as more sophisticated AR$2$ variants, while involving fewer hyperparameters and providing a simple, standard inner solver across all experiments.
In addition to terminating when $\norm{\nabla f(\xx_k)}$ is sufficiently small, the algorithm and the subproblem solvers also terminate after a maximum of \num{1000} iterations or when numerical issues are detected.
Unless otherwise stated, algorithm parameters for all updating strategies are kept at their default values.
In particular, we use $\eta_1 = 0.01$ and $\eta_2 = 0.95$ to determine successful and very successful iterations following \cite{gould2012updating}.

\paragraph{Hessian- and tensor-free implementation}
Our AR$3$ implementation is flexible in the way the Hessian and the third derivative can be provided to the algorithm.
For Hessians it can accept an explicit matrix $\nabla^2 f(\xx)$ or a way to perform matrix-vector products for arbitrary vectors $\vv \mapsto \nabla^2 f(\xx)\vv$.
For third derivatives, explicit tensors $\nabla^3 f(\xx)$, tensor-vector products $\vv \mapsto \nabla^3 f(\xx)[\vv]$ returning matrices, or tensor-vector-vector products $(\vv, \ww) \mapsto \nabla^3 f(\xx)[\vv, \ww]$ returning vectors are possible.
The key restriction is the inner-inner solver for AR$3$, i.e., the algorithm used as a subproblem solver for AR$2$, which in turn is a subproblem solver for AR$3$.
With \textsf{MCM} \cite[Algorithm 6.1]{cartis2011adaptiveI} as the inner-inner solver, explicit matrices are required in order to compute factorizations, and therefore explicit Hessians and explicit tensors or tensor-vector products are needed.
We have also implemented \textsf{GLRT} \cite[Section 10]{cartis2022evaluation}, an iterative Krylov subspace solver for cubic regularization problems, as the inner-inner solver, allowing for fully matrix-free derivative evaluations.
For most of our experiments we use \textsf{MCM} as the inner-inner solver for two reasons.
Firstly, most of the test problems are relatively small in dimension, so the cost of the Cholesky factorizations required by \textsf{MCM} is acceptable both in terms of efficiency and storage.
Secondly, \textsf{MCM}, as a factorization-based solver, is expected to solve the AR$2$ subproblem to high accuracy, which allows us to verify the effectiveness of the pre-rejection mechanism more precisely in \cref{sec:prerejection}.
For large-scale optimization problems, we recommend using \textsf{GLRT}.
A comparison between \textsf{MCM} and \textsf{GLRT} can be found in \cref{sec:tensor_free}.
A summary of how the Hessians and third derivatives are provided for our implementations of the test problems can be found in \cref{tab:storage_formats} in \cref{sec:functions}.

\paragraph{Problem set}
The problem test set used in this paper includes (all or a subset of) the $35$ unconstrained minimization problems from Moré, Garbow, and Hillstrom's (MGH) test set \cite{more1981testing}, as implemented by Birgin et al.~\cite{birgin2018third}.\footnote{Their code can be found at \url{https://github.com/johngardenghi/mgh} and \url{https://www.ime.usp.br/~egbirgin/sources/bgms2}.}
As most of the problems in the MGH test set are small dimensional ($2 \leq d \leq 40$), we add additional test functions that are larger scale ($d = 100$), including the multidimensional Rosenbrock function \cref{eqn:rosenbrock}, a nonlinear least-squares function \cref{eqn:NLS} and four regularized third-order polynomials \cref{eqn:regularized_cubic}. For the latter,
we construct a well-conditioned regularized third-order polynomial and three ill-conditioned regularized third-order polynomials (with either an ill-conditioned Hessian, an ill-conditioned tensor,\footnote{By \enquote*{ill-conditioned tensor}, we refer to the construction of the tensor via a canonical polyadic (CP) decomposition \cite{hitchcock1927expression} with an ill-conditioned super-diagonal core.} or both an ill-conditioned Hessian and tensor). For more details, see \cref{sec:functions}.
Finally, unless otherwise specified, our problem set is chosen by including half of the MGH test problems (those with odd labels) along with the six constructed problems mentioned above, resulting in a total of $24$ problems.
The starting points for all MGH test problems are set to their default values, while the starting points for the six additional constructed problems are provided in \cref{sec:functions}.
The selection of our representative test set is guided by the following considerations: firstly, since open source libraries contain only a limited number of test functions with (calculated/coded) third-order derivatives, we added additional examples of our own; secondly, to prevent over-fitting, we withhold
half of the MGH test problems until the final benchmark; lastly, we
ensure that the test set includes high(er)-dimensional and ill-conditioned problems.

\paragraph{Evaluation Metrics}
To measure the performance of different variants of the AR$p$ algorithm, we consider three different metrics in this paper: number of objective function evaluations, number of derivative evaluations and number of subproblem solves required to compute a solution.
Note that for the purposes of counting derivative evaluations it does not matter which derivative is considered.
The number of gradient evaluations of the objective function is the same as the number of Hessian evaluations, which again is the same as the number of third derivative evaluations (for AR$3$).
We chose these three metrics to be as much as possible independent of implementation details, and instead focus on the same oracle complexities considered in theory that guarantees worst-case improvements of higher-order algorithms over lower-order ones.
In this paper, our primary aim is to understand which algorithmic strategies (such as the choice of $\sigma_0$, interpolation-based updates, and pre-rejection mechanisms) reduce these evaluation-based measures for AR$3$, rather than to design the fastest possible implementation. For this reason, we treat wall-clock time as secondary and do not optimize or compare it systematically across different inner or inner-inner solvers.

\paragraph{Structure of the paper} The remainder of this paper is organized as follows. In \cref{sec:sigma0}, we first explore the performance of \cref{alg:arp_framework} with various $\sigma_0$ options to highlight the importance of choosing a good starting value of the regularization parameter. Based on these results, we then select an efficient $\sigma_0$ setting for the remainder of our experiments and as default in our code. In \cref{sec:interpolation-sigma-update}, we introduce  an  interpolation-based updating strategy in AR$p$ with $p\geq 3$ and investigate its numerical performance for $p=3$. Then, we present a novel pre-rejection framework in \cref{sec:prerejection} that aims to avoid unnecessary function evaluations, while the impact of subproblem termination criteria is addressed  in \cref{sec:theta}. Finally, in \cref{sec:benchmark}, we provide a comprehensive benchmark comparing our implementations of second- and third-order regularization methods, as well as those in \cite{birgin2020use}, on the full set of $35$ MGH test problems.

\paragraph{Notation and definitions}
\label{sec:notation}
We end this section by declaring the notation and definitions used throughout the paper. Vectors and matrices will always be denoted by bold lower-case and bold upper-case letters, respectively, e.g., $ \vv $ and $ \VV $.
We use regular lower-case and upper-case letters to denote scalar constants, e.g., $d$ or $L$. The transpose of a real vector $ \vv $ is denoted by $ \vv^{\T} $.
The inner product of two vectors $ \vv$ and $\yy $ is denoted by $ \dotprod{\vv, \yy} = \vv^{\T} \yy $. $ \|\vv\| $ and $ \|\VV\| $ denote vector- and matrix-$ \ell_{2} $-norms of a vector $\vv$ and a matrix $ \VV $ respectively.
A superscript $p$ either denotes the power $p$ or a quantity using $p$th-order information, while a subscript $k$ denotes the iteration counter of outer iterations of the AR$p$ algorithm.
The first $p$ derivatives of $f$ at $\xx_k$ are denoted as $\nabla f(\xx_k), \nabla^{2} f(\xx_k), \ldots, \nabla^{p} f(\xx_k)$.
The application of a $p$-tensor $\tns{T}$ as a multilinear map to vectors $\ss_1$, $\ldots$, $\ss_m$, where $m \leq p$, is written as $\tns{T} [\ss_1, \ldots, \ss_m]$, and $\tns{T} [\ss]^m$ represents the special case when $\ss = \ss_1 = \ldots = \ss_m$.
In both cases the result is a $(p - m)$-tensor.
In particular, when \(m = p\) we get a scalar.
The minimum function value of any objective $f$ is denoted by $f^*$.

\section{\texorpdfstring{Choosing the initial regularization parameter $\sigma_0$}{Choosing the initial regularization parameter sigma0}}
\label{sec:sigma0}

The choice of regularization parameter $\sigma$ is crucial for the performance of  AR$p$ algorithms.
It can be viewed, by analogy to trust-region methods, as an approximate inverse of a trust-region radius.
Choosing $\sigma$ too large results in successful but small steps, while choosing too small a value leads to large but unsuccessful steps.
In this section, we use \textsf{AR$3$-Simple} (\cref{alg:arp_framework} combined with \cref{alg:simple_sigma_update}) to demonstrate the impact on algorithm performance of  the choice of the initial regularization parameter $\sigma_0$.
This straightforward experiment allows us to describe in detail the visualizations we use throughout the paper.

We consider $\sigma_0=10^{-8}$, $\sigma_0=1$ and $\sigma_0=10^8$, as well as a simple and cost-effective method for approximating the \enquote*{right} value of $\sigma_0$ using one additional function evaluation.
The function is evaluated at an offset point $\yy$, whose entries are randomly chosen from a standard normal distribution, and $\sigma_0$ is computed as follows:
\begin{equation}\label{eqn:sigma0_taylor}
	\sigma_0 = \max \left\{(p+1) \frac{\abs{f(\xx_0+\yy) - t_{\xx_0}^p(\yy)}}{\norm{\yy}^{p+1}}, \sigma_{\min} \right\},
\end{equation}
where $t_{\xx_0}^p$ is defined as in \cref{eqn:Taylor}.
The idea behind this approach is that $\sigma_0$ is chosen to compensate for the error in the Taylor approximation, which is why this method is referred to as \texttt{Taylor} in the plots below.
It also has a connection to the Lipschitz constant $L_p$ of the $p$th derivative.
We know from \cite[Lemma 2.1]{cartis2020sharp} that $\abs{f(\xx_0+\yy) - t_{\xx_0}^p(\yy)} \leq \frac{L_p}{(p+1)!} \norm{\yy}^{p+1}$, leading to $\sigma_0 = \sigma_{\min}$ or $\sigma_0 \leq \frac{L_p}{p!}$.
While this approximation cannot compute the exact Lipschitz constant, it provides a value that is within the correct order of magnitude in most cases.

\subsection{\texorpdfstring{Numerical studies of the initial regularization parameter $\sigma_0$}{Numerical studies of the initial regularization parameter sigma0}}

We use two different types of graphs to visualize our numerical results: one type focuses on individual test functions to illustrate the effect of specific features introduced in different algorithmic variants, while the second type compares the  overall performance of algorithms across a range of test problems.
We introduce them below, alongside the interpretation of the resulting graphs for the choice of \(\sigma_0\).

\subsubsection{Convergence dot plots}
\label{sec:convergence_dot}

\begin{sidewaysfigure}
	\centering
	\subimport{plots/sigma0}{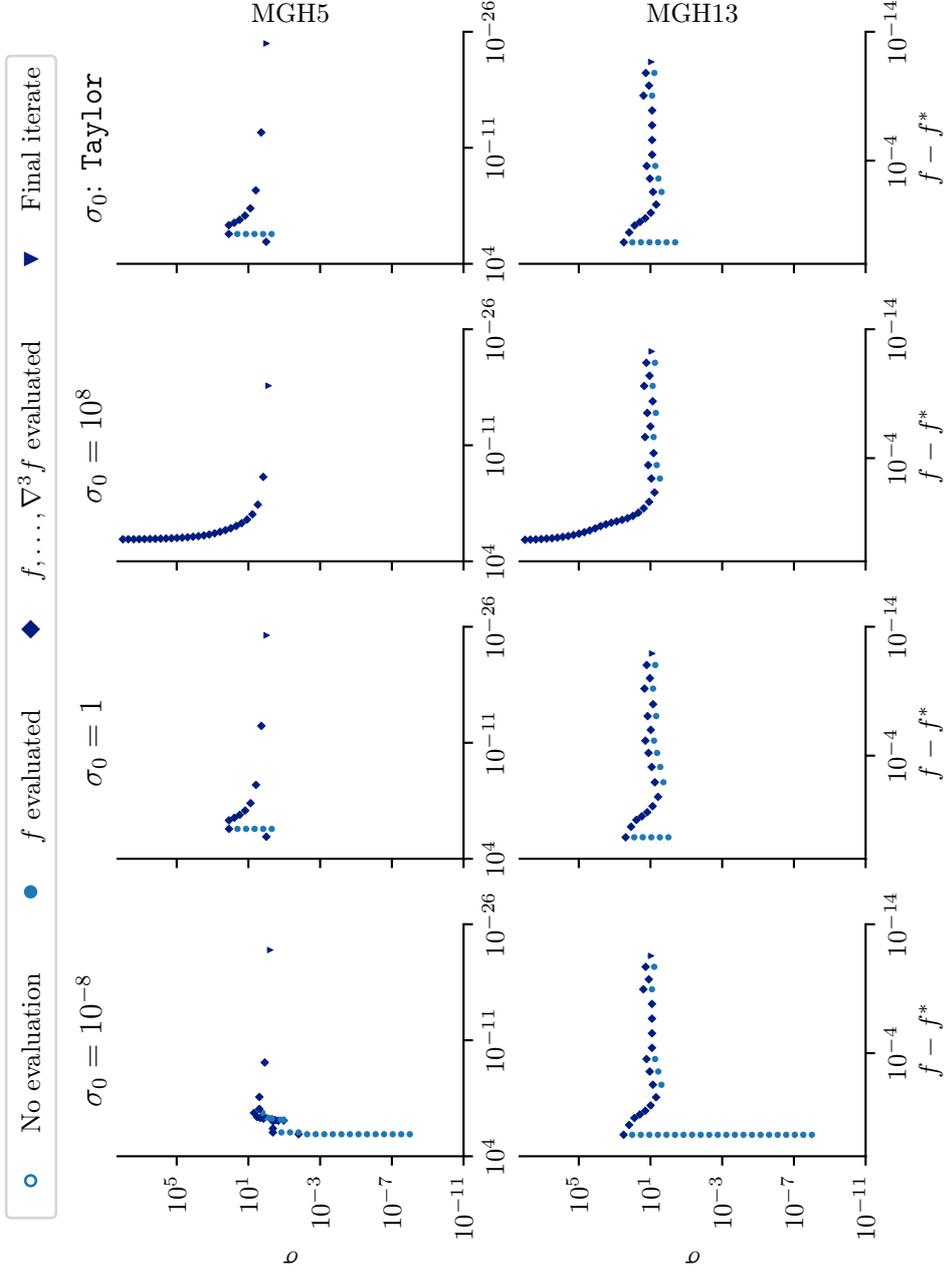}
	\caption{
    The convergence dot plot shows the impact of different $\sigma_0$ values for \textsf{AR$3$-Simple} using \cref{eqn:subproblem_absolute_tc} with $\eps_{\textrm{sub}} = 10^{-9}$.
    An extremely small $\sigma_0$ results in more unsuccessful iterations (blue circles), while an extremely large $\sigma_0$ leads to many successful iterations (dark blue diamonds) with small step sizes, resulting in slow convergence.
    \texttt{Taylor} is able to automatically adjust $\sigma_0$ to a reasonable value, providing the best performance.
    See \cref{sec:convergence_dot} for details on the convergence dot plots.
    }
    \label{fig:sigma0-convergence}
\end{sidewaysfigure}

To illustrate how the progress of AR$3$ towards the function minimizer interacts with the regularization parameter $\sigma_k$, we designed a customized plot, which we refer to as a \enquote*{convergence dot plot}.
In this type of plot (see for example, \cref{fig:sigma0-convergence}), each iteration $k$ is represented by a dot.
The position on the vertical axis is determined by $\sigma_k$ and the horizontal position by $f(\xx_k)-f^*$, where $f^*$ is the minimum function value of the objective function.\footnote{For the two examples used in \cref{fig:sigma0-convergence} and throughout the paper we use the known analytical minimum, namely $f^* = 0$, in both cases.}
Both axes are log-scaled and the x- and y-coordinates of the dots are rounded up to $10^{-25}$ and $10^{-10}$ respectively, if necessary.
The type of dot indicates the oracle calls at the tentative iterate $\xx_k + \ss_k$:
If the step $\ss_k$ is accepted, meaning $\xx_{k+1} = \xx_k + \ss_k$ and both the function and all derivatives are evaluated at this point, the dot is represented by a filled diamond.
If the step $\ss_k$ is rejected, only the function value at $\xx_k + \ss_k$ is computed before the algorithm decides to remain at the current iterate ($\xx_{k+1} = \xx_k$), and the dot takes the form of a filled circle.
Lastly, if the step is \enquote*{pre-rejected}, meaning that neither the function value nor any derivatives are evaluated at $\xx_k + \ss_k$, the dot is shown as an empty circle.
Since we postpone the discussion of pre-rejection mechanisms until \cref{sec:prerejection}, the first few plots will not include any empty circles.
A special case occurs at the very last iterate, where no tentative step $\ss_k$ is computed, as
there are no further evaluations and the value $\sigma_k$ is computed but not used.
We still include this iterate in the plot, because it nicely illustrates the convergence of $f(\xx_k)$ to $f^*$, but we use a triangle as a marker.

Since the x-axis represents the function value gap and is oriented such that smaller function values are on the right, any successful iteration will produce a new dot to the right of the current one.
Similarly, on the y-axis, larger values of $\sigma_k$ are positioned higher up, so unsuccessful iterations will produce a new dot above the current one.
In other words, to follow the sequence of iterations in order, one should read bottom-to-top within each vertical line of dots and left-to-right across the plot.
The total number of function evaluations required to determine $\xx_k$ can be established by counting how many of the dots corresponding to iterations $0$ to $k-1$ are filled (either circle or diamond) and adding one for the function evaluation at $\xx_0$.
Similarly, the number of derivative evaluations is one larger than the number of diamond-shaped dots among iterations $0$ to $k-1$.

\Cref{fig:sigma0-convergence} presents the convergence dot plots for problems $5$ and $13$ from the MGH test set, corresponding to the \enquote*{Beale} function \cite{beale1958iterative} and the \enquote*{Powell singular} function \cite{powell1962iterative}, respectively.
We will keep them as running examples for all numerical illustrations in this paper, as the plots are generally straightforward to interpret and display the typical behaviour of our method(s).
These two functions exhibit different local convergence behaviours, as can be observed in \cref{fig:sigma0-convergence}.
Since the x-axis is log-scaled, equispaced diamond dots indicate that the successful iterates are converging linearly with respect to the function value. As the algorithm progresses, increasing distance between these dots corresponds to superlinear convergence.

For MGH5, there is clear superlinear convergence of the function values to the objective's minimum in each case.
This aligns with the theory predicting that for (locally) strongly convex functions with Lipschitz continuous $p$th derivatives, the local convergence rate of AR$p$ methods is of $p$th-order \cite{doikov2022local,welzel2025local}. Therefore, we should expect third-order convergence asymptotically.
In contrast, the MGH13 function was constructed by Powell such that the Hessian is singular at the minimizer, meaning that second-order methods only converge at a linear rate. For this function, the third derivative at the minimizer is zero, so even with additional derivative information, the local convergence rate remains linear.\footnote{In general, in directions that lie in the null space of $\nabla^2 f(\xx_*)$, the third derivative must also be zero at a local minimizer $\xx_*$, so we cannot expect any acceleration of the local convergence rate when the Hessian is degenerate.}

Returning to the influence of $\sigma_0$ on the algorithm's behaviour, \cref{fig:sigma0-convergence} clearly illustrates the point made earlier: there is an optimal value or \enquote*{sweet spot} for $\sigma_0$ (and $\sigma_k$ in general).
When $\sigma_0$ is excessively large (e.g., $\sigma_0 = 10^8$), all iterations are successful, but the progress is minuscule until $\sigma_k$ is reduced to an appropriate level.
Conversely, when $\sigma_0$ is too small (e.g., $\sigma_0 = 10^{-8}$), the initial iterations are all unsuccessful until, once again, the regularization parameter adjusts to the correct level.
For the two problems shown in \cref{fig:sigma0-convergence}, the \texttt{Taylor} approximation computes a $\sigma_0$ very close to $1$.
While this may not be the perfect value, it is at least within a few orders of magnitude of the right value, significantly reducing the number of iterations and function evaluations compared to the two extreme values of $\sigma_0$.

\subsubsection{Performance profile plots}
\label{sec:performance_profile}

\begin{sidewaysfigure}
	\centering
	\subimport{plots/sigma0}{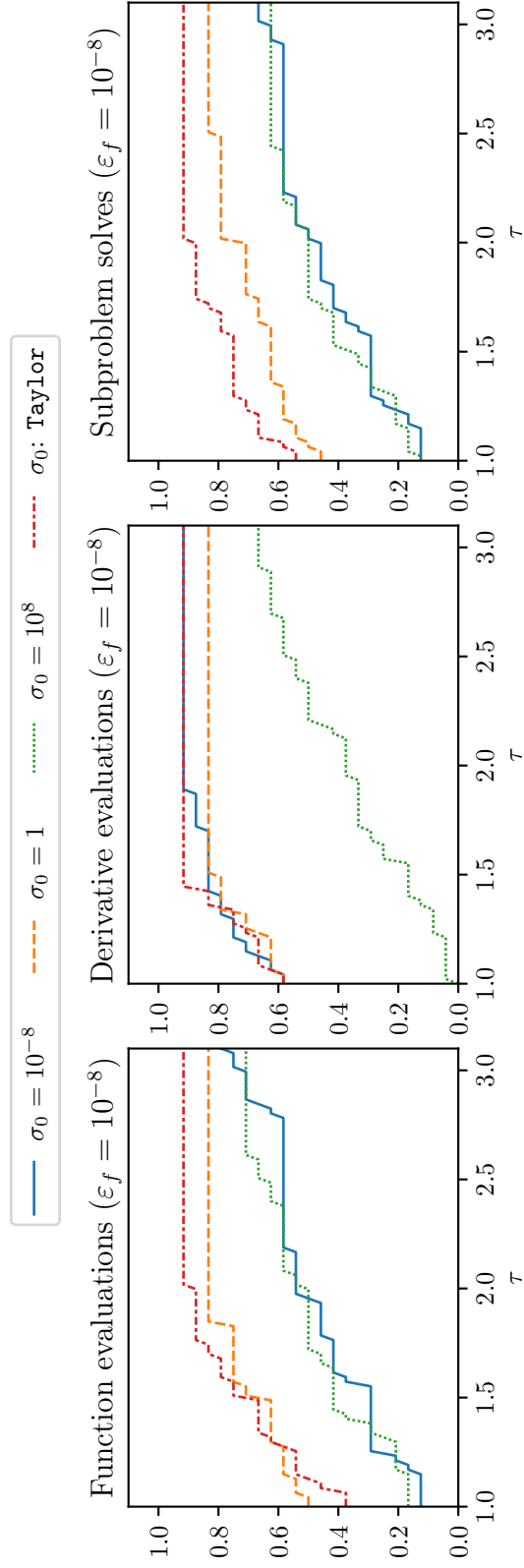}
	\caption{
    The performance profile plot shows the impact of using different $\sigma_0$ values for \textsf{AR$3$-Simple} using \cref{eqn:subproblem_absolute_tc} with $\eps_{\textrm{sub}} = 10^{-9}$.
    Although \texttt{Taylor} requires one extra function evaluation, it provides the most favourable and stable performance across all three measures, particularly in terms of subproblem solves.
    See \cref{sec:performance_profile} for details on the performance profile plots.}
	\label{fig:sigma0-performance-profile}
\end{sidewaysfigure}

We aim to determine whether the conclusions drawn from our two test problems also apply to the entire test set.
Performance profile plots are a standard tool for comparing the performance of different algorithms with respect to a given cost \cite{dolan2002benchmarking,gould2016note}.
In this work, we follow the approach and notation from \cite{birgin2020use}, considering the value of $\Gamma_i(\tau, \eps_f)$ for each method $M_i$.
This $\Gamma_i(\tau, \eps_f)$ represents \enquote*{the proportion of problems for which method $M_i$ was able to find an $\eps_f$-approximation to a solution with a cost up to $\tau$ times the cost required by the most efficient method.}
In this context, an $\eps_f$-approximate solution refers to an iterate $\xx_k$ during the algorithm's run for which
\begin{equation}
	\frac{f(\xx_k) - f_{\textrm{best}}}{\max\{1, \abs{f_{\textrm{best}}}\}} < \eps_f.
\end{equation}
Since the minimum function value is often unknown, we use the smallest function value $f_{\textrm{best}}$ found by any of the methods as a practical proxy.

As described in \cref{sec:implementation_details}, our test set consists of $24$ problems, including $18$ odd-numbered MGH problems and six additional large-dimensional problems.
In the main paper, we report performance profile plots using the number of function evaluations, derivative evaluations and subproblem solves as the cost metrics, with $\eps_f = 10^{-8}$ as the required function value accuracy.
Additional performance profiles for different values of $\eps_f$ and $\tau$ can be found in \cref{sec:additional_experiments}.

Analysing \cref{fig:sigma0-performance-profile}, we find that extreme values of $\sigma_0$ require significantly more function evaluations and subproblem solves compared to a more reasonable choice, such as $\sigma_0 = 1$.
The very large $\sigma_0$ also results in more frequent derivative evaluations due to many successful iterations with minimal progress.
Since the \texttt{Taylor} method of estimating $\sigma_0$ is competitive across all metrics and adapts to the scaling of the problem, we use it for the remainder of the paper.

\section{\texorpdfstring{An interpolation-based update for AR$p$}{An interpolation-based update for ARp}}
\label{sec:interpolation-sigma-update}

After establishing that the correct choice of the regularization parameter is essential for an efficient AR$p$ method, we consider the update mechanism for $\sigma_k$ next.
Gould, Porcelli and Toint~\cite{gould2012updating} introduced a method for updating $\sigma_k$ based on a one-dimensional interpolation of the objective function.
While their work  addresses the case  $p=2$, we are motivated by this approach to introduce an interpolation-based update for an arbitrary $p \geq 2$.
The primary advantage of the interpolation-based strategy is its ability to adapt more quickly to the optimal level of $\sigma$ in many situations, using only the information that is already available and avoiding costly computations.

\subsection{Derivation}

The general structure presented in \cref{alg:interpolation_sigma_update} is similar to the simple update in \cref{alg:simple_sigma_update}: a decrease ratio $\rho_k$ is computed, and $\sigma_k$ is adjusted by a constant factor as long as $\rho_k \in (0, 1)$.
More sophisticated rules are applied in \cref{alg:interpolation_sigma_update} only when the step is either extremely successful ($\rho_k \geq 1$) or extremely unsuccessful ($\rho_k < 0$).
Furthermore, the predicted function decrease is now calculated on the basis of the decrease in $m_k$, which is strictly smaller than the decrease in $t_k$ used in \cref{alg:simple_sigma_update}. Therefore, the particular criterion used in \cref{alg:interpolation_sigma_update}  increases the likelihood of a step being considered extremely successful, and this has a positive impact on  practical performance.

\begin{procedure}
	\KwParameters{$0 < \eta_1 \leq \eta_2 < 1$, $0 < \gamma_1 < 1 < \gamma_2$, $\sigma_{\min} > 0$}
	Compute the predicted function decrease $\deltapred_k = m_k(\vek{0}) - m_k(\ss_k)$\;
	Compute the actual function decrease $\deltaact_k = f(\xx_k) - f(\xx_k + \ss_k)$\;
	Compute the decrease ratio $\rho_k = \deltaact_k / \deltapred_k$\;
	\uIf(\tcp*[f]{extremely successful step}){$\rho_k \geq 1$}{
		Set $\xx_{k+1} = \xx_k + \ss_k$ and determine $\sigma_{k+1}$ by interpolation
	}
	\uElseIf(\tcp*[f]{very successful step}){$\rho_k \geq \eta_2$}{
		Set $\xx_{k+1} = \xx_k + \ss_k$ and $\sigma_{k+1} = \max\{\gamma_1 \sigma_k, \sigma_{\min}\}$
	}
	\uElseIf(\tcp*[f]{successful step}){$\rho_k \geq \eta_1$}{
		Set $\xx_{k+1} = \xx_k + \ss_k$ and $\sigma_{k+1} = \sigma_k$
	}
	\uElseIf(\tcp*[f]{unsuccessful step}){$\rho_k \geq 0$}{
		Set $\xx_{k+1} = \xx_k$ and $\sigma_{k+1} = \gamma_2 \sigma_k$
	}
	\Else(\tcp*[f]{extremely unsuccessful step}){
		Set $\xx_{k+1} = \xx_k$ and determine $\sigma_{k+1}$ by interpolation
	}
	\caption{Interpolation update (general structure)}
	\label[procedure]{alg:interpolation_sigma_update}
\end{procedure}

From a theoretical perspective, whether the predicted function decrease is calculated using $m_k$ or $t_k$ decrease does not impact the
global complexity guarantees that can be provided for AR2\footnote{Note that the same complexity guarantees are obtained under slightly different assumptions on the subproblem solution. When measuring the predicted decrease using the regularized model, one needs to assume that $\ss_k$ has a smaller model value than the Cauchy point and that $\norm{\nabla^2 f(\xx_k)}$ stays bounded. When using the Taylor approximation, it suffices to assume that $\norm{\nabla m_k(\ss_k)} \leq \theta \norm{\ss_k}^2$ for some $\theta \geq 0$.} \cite{birgin2017worst,cartis2020concise,cartis2011adaptiveII}. However, for AR$p$ with $p \geq 3$, though  both  measures of decrease  still lead to convergent variants,  only the simple update \cref{alg:simple_sigma_update} has been shown to satisfy the optimal global complexity rate \cite{cartis2022evaluation}. The use of \cref{alg:interpolation_sigma_update} is therefore driven by practical performance considerations rather than theoretical complexity guarantees.

\paragraph{Simplifying design principles} The interpolation-based updates in the extremely successful and extremely unsuccessful cases are derived using the following simplifying principles, that guide the design of our technique.
\begin{enumerate}[(S1)]
	\item \emph{Consistent regularization}: An alternative value of the regularization parameter $\tilde{\sigma}_k$, that would have been a good choice for the current iteration, is also a good choice for the next iteration. \label{itm:simplification-same-sigma}
	\item \emph{Ray-based minimization}: Any minimizer $\tilde{\ss}_k$ for this counterfactual choice of $\tilde{\sigma}_k$ satisfies $\tilde{\ss}_k = \alpha \ss_k$ for some $\alpha > 0$. \label{itm:simplification-ray}
	\item \emph{Interpolation along the ray}: Along the ray $\{\alpha \ss_k \mid \alpha \geq 0\}$, the objective function is well-approximated by the interpolating polynomial $p_f$ defined below. \label{itm:simplification-interpolating-polynomial}
\end{enumerate}
\ref{itm:simplification-same-sigma} allows us to use the information already calculated about the regularized Taylor expansion and its minimizer at the current iteration for the next iteration, even though, in theory, there is no guarantee that the optimal values of $\sigma_k$ and $\sigma_{k+1}$ will be similar if the current iteration is successful.
\ref{itm:simplification-ray} implies that we only need to consider the relevant functions on the ray extending from the origin through $\ss_k$.
Let
\begin{equation}\label{eqn:t_m}
	t_{\rightarrow}(\alpha) = t_{\xx_k}^p (\alpha \ss_k / \norm{\ss_k}) \quad \text{and} \quad m_{\rightarrow, \sigma}(\alpha) = m_{\xx_k, \sigma}^p (\alpha \ss_k / \norm{\ss_k}) \quad \text{for } \alpha \in [0, \infty)
\end{equation}
be the current Taylor expansion and model along the ray pointing in the direction of $\ss_k$.
We note that $t_{\rightarrow}(\alpha)$ and $m_{\sigma,\rightarrow}(\alpha)$ are one-dimensional polynomials of order $p$ and $p+1$, respectively.
The interpolating polynomial $p_f(\alpha) \approx f(\xx_k + \alpha \ss_k / \norm{\ss_k})$ is defined as the unique polynomial of degree $p+1$ that satisfies
\begin{equation}
	p_f(0) = f(\xx_k), \quad p_f^{(j)}(0) = \nabla^j f(\xx_k) \big[\ss_k / \norm{\ss_k}\big]^j \quad \text{and} \quad p_f(\norm{\ss_k}) = f(\xx_k + \ss_k),
\end{equation}
incorporating all the information gathered about the objective function along the ray. It is straightforward to show that $p_f(\alpha)$ can be explicitly written as
\begin{equation}
	p_f(\alpha) = f(\xx_k) + \sum_{j=1}^p \frac{1}{j!} \nabla^j f(\xx_k) \big[\ss_k / \norm{\ss_k}\big]^j \alpha^j + \brack*{\frac{f(\xx_k + \ss_k) - t_k(\ss_k)}{\norm{\ss_k}^{p+1}}} \alpha^{p+1},
\end{equation}
with coefficients that are easy to compute.
\ref{itm:simplification-interpolating-polynomial} enables us to use $p_f$ instead of $f$ and avoid additional evaluations of the objective function.

\paragraph{Design of the regularization parameter updates}
In the \emph{extremely unsuccessful} case, since the step was unsuccessful, a different regularization parameter $\tilde{\sigma}_k$ should have been chosen such that the corresponding step $\tilde{\ss}_k$ would have been successful, satisfying
\begin{equation}
	\tilde{\rho}_k = \frac{f(\xx_k) - f(\xx_k + \tilde{\ss}_k)}{m_{\xx_k, \tilde{\sigma}_k}^p(\vek{0}) - m_{\xx_k, \tilde{\sigma}_k}^p (\tilde{\ss}_k)} \geq \eta_1.
\end{equation}
Using the simplifications described earlier, we therefore search for the smallest $\sigma > \sigma_k$ such that a minimizer $\alpha$ of $m_{\rightarrow, \sigma}$ satisfies
\begin{equation}\label{eqn:interpolation_unsuccessful_constraint}
	\frac{p_f(0) - p_f(\alpha)}{m_{\rightarrow, \sigma}(0) - m_{\rightarrow, \sigma}(\alpha)} \geq \eta_1.
\end{equation}

The \emph{extremely successful} case further splits into two scenarios.
Firstly, when $m_k (\ss_k) \geq f(\xx_k + \ss_k) \geq t_k (\ss_k)$, we aim to find the largest $\sigma < \sigma_k$ such that the difference between the regularized model and the objective function is reduced by a factor of $\beta \in (0, 1)$ at the minimizer $\alpha$.
This can be expressed as
\begin{equation}\label{eqn:interpolation_successful_constraint_fgeqt}
	m_{\rightarrow, \sigma}(\alpha) - p_f(\alpha) \leq \beta (m_{\rightarrow, \sigma_k}(\norm{\ss_k}) - p_f(\norm{\ss_k})) = \beta (m_k(\ss_k) - f(\xx_k + \ss_k)).
\end{equation}
In the second scenario, when $f(\xx_k + \ss_k) < t_k (\ss_k)$, the highest-order coefficient of $p_f$ is negative, meaning the polynomial is not lower bounded on $\alpha \in (0, \infty)$ and therefore not useful.
Then, instead of using $p_f$, we resort to the Taylor expansion and simply aim to decrease the difference between the regularized model and the Taylor expansion by the same factor $\beta$:
\begin{equation}\label{eqn:interpolation_successful_constraint_fltt}
	m_{\rightarrow, \sigma}(\alpha) - t_{\rightarrow}(\alpha) \leq \beta(m_{\rightarrow, \sigma_k}(\norm{\ss_k}) - t_{\rightarrow}(\norm{\ss_k})) = \beta (m_k(\ss_k) - t_k(\ss_k)).
\end{equation}
\paragraph{Calculating the updated regularization parameter} In each of the above scenarios, we establish a constraint for the new $\sigma$ and its corresponding minimizer $\alpha$ given by \cref{eqn:interpolation_unsuccessful_constraint}, \cref{eqn:interpolation_successful_constraint_fgeqt} or \cref{eqn:interpolation_successful_constraint_fltt}.
We will abbreviate this constraint as $c(\alpha, \sigma) \leq 0$.
By additionally replacing the constraint that $\alpha$ is a minimizer of $m_{\rightarrow, \sigma}$ with second-order necessary conditions\footnote{In \cite{gould2012updating}, it sufficed to demand stationarity of $\alpha$, since there would only ever be exactly one local minimizer $\alpha > 0$.} we arrive at the following optimization problems:
\begin{equation}\label{eqn:interpolation_optimization_formulation}
	\begin{aligned}
		\min_{\alpha, \sigma} &  & \sigma                                             \\
		\text{s.t.}           &  & m_{\rightarrow, \sigma}' (\alpha) & = 0           \\
		&  & m_{\rightarrow, \sigma}''(\alpha) & \geq 0        \\
		&  & c(\alpha, \sigma)    & \leq 0        \\
		&  & \sigma               & \geq 0        \\
		&  & \sigma               & \geq \sigma_k \\
		&  & \alpha               & > 0
	\end{aligned}
	\quad \quad \text{or} \quad \quad
	\begin{aligned}
		\max_{\alpha, \sigma} &  & \sigma                               \\
		\text{s.t.}           &  & m_{\rightarrow, \sigma}' (\alpha) & = 0           \\
		&  & m_{\rightarrow, \sigma}''(\alpha) & \geq 0        \\
		&  & c(\alpha, \sigma)    & \leq 0        \\
		&  & \sigma               & \geq 0        \\
		&  & \sigma               & \leq \sigma_k \\
		&  & \alpha               & > 0
	\end{aligned}
\end{equation}
In the unsuccessful case, we need to solve the problem on the left, and in the successful case, the one on the right;
we analyse them simultaneously.
The stationarity condition can be rearranged to express $\sigma$ as a function of $\alpha$:
\begin{equation}\label{eqn:sigma_formula_alpha}
	m_{\rightarrow, \sigma}'(\alpha) = t_{\rightarrow}'(\alpha) + \sigma \alpha^p = 0 \qquad \Leftrightarrow \qquad\sigma = \frac{-t_{\rightarrow}'(\alpha)}{\alpha^p}.
\end{equation}
Substituting $\sigma$ with the expression in \eqref{eqn:sigma_formula_alpha} (and multiplying the inequalities by $\alpha > 0$ as needed) leads to the following equivalent system, that only involves  polynomial inequalities in $\alpha$:
\begin{equation}\label{eqn:interpolation_optimization_formulation_alpha}
	\begin{aligned}
		\min_{\alpha} &  & -\frac{t_{\rightarrow}'(\alpha)}{\alpha^p}            \\
		\text{s.t.}   &  & t_{\rightarrow}''(\alpha)\alpha - p t_{\rightarrow}'(\alpha)                  & \geq 0 \\
		&  & c(\alpha)                                         & \leq 0 \\
		&  & t_{\rightarrow}'(\alpha)                                        & \leq 0 \\
		&  & t_{\rightarrow}'(\alpha) + \sigma_k \alpha^p & \leq 0 \\
		&  & \alpha                                            & > 0
	\end{aligned}
	\quad \text{or} \quad
	\begin{aligned}
		\max_{\alpha} &  & -\frac{t_{\rightarrow}'(\alpha)}{\alpha^p}            \\
		\text{s.t.}   &  & t_{\rightarrow}''(\alpha)\alpha - p t_{\rightarrow}'(\alpha)                  & \geq 0 \\
		&  & c(\alpha)                                         & \leq 0 \\
		&  & t_{\rightarrow}'(\alpha)                                        & \leq 0 \\
		&  & t_{\rightarrow}'(\alpha) + \sigma_k \alpha^p & \geq 0 \\
		&  & \alpha                                            & > 0.
	\end{aligned}
\end{equation}
The inequality $c(\alpha) = c(\alpha, -t_{\rightarrow}'(\alpha) / \alpha^p) \leq 0$ can also be expressed as a polynomial inequality since $m_{\rightarrow, \sigma}(\alpha) = t_{\rightarrow}(\alpha) - \frac{1}{p+1} t_{\rightarrow}'(\alpha)\alpha$. Therefore, we can rewrite \cref{eqn:interpolation_unsuccessful_constraint,eqn:interpolation_successful_constraint_fgeqt,eqn:interpolation_successful_constraint_fltt} as follows:
\begin{align}
	\eta_1 \paren*{t_{\rightarrow}(0) - t_{\rightarrow}(\alpha) + \frac{1}{p+1} t_{\rightarrow}'(\alpha)\alpha} - (p_f(0) - p_f(\alpha))      & \refleq{eqn:interpolation_unsuccessful_constraint} 0 \label{eqn:interpolation_unsuccessful_constraint_alpha}         \\
	t_{\rightarrow}(\alpha) - \frac{1}{p+1} t_{\rightarrow}'(\alpha)\alpha - p_f(\alpha) - \beta(m_k(\ss_k) - f(\xx_k + \ss_k)) & \refleq{eqn:interpolation_successful_constraint_fgeqt} 0 \label{eqn:interpolation_successful_constraint_fgeqt_alpha} \\
	\frac{1}{p+1} t_{\rightarrow}'(\alpha)\alpha - \beta (m_k(\ss_k) - t_k(\ss_k))                                & \refleq{eqn:interpolation_successful_constraint_fltt} 0 \label{eqn:interpolation_successful_constraint_fltt_alpha}
\end{align}

In all cases, the feasible set of $\alpha$ is described by a series of polynomial inequalities of order at most $p$.
Since the derivative of the objective in \cref{eqn:interpolation_optimization_formulation_alpha} is given by
\begin{equation}
	\frac{-t_{\rightarrow}''(\alpha) \alpha + p t_{\rightarrow}'(\alpha)}{\alpha^{p + 1}},
\end{equation}
the objective is monotonic over the feasible set, meaning that the optimal value will be attained at the boundary.
It is therefore sufficient to compute the roots of all polynomial left-hand sides in \cref{eqn:interpolation_optimization_formulation_alpha}, determine which of them satisfy all other inequalities, and then choose the one with the smallest or largest value of $\sigma = -t_{\rightarrow}'(\alpha) / \alpha^p$.
It is important to note that the strict inequality $\alpha >0$ is not troublesome, as for $t_{\rightarrow}'(\alpha) \neq 0$, the value of $\sigma$ converges to $\pm \infty$ as $\alpha$ approaches zero, making it either non-optimal or in violation of the bound on $\sigma$.

Since the parameter $\alpha^*$ and corresponding $\sigma^*$ computed from solving \cref{eqn:interpolation_optimization_formulation_alpha} are based on simplifications, it is important to safeguard their outcomes.
Moreover, there may be instances where no feasible solutions to \cref{eqn:interpolation_optimization_formulation_alpha} exist, meaning that $\alpha^*$ and $\sigma^*$ cannot be determined. In such cases, a fallback procedure is necessary.
We adopt the approach in \cite{gould2012updating}, which involves enforcing a minimum and maximum increase in the extremely unsuccessful case and falling back to a simple decrease by a constant factor if $\alpha^* > \alpha_{\max} \norm{\ss_k}$ in the successful case.
The complete procedure is shown in \cref{alg:interpolation_sigma_update_full}.

\begin{procedure}
	\KwParameters{
		$0 < \eta_1 \leq \eta_2 < 1$,
		$0 < \gamma_{\min} < \gamma_1 < 1 < \gamma_2 < \gamma_{\max}$,
		$0 < \beta < 1$,
		$\alpha_{\max} > 1$,
		$\chi_{\min} > 0$,
		$\sigma_{\min} > 0$
	}
	\KwDefaults{
		$\eta_1=0.01$, $\eta_2=0.95$,
		$\gamma_{\min}=10^{-1}$, $\gamma_1=0.5$, $\gamma_2=3$, $\gamma_{\max}=10^{2}$,
		$\beta=10^{-2}$,
		$\alpha_{\max}=2$,
		$\chi_{\min}=10^{-8}$,
		$\sigma_{\min}=10^{-8}$
	}
	Compute the predicted function decrease $\deltapred_k = m_k(\vek{0}) - m_k(\ss_k)$\;
	Compute the actual function decrease $\deltaact_k = f(\xx_k) - f(\xx_k + \ss_k)$\;
	Compute the decrease ratio $\rho_k = \deltaact_k / \deltapred_k$\;
	\uIf(\tcp*[f]{extremely successful step}){$\rho_k \geq 1$}{
		Set $\xx_{k+1} = \xx_k + \ss_k$ \;
		Compute $\chi_k = m_k(\ss_k) - \max\{f(\xx_k + \ss_k), t_k(\ss_k)\}$ \;
		\uIf{$\chi_k \geq \chi_{\min}$}{
			\uIf{$f(\xx_k + \ss_k) \geq t_k(\ss_k)$}{
				Compute $(\alpha^*, \sigma^*)$ from \cref{eqn:interpolation_optimization_formulation_alpha}-max using \cref{eqn:interpolation_successful_constraint_fgeqt_alpha} as the constraint $c(\alpha)\leq 0$\;
			}
			\Else{
				Compute $(\alpha^*, \sigma^*)$ from \cref{eqn:interpolation_optimization_formulation_alpha}-max using \cref{eqn:interpolation_successful_constraint_fltt_alpha} as the constraint $c(\alpha)\leq 0$\;
			}
			Set $\sigma_{k+1} = \begin{cases}
				\max \{\sigma^*, \sigma_{\min}\}               & \text{if } (\alpha^*, \sigma^*) \text{ exists and } \alpha^* \leq \alpha_{\max} \norm{\ss_k} \\
				\max \{\gamma_{\min} \sigma_k, \sigma_{\min}\} & \text{otherwise}
			\end{cases}$\;
		}
		\Else{
			Set $\sigma_{k+1} = \max\{\gamma_1 \sigma_k, \sigma_{\min}\}$\;
		}
	}
	\uElseIf(\tcp*[f]{very successful step}){$\rho_k \geq \eta_2$}{
		Set $\xx_{k+1} = \xx_k + \ss_k$ and $\sigma_{k+1} = \max\{\gamma_1 \sigma_k, \sigma_{\min}\}$
	}
	\uElseIf(\tcp*[f]{successful step}){$\rho_k \geq \eta_1$}{
		Set $\xx_{k+1} = \xx_k + \ss_k$ and $\sigma_{k+1} = \sigma_k$
	}
	\uElseIf(\tcp*[f]{unsuccessful step}){$\rho_k \geq 0$}{
		Set $\xx_{k+1} = \xx_k$ and $\sigma_{k+1} = \gamma_2 \sigma_k$
	}
	\Else(\tcp*[f]{extremely unsuccessful step}){
		Set $\xx_{k+1} = \xx_k$\;
		Compute $(\alpha^*, \sigma^*)$ from \cref{eqn:interpolation_optimization_formulation_alpha}-min using \cref{eqn:interpolation_unsuccessful_constraint_alpha} as the constraint $c(\alpha)\leq 0$\;
		Set $\sigma_{k+1} = \begin{cases}
			\min\{\max\{\sigma^*, \gamma_2 \sigma_k\}, \gamma_{\max} \sigma_k\} & \text{if } (\alpha^*, \sigma^*) \text{ exists} \\
			\gamma_2 \sigma_k                                                   & \text{otherwise}
		\end{cases}$\;
	}
	\caption{Interpolation update}
	\label[procedure]{alg:interpolation_sigma_update_full}
\end{procedure}

\subsection{Numerical studies of the interpolation-based update}

For the new parameters in \cref{alg:interpolation_sigma_update_full}, we use the default values suggested in \cite{gould2012updating}, except for $\gamma_1$ and $\gamma_2$.
The increase and decrease of $\sigma_k$ when $\rho_k \in (0, 1)$ are kept consistent between the simple and the interpolation-based update to ensure a fair comparison.

\begin{figure}
	\centering
	\subimport{plots/interpolation}{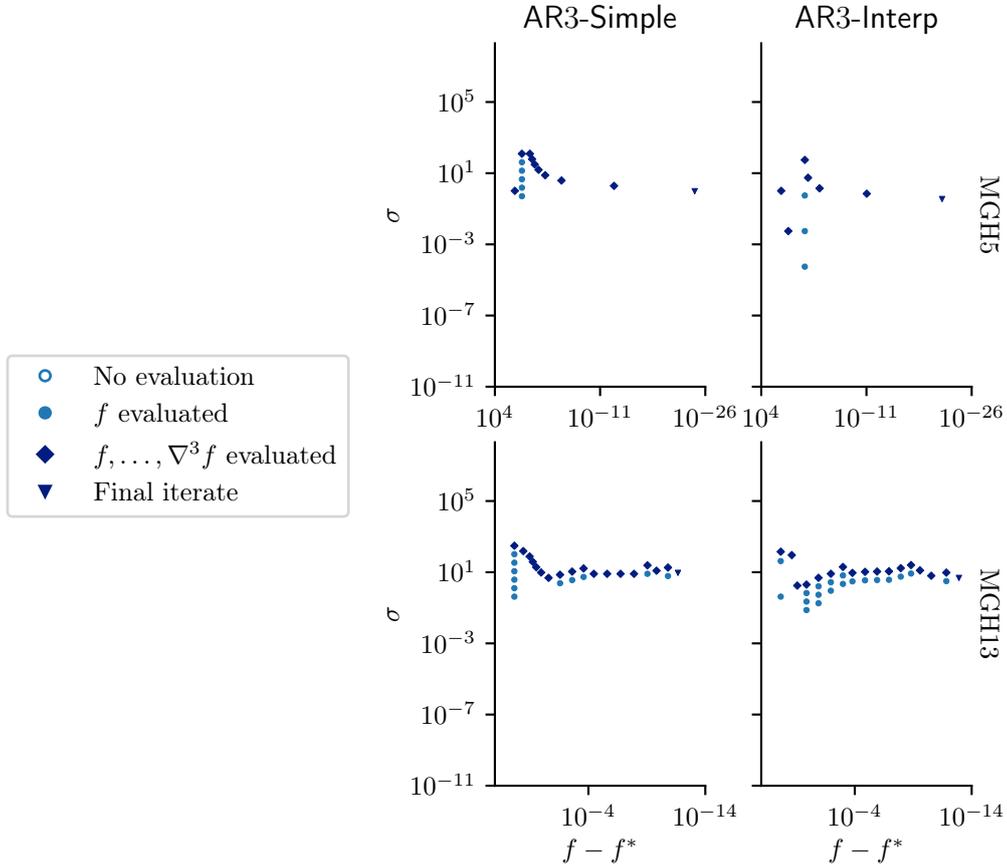}
	\caption{
    The convergence dot plot illustrates  the interpolation-based strategy.In the experiments the subproblem is solved to a high accuracy using \cref{eqn:subproblem_absolute_tc} with $\eps_{\textrm{sub}} = 10^{-9}$ and $\sigma_0$ is chosen according to \cref{eqn:sigma0_taylor}.
    As expected, \textsf{AR$3$-Interp} changes the regularization parameter more aggressively in many cases.
    Compared to \textsf{AR$3$-Simple}, this strategy enhances performance for MGH5 but causes additional unsuccessful steps for MGH13 by decreasing $\sigma_k$ too rapidly.
    See \cref{sec:convergence_dot} for details on the convergence dot plots.
	}
	\label{fig:interpolation-convergence}
\end{figure}

\begin{sidewaysfigure}
	\centering
	\subimport{plots/interpolation}{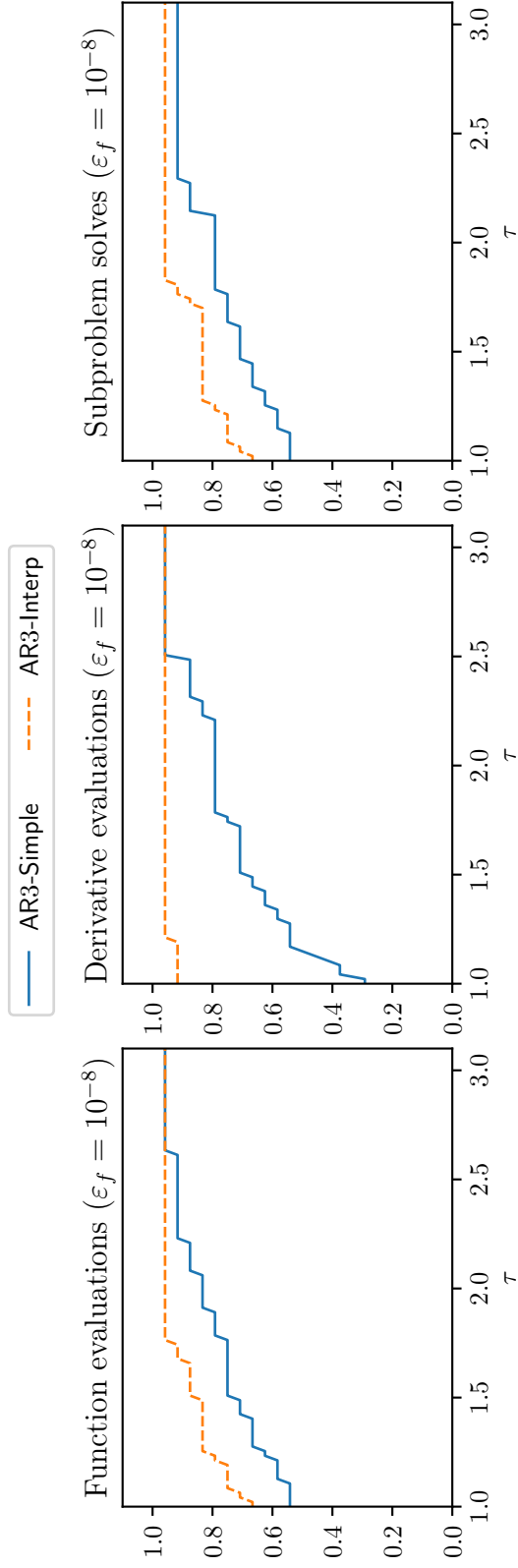}
	\caption{
    The performance profile illustrates  the interpolation-based strategy.
    In the experiments, the subproblem is solved to  high accuracy using \cref{eqn:subproblem_absolute_tc} with $\eps_{\textrm{sub}} = 10^{-9}$ and $\sigma_0$ is chosen according to \cref{eqn:sigma0_taylor}.
    On our selection of test problems, \textsf{AR$3$-Interp} performs significantly better than \textsf{AR$3$-Simple}, especially on the number of derivative evaluations.
    See \cref{sec:performance_profile} for details on the performance profile plots.
	}
	\label{fig:interpolation-performance-profile}
\end{sidewaysfigure}

\Cref{fig:interpolation-convergence} presents the convergence dot plots for \textsf{AR$3$-Simple}, and the interpolation-based update strategy  labelled \textsf{AR$3$-Interp} (\cref{alg:arp_framework} combined with \cref{alg:interpolation_sigma_update_full}).
For MGH5,  \textsf{AR$3$-Interp} significantly reduces the overall number of iterations as well as the number of function and derivative evaluations by adaptively adjusting $\sigma_k$ much more aggressively than \textsf{AR$3$-Simple}.
However, MGH13 illustrates that there are cases where the interpolation-based update is less efficient than the simple one because it decreases $\sigma_k$ too quickly.
In these cases, while the number of derivative evaluations decreases, the number of function evaluations increases due to the smaller values of $\sigma_k$ leading to more unsuccessful iterations.

To assess the overall impact of switching from \textsf{AR$3$-Simple} to \textsf{AR$3$-Interp}, we examine the performance profile plots in \cref{fig:interpolation-performance-profile}.
These curves suggest that, on average, the new updating strategy reduces the number of function evaluations, derivative evaluation and subproblem solves.
The most striking improvement is seen in the number of derivative evaluations, or equivalently the number of successful iterations, where \textsf{AR$3$-Interp} outperforms \textsf{AR$3$-Simple} on almost all test problems.

\section{A pre-rejection framework}
\label{sec:prerejection}

The performance of \cref{alg:arp_framework} can be further enhanced by avoiding function evaluations when a step is unlikely to decrease the objective function.
In this section, we introduce a novel pre-rejection mechanism that rejects steps based on the Taylor expansion along the step.
This mechanism is motivated by the fundamentally different behaviour of minimizers in the AR$3$ subproblem when $\sigma$ changes, compared to minimizers in the AR$2$ subproblem. The following subsections will discuss this difference, define various types of local minimizers, provide conditions to characterize these minimizers in the one-dimensional case, and propose a new pre-rejection heuristic based on this analysis.

\subsection{\texorpdfstring{Discontinuity in the AR$p$ subproblem solutions}{Discontinuity in the ARp subproblem solutions}}
\label{sec:sigma-discontinuity}

As $p$ increases, the AR$p$ subproblem of minimizing the regularized model in \cref{eqn:subproblem} becomes increasingly challenging to solve.
For $p=1$ we are searching for the minimum of a quadratically regularized linear model.
The solution in this case is always a scaled negative gradient, and the scaling factor can be easily computed from the regularization parameter.

For $p=2$, the local model can become non-convex and may have multiple local minima. However, it is still possible to solve the optimization problem due to the characterization of its global minimum \cite[Theorem 8.2.8]{cartis2022evaluation}.
Moreover, this global minimizer is guaranteed to be a descent direction, that is, it satisfies $\nabla f(\xx_k)^\T \ss_k < 0 $. Given this, solvers for the AR$2$ subproblem typically aim to find the global minimizer \cite{cartis2022evaluation}.

For $p \geq 3$,
such nice properties that hold for lower-order methods no longer apply. The AR$p$ subproblem becomes significantly more challenging to solve, there is no complete characterization for any of its minimizers, and the global minimizer does not even have to be a descent direction (see \cref{thm:ar3_min_not_descent}).

\begin{example}[The global AR3 model minimizer is not a descent direction]\label{thm:ar3_min_not_descent}
	Consider the one-dimensional Taylor expansion $t(s) = s^3 + s^2 - s$.
	Clearly, any positive $s$ is a descent direction.
	At the same time, $\min_{s \geq 0} t(s) = t(\frac{1}{3}) = -\frac{5}{27}$ and $\lim_{s \to -\infty} t(s) = -\infty$.
	Therefore, for sufficiently small $\sigma$, the global minimizer of the model is negative and not a descent direction.
\end{example}

Notably, both optimal complexity guarantees \cite{cartis2022evaluation} and local convergence results \cite{doikov2022local} do not impose restrictions on which minimizer should be chosen.
This is because when $\sigma$ is sufficiently large, any local or global minimizer of the model will provide sufficient decrease in the objective, and the updates will be accepted.
From an efficiency standpoint, it is preferable to maintain $\sigma$ small, to allow for larger steps.
However, when the third derivative of $f$ at $\xx_k$ is nonzero, the third-order Taylor expansion is always nonconvex and does not have a lower bound. In the case of small $\sigma$, the global minimizer of the model is a large step in the negative tensor direction regardless of the local behaviour of the model.
This indicates that understanding the minimizers of the model and their dependence on $\sigma$ is important for an efficient implementation of AR$3$.

\begin{figure}
	\centering
	\includegraphics{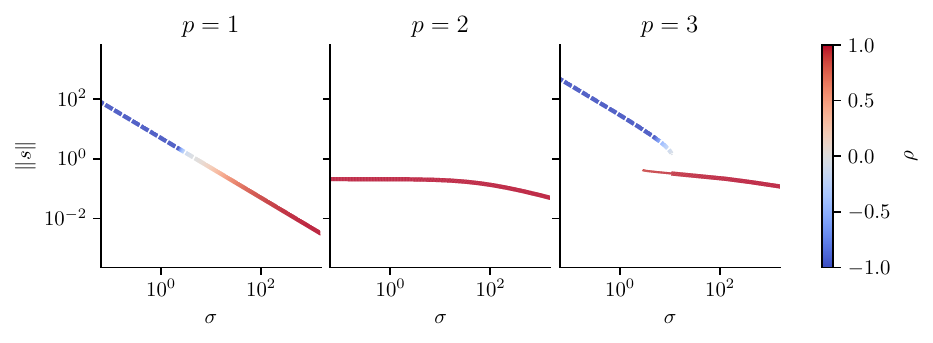}
	\caption{The graph shows how the absolute values of the local minimizers $s_*$ of the regularized model $m_{\sigma}(s)$ depend on the regularization parameter $\sigma$ for different orders $p$.
    The models are constructed at $x = 0$ for $f(x) = 3x^4 - 10x^3 + 12x^2 - 5x$. Non-global minimizers are drawn using thin lines. The values of $\rho$ are capped to $[-1, 1]$ and negative values are drawn using a dashed line.}\label{fig:discontinuity}
\end{figure}

\Cref{fig:discontinuity} illustrates how the local minimizers of AR$p$ subproblems for $p \in \{1, 2, 3\}$ change as $\sigma$ changes for one specific objective function and expansion point.
The color of the lines corresponds to $\rho$, the success measure employed in the simple update (\cref{alg:simple_sigma_update}).
In the case of $p=3$, for this example, no minimizer changes continuously as $\sigma$ increases from zero to infinity.
Instead, there are two distinct branches: one branch corresponds to a local minimizer that exists only for small values of $\sigma$, and the other corresponds to a local minimizer that exists only for large values of $\sigma$, with some overlap in the middle.
During this overlap, the global minimizer of the model switches from one branch to the other.
As indicated by the color, the local minimizers on the first branch would result in unsuccessful steps ($\rho \leq 0$), whereas the minimizers on the second branch yield successful steps ($\rho$ close to $1$).
This is an example of a new phenomenon for AR$p$ minimizers that only occurs for $p \geq 3$: a discontinuity in the minimizer curves as $\sigma$ changes.
The function $f$ is carefully chosen such that the third-order Taylor expansion is strictly monotonically decreasing but becomes very flat around the minimizer of $f$.
This leads to the behaviour seen in \cref{fig:discontinuity}, where initially there is only a single stationary point. However, once $\sigma$ becomes large enough, a local maximum forms near the flat part of the Taylor expansion, leading to two minimizers.
As $\sigma$ increases further, the local maximum merges with the \enquote*{original minimizer}, becomes a saddle point, and eventually vanishes.
See \cref{fig:prerejection_illustration_AR3_func0} for a graph of the function and the models for different values of~$\sigma$.

\subsection{Persistent and transient model minimizers}
We introduce some definitions that help describe the discontinuity phenomenon shown in \cref{fig:discontinuity}.
To simplify the notation, we write $t \colon \real^d \to \R$ for the $p$th-order polynomial representing the Taylor expansion and $m_{t, \sigma}(\ss) = t(\ss) + \frac{\sigma}{p+1} \norm{\ss}^{p+1}$ for the corresponding regularized model with regularization parameter $\sigma$.

\begin{definition}[Regularized minimizer]
    A \emph{regularized minimizer} of $t$ is a point $\ss \in \R^d$ that is a local minimizer of the regularized model $m_{t, \sigma}$ for some $\sigma > 0$.
\end{definition}

\begin{definition}[Minimizer curve]
	Let $I \subset \real_{\geq 0}$ be an interval. A \emph{minimizer curve} of $t$ is a continuous function $\vek{\psi} \colon I \to \real^d$ such that $\vek{\psi}(\sigma)$ is a local minimizer of $m_{t, \sigma}$ for every $\sigma \in I$.
    The curve is called \emph{maximal} if $\vek{\psi}$ cannot be extended to any minimizer curve on a strictly larger interval than $I$.
    By definition the image of any minimizer curve of $t$ is contained in the set of regularized minimizers of $t$ and every regularized minimizer of $t$ lies on a maximal minimizer curve of $t$.
\end{definition}

\begin{definition}[Persistent/transient minimizer]
	Let $\vek{\psi} \colon I \to \real^d$ be a maximal minimizer curve of $t$. We say
	\begin{enumerate}
		\item $\vek{\psi}$ is \emph{$c$-persistent} for $c \geq 0$ if the minimizer curve persists as $\sigma$ increases from $c$ to $\infty$, i.e., if $(c, \infty) \subset I$.
		\item $\vek{\psi}$ is \emph{persistent} if it is $c$-persistent for some $c \geq 0$.
        \item $\vek{\psi}$ is \emph{globally persistent} if it is $0$-persistent.
		\item $\vek{\psi}$ is \emph{transient} if it is not persistent.
		\item $\vek{\psi}$ is \emph{bounded} if $\sup_{\sigma \in I} \norm{\vek{\psi}(\sigma)} < \infty$.
	\end{enumerate}
	A regularized minimizer of $t$ is persistent if it lies on a persistent minimizer curve, otherwise it is transient.
	\label{def:persistent_transient_minimizers}
\end{definition}

Based on the definitions introduced above, we obtain the following theorems.
\begin{theorem}\label{thm:persistent_minimizers_local}
	Let $\vek{\psi} \colon I \to \R^d$ be a maximal minimizer curve of a Taylor expansion $t \colon \R^d \to \R$ with $\nabla t(\vek{0}) \neq \vek{0}$. Then the following statements are equivalent:
	\begin{enumerate}
		\item $\vek{\psi}$ is persistent
		\item $\vek{\psi}$ is persistent and $\lim_{\sigma \to \infty} \vek{\psi}(\sigma) = \vek{0}$
		\item $\inf_{\sigma \in I} \norm{\vek{\psi}(\sigma)} = 0$
	\end{enumerate}
\end{theorem}
\begin{proof}
    See \cref{sec:thm_prf_persistent_minimizers_local} for the proof.
\end{proof}

\begin{theorem}\label{thm:ar2_global_min_persistent}
	Let $t \colon \R^n \to \R$ be a quadratic polynomial with $\nabla t(\vek{0}) \neq \vek{0}$, and let $m_{\sigma}(\ss) = t(\ss) + \frac{\sigma}{3}\norm{\ss}^3$ be the corresponding regularized model.
	Then every global minimizer of $m_{\sigma}$ for any $\sigma \geq 0$ is a persistent regularized minimizer of $t$.
\end{theorem}
\begin{proof}
    See \cref{sec:thm_prf_ar2_global_min_persistent} for the proof.
\end{proof}

The fundamental difference between persistent and transient regularized minimizers lies in their relationship with the local behaviour of the Taylor expansion  $t$.
Specifically, we know from \cref{thm:persistent_minimizers_local} that $\lim_{\sigma \to \infty} \vek{\psi}(\sigma) = \vek{0}$ holds for any persistent minimizer curve. In contrast, for any transient curve, $\norm{\vek{\psi}(\sigma)}$ is bounded away from zero.
Thus, transient minimizers disappear when $\sigma$ becomes sufficiently large, whereas persistent minimizers follow the local descent path of the objective function.
With this in mind, our proposed pre-rejection heuristic considers all transient minimizers as not useful and rejects them without evaluating the objective function.

To illustrate \cref{def:persistent_transient_minimizers}, we describe the different cases shown in \cref{fig:discontinuity} using these definitions.
For $p=1$, there is a single maximal minimizer curve, which is persistent but not bounded.
This contrasts with the maximal minimizer curve for $p=2$, which is both persistent and bounded.
In the $p=3$ case, there are two maximal minimizer curves: one is persistent, and the other is transient.\footnote{The persistent minimizer curve is defined on $(c_1, \infty)$ with $c_1 = \frac{176}{25} - \frac{28 \sqrt{14}}{25} \approx 2.849$ and the transient one on $(0, c_2)$ with $c_2 = \frac{176}{25} + \frac{28 \sqrt{14}}{25} \approx 11.231$.}

The example in \cref{fig:discontinuity} is specifically chosen to show that for AR$3$ transient minimizers exist, but it is worth mentioning that the need to deal with transient minimizers only arises for $p \geq 3$.
For the AR$1$ subproblem of minimizing $\nabla f(\xx_k)^\T \ss + \frac{\sigma}{2}\norm{\ss}^2$, as alluded to earlier, there is only a single local and global minimizer for any $\sigma$ and it takes the form $\ss = - \sigma^{-1} \nabla f(\xx_k)$ which is continuous in $\sigma$.
Therefore, there is always a single maximal minimizer curve and it is always globally persistent.
In the AR$2$ subproblem, there can be more than one maximal minimizer curve, and some of these curves can indeed be transient, but crucially any global minimizer of the model will always be a persistent one, ensuring the existence of a globally persistent minimizer curve (see \cref{thm:ar2_global_min_persistent}).
Because solvers have been developed that can compute a global minimum of the AR$2$ subproblem (e.g. Algorithm 6.1 in \cite{cartis2011adaptiveI}), the presence of transient minimizers does not impact the outer loop of the algorithm.
The situation changes for $p=3$.
As shown in \cref{fig:discontinuity}, there are instances when no minimizer curve is globally persistent. In other words, as long as $\sigma$ is sufficiently small, any subproblem solver will find a transient minimizer.

\subsection{Detecting transient minimizers in one dimension}\label{sec:transient_minimizer_detection}

According to \cref{def:persistent_transient_minimizers}, determining whether a given minimizer of $m_{t, \sigma}$ is transient would require following its associated minimizer curve as $\sigma$ increases to infinity to see whether it eventually disappears.
Although continuation methods could potentially achieve this for arbitrary dimension $d$, such an approach would be computationally expensive, and it would be unclear at what point to stop the process and determine that the minimizer is persistent.
Fortunately, in the one-dimensional case, we can determine whether a given point is a persistent regularized minimizer of $t$ or not (\cref{thm:one_dimensional_persistent_minimizers} and \cref{rem:simple_root}).
In this scenario, it is possible to compute an interval such that all points within the interval are persistent regularized minimizers, and all points outside the interval are either not regularized minimizers or are transient regularized minimizers. Moreover, whenever this interval is bounded, the unique persistent minimizer curve is also bounded, and vice versa.
This insight will serve as the foundation for a rejection heuristic introduced in \cref{sec:transient_minimizer_rejection}.

\begin{theorem}\label{thm:one_dimensional_persistent_minimizers}
	Let $t$ be a univariate polynomial of degree $p$ such that\footnote{
	The case of $t'(0) > 0$ follows similarly as explained in \cref{rem:negative_derivative_t}.
	} $t'(0) < 0$ and define
	\begin{equation}
		\mathcal{R}_{+} = \{ \alpha > 0 \mid t'(\alpha) = 0 \text{ or } t''(\alpha)\alpha - p t'(\alpha) = 0 \}.
	\end{equation}
	If $\mathcal{R}_+$ is empty, then the persistent regularized minimizers of $t$ are exactly $\alpha \in (0, \infty)$.
	If $\mathcal{R}_+$ is non-empty and its smallest entry $\bar{\alpha}$ is a simple root of either polynomial, then all persistent regularized minimizers of $t$ are contained in $(0, \bar{\alpha}]$, and all points in $(0, \bar{\alpha})$ are persistent regularized minimizers of $t$.
\end{theorem}

\begin{proof}
	For any tuple $(\alpha, \sigma) \in \R \times \R_{\geq 0}$ such that $\alpha$ is a local minimizer of $m_{t, \sigma}$, the second-order necessary conditions imply
	\begin{subequations}\label{eqn:alpha_necessary_m}
		\begin{align}
			t'(\alpha) + \sigma \abs{\alpha}^{p-1} \alpha = m_{\sigma}'(\alpha) & = 0 \label{eqn:alpha_first_order_m}            \\
			t''(\alpha) + p \sigma \abs{\alpha}^{p-1} = m_{\sigma}''(\alpha)    & \geq 0 \label{eqn:alpha_second_order_m}        \\
			\sigma                                                              & \geq 0. \label{eqn:nonnegative_regularization}
		\end{align}
	\end{subequations}
	Since $t'(0) < 0$, we know that $\alpha \neq 0$, and we can rewrite \cref{eqn:alpha_first_order_m} as follows:
	\begin{align}\label{eqn:regularization_first_order}
		\sigma = \frac{-t'(\alpha)}{\abs{\alpha}^{p-1} \alpha}.
	\end{align}
	By combining \cref{eqn:regularization_first_order,eqn:alpha_necessary_m}, we get the following system, which is equivalent to \cref{eqn:alpha_necessary_m}:
	\begin{subequations}\label{eqn:alpha_necessary_t}
		\begin{align}
			\sigma                                     & = \frac{-t'(\alpha)}{\abs{\alpha}^{p-1} \alpha} \\
			t''(\alpha) + p \frac{-t'(\alpha)}{\alpha} & \geq 0 \label{eqn:alpha_second_order_t}         \\
			\frac{-t'(\alpha)}{\alpha}                 & \geq 0 \label{eqn:alpha_first_order_t}
		\end{align}
	\end{subequations}
	Notably, the last two inequalities do not involve $\sigma$, which means we can use them to identify which $\alpha$ are regularized minimizers of $t$.
	According to \cref{thm:persistent_minimizers_local}, any persistent minimizer lies on a continuous curve of regularized minimizers that converges to zero.
	Conversely, if there exists a point $\alpha$ that is not a regularized minimizer of $t$, we know that all regularized minimizers beyond that value in the direction away from zero must be transient.

	For $\alpha$ close enough to zero, say $\alpha \in (-\beta, \beta)$ for some $\beta > 0$, we know that $t'(\alpha)$ is negative, implying that in this region, the only solutions to \cref{eqn:alpha_first_order_t} are positive.
	Using the reasoning above, we know that there are no regularized minimizers $\alpha \in (-\beta, 0)$, and so all persistent minimizers must be positive.
	For positive $\alpha$ we can simplify conditions \cref{eqn:alpha_second_order_t,eqn:alpha_first_order_t}.
	Positive regularized minimizers of $t$ must satisfy:
	\begin{subequations}
		\label{eqn:alpha_arp}
		\begin{align}
			\alpha                            & > 0 \label{eqn:alpha_nonnegative}            \\
			-t'(\alpha)                       & \geq 0 \label{eqn:alpha_first_order_t_poly}  \\
			t''(\alpha) \alpha - p t'(\alpha) & \geq 0. \label{eqn:alpha_second_order_t_poly}
		\end{align}
	\end{subequations}

	Since $t$ is a polynomial of degree $p$, \cref{eqn:alpha_first_order_t_poly,eqn:alpha_second_order_t_poly} are polynomial inequalities of degree $p-1$, and they are satisfied for $\alpha$ small enough because $t'(0) < 0$.
	If $\mathcal{R}_+$ is non-empty, by the definition of $\bar{\alpha}$, it is the smallest $\alpha$ such that one of the two conditions is satisfied with equality.
	Furthermore, since we assume it is a simple root of one of the two polynomials, there must be a sign change in one of them.
	This means there are no minimizers $\alpha \in (\bar{\alpha}, \bar{\alpha} + \beta)$ for some $\beta > 0$.
	Using the same argument as above, all persistent minimizers must therefore be in contained $(0, \bar{\alpha}]$.
	If $\mathcal{R}_+$ is empty, on the other hand, we have already shown that all persistent minimizers are in contained in $(0, \infty)$.
	This is the first part of the claim.

	For the second part, let $\bar{\alpha} = \infty$ if $\mathcal{R}_+$ is empty.
	Using the definition of $\bar{\alpha}$ again, all $\alpha \in (0, \bar{\alpha})$ satisfy \cref{eqn:alpha_first_order_t_poly,eqn:alpha_second_order_t_poly} strictly and are strict local minimizers of the corresponding $m_{\sigma}$.
	To show that they are all persistent, let $\vek{\psi}$ be the minimizer curve defined by $\vek{\psi}(\sigma(\alpha)) = \alpha$ for $\alpha \in (0, \bar{\alpha})$ where
	\begin{equation}\label{eqn:alpha_sigma_of_alpha}
		\sigma(\alpha) = \frac{-t'(\alpha)}{\abs{\alpha}^{p-1} \alpha}.
	\end{equation}
	We need to prove that $\vek{\psi}$ is a well-defined function, that it is continuous, and that its domain is unbounded from above.
	Note first that $\sigma(\alpha)$ is continuously differentiable for positive $\alpha$ and that the derivative of $\sigma(\alpha)$ is
	\begin{equation}
		\sigma'(\alpha) = \frac{-t''(\alpha)\alpha + p t'(\alpha)}{\abs{\alpha}^{p+1}} \ \refl{eqn:alpha_second_order_t_poly} \ 0.
	\end{equation}
	This means that the function is strictly monotonically decreasing and, therefore, a one-to-one mapping between $(0, \bar{\alpha})$ and its range.
	Moreover, as $\alpha$ converges to zero, $-t'(\alpha)$ converges to a positive constant and $\abs{\alpha}^{p-1} \alpha$ converges to zero and so its range is unbounded from above.
	If $\bar{\alpha} < \infty$, this shows that $\vek{\psi} \colon (\sigma(\bar{\alpha}), \infty) \to (0, \bar{\alpha})$ defines (part of) a persistent minimizer curve, which contains all points in $(0, \bar{\alpha})$.
	When $\bar{\alpha} = \infty$, then $\lim_{\alpha \to \bar{\alpha}} \sigma(\alpha) = 0$ because $t'(\alpha) = O(\alpha^{p-1})$ as $\alpha \to \infty$, so that $\vek{\psi} \colon (0, \infty) \to (0, \infty)$ is a persistent minimizer curve, which contains all points in $(0, \infty)$.
	This completes the second part of the claim.
\end{proof}

\begin{corollary}\label{thm:0_persistent}
	Assuming $\bar{\alpha}$ as defined in \cref{thm:one_dimensional_persistent_minimizers} exists, there is a minimizer curve which is bounded and $0$-persistent if and only if $t'(\bar{\alpha}) = 0$.
\end{corollary}
\begin{proof}
    See \cref{sec:thm_prf_0_persistent} for the proof.
\end{proof}

\begin{remark}\label{rem:simple_root}
	Note that the assumption that $\bar{\alpha}$ is a simple root, if it exists, holds for a generic univariate polynomial $t$ of degree $p$, since a generic polynomial has only simple roots.
	In this case, determining the interval of persistent minimizers only requires calculating the roots of two polynomials of degree $p-1$ and identifying the smallest positive root $\bar{\alpha}$, which can be done very efficiently.
\end{remark}

\begin{remark}\label{rem:negative_derivative_t}
	The same approach applies if $t'(0) > 0$, as one can simply analyse $\tilde{t}(\alpha) = t(-\alpha)$ with $\tilde{t}'(0) < 0$ and then reverse the signs for the resulting interval.
\end{remark}

\begin{sidewaysfigure}
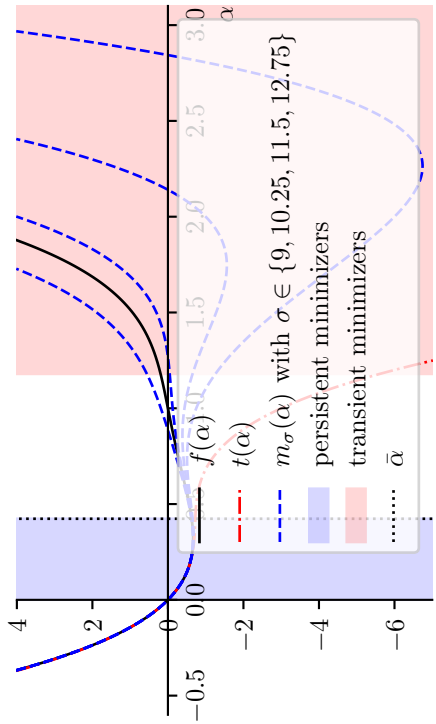
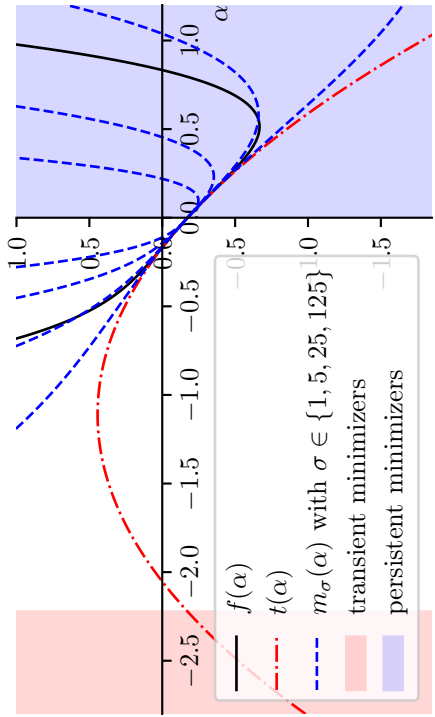
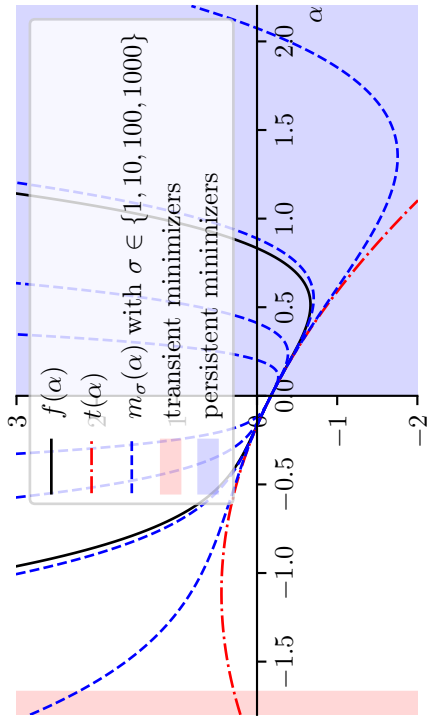

	\begin{subfigure}{0.49\textwidth}
		\centering
		\subimport{plots/illustration}{AR2-func0.pgf}
		\caption{AR$2$ models for $f(x) = 3x^4 - 10x^3 + 12x^2 - 5x$, $\sigma \in \{1, 10, 100, 1000\}$}
		\label{fig:prerejection_illustration_AR2_func0}
	\end{subfigure}
	\hfill
	\begin{subfigure}{0.49\textwidth}
		\centering
		\subimport{plots/illustration}{AR3-func0.pgf}
		\caption{AR$3$ models for $f(x) = 3x^4 - 10x^3 + 12x^2 - 5x$, $\sigma \in \{9, 10.25, 11.5, 12.75\}$}
		\label{fig:prerejection_illustration_AR3_func0}
	\end{subfigure}

	\begin{subfigure}{0.49\textwidth}
		\centering
		\subimport{plots/illustration}{AR2-func1.pgf}
		\caption{AR$2$ models for $f(x) = 3x^4 - \frac{1}{2}x^2 - \frac{10}{9} x - \frac{25}{144}$, $\sigma \in \{1, 5, 25, 125\}$}
		\label{fig:prerejection_illustration_AR2_func1}
	\end{subfigure}
	\hfill
	\begin{subfigure}{0.49\textwidth}
		\centering
		\subimport{plots/illustration}{AR3-func1.pgf}
		\caption{AR$3$ models for $f(x) = 3x^4 - \frac{1}{2}x^2 - \frac{10}{9} x - \frac{25}{144}$, $\sigma \in \{1, 10, 100, 1000\}$}
		\label{fig:prerejection_illustration_AR3_func1}
	\end{subfigure}
	\caption{
		Illustration of AR$2$ and AR$3$  models for different $\sigma$.
		The regions where \cref{eqn:alpha_necessary_t} is satisfied are shaded blue and red, depending on whether the corresponding points are persistent or transient minimizers, respectively.
	}
	\label{fig:prerejection_illustration}
\end{sidewaysfigure}

To illustrate the idea of persistent and transient minimizers we consider the univariate function $f(x) = 3x^4 - 10x^3 + 12x^2 - 5x$ from \cref{fig:discontinuity} and its Taylor expansion at $x = 0$ again.
\Cref{fig:prerejection_illustration_AR3_func0} shows how the AR$3$ local model changes as $\sigma$ varies.
The rightmost valley and its local minimizer disappear as $\sigma$ increases, exactly as shown in \cref{fig:discontinuity}.
Using \cref{thm:one_dimensional_persistent_minimizers}, we find that  $\alpha \in [\alpha_{\textrm{left}}, \alpha_{\textrm{right}}] \coloneqq [(4 - \sqrt{7/2})/5, (4 + \sqrt{7/2})/5] \approx [0.426, 1.174]$ are never minimizers of the regularized model for any value of $\sigma$.
Therefore, all points to the right of that interval, namely, $\alpha > \alpha_{\textrm{right}}$, are transient minimizers (marked in red), and all points to the left up to zero, i.e., $\alpha \in (0,  \alpha_{\textrm{left}})$,  are persistent minimizers (marked in blue); here $\bar{\alpha} = \alpha_{\textrm{left}}$.

\Cref{fig:prerejection_illustration_AR2_func0} shows the AR$2$ models for the same function and expansion point:
since the function is locally convex around the expansion point, the second-order Taylor expansion is a convex quadratic with a bounded minimizer even when $\sigma = 0$.
This unique model minimizer slowly shifts towards zero when $\sigma$ increases as also shown by \cref{fig:discontinuity}. In this case, all possible minimizers are persistent for the AR$2$ model.

The locally nonconvex case is depicted in \cref{fig:prerejection_illustration_AR3_func1,fig:prerejection_illustration_AR2_func1}.
Here, we consider essentially the same fourth-degree polynomial as before but rotated and expanded at a different point ($x = 5/6$).
Both AR$2$ and AR$3$ have transient minimizers that correspond to ascent directions and all values $\alpha \in (0, \infty)$ are persistent.

\subsection{Rejecting directionally transient minimizers}\label{sec:transient_minimizer_rejection}

As motivated at the beginning of the section, transient minimizers are unlikely to accurately predict the behavior of the objective function.
As a practical technique, again following the ray-based minimization idea \ref{itm:simplification-ray} from \cref{sec:interpolation-sigma-update}, we take the step $\ss_k$ found by a generic unconstrained nonconvex optimization algorithm applied to the subproblem, consider the one-dimensional subspace spanned by $\ss_k$ and determine whether the step is a persistent minimizer along this subspace.
If $\alpha = \norm{\ss_k}$ is a persistent regularized minimizer of $t_{\rightarrow}(\alpha) = t^p_{\xx_k}(\alpha \ss_k / \norm{\ss_k})$, we classify the step as \emph{directionally persistent}, otherwise as \emph{directionally transient}.
According to this definition, a step is automatically directionally transient if it is not a descent direction, i.e., if $\nabla f(\xx_k)^\T \ss_k \geq 0$.
Any directionally transient step is rejected without evaluating the objective function at the potential new iterate, and $\sigma_k$ is increased by a constant factor.
For any directionally persistent step, the function is evaluated, and $\sigma_k$ is updated as before.
The implementation of this technique is described in \cref{alg:prerejection}.
It can be combined with any existing updating strategy in \cref{lin:prerejection_sigma_update}.
We denote AR$p$ variants that have been enhanced with our pre-rejection technique with a superscript \textsf{\textsuperscript{+}}, so \textsf{AR$3$-Simple} and \textsf{AR$3$-Interp} become \textsf{AR$3$-Simple\textsuperscript{+}} and \textsf{AR$3$-Interp\textsuperscript{+}}.

\begin{remark}\label{rem:interp_plus}
	When combining \cref{alg:prerejection} and \cref{alg:interpolation_sigma_update_full} into \textsf{AR$3$-Interp\textsuperscript{+}}, we also adapt the interpolation-based search inside \cref{alg:interpolation_sigma_update_full} to only consider directionally persistent minimizers.
	This is  done  by simply adding the \enquote*{constraint} $\alpha \leq \bar{\alpha}$ to \cref{eqn:interpolation_optimization_formulation_alpha}, for $\bar{\alpha}$ calculated in \cref{lin:calculate_alpha_bar} of \cref{alg:prerejection}.
\end{remark}

\begin{procedure}
	\caption{Transient minimizer pre-rejection update}
	\label[procedure]{alg:prerejection}
	\KwParameters{$\gamma_2 > 0$}
	\KwDefaults{$\gamma_2 = 3$}
	\If(\tcp*[f]{directionally transient}){$\nabla f(\xx_k)^\T \ss_k \geq 0$ \label{lin:reject_nondescent_direction}}{
		Set $\xx_{k+1} = \xx_k$ and $\sigma_{k+1} = \gamma_2 \sigma_k$\;
	}
	Define $t_{\rightarrow}(\alpha) = t^p_{\xx_k}(\alpha \ss_k / \norm{\ss_k})$ and $m_{\rightarrow, \sigma}(\alpha) = m_{t_{\rightarrow}, \sigma}(\alpha) = m_{\xx_k, \sigma}^p (\alpha \ss_k / \norm{\ss_k})$\;
	Define $\xi = \max(0, m_{\rightarrow, \sigma_k}'(\norm{\ss_k}))$\; \label{lin:relaxation_xi}
	Compute the set of roots $\mathcal{R} = \{\xi - t_{\rightarrow}'(\alpha) = 0\} \cup \{t_{\rightarrow}''(\alpha)\alpha + p (\xi - t_{\rightarrow}'(\alpha)) = 0\}$\;
	Select the smallest positive root $\bar{\alpha} \in \mathcal{R}$ or let $\bar{\alpha} = \infty$ if there is none \label{lin:calculate_alpha_bar}\;
	\uIf(\tcp*[f]{directionally persistent}){$\norm{\ss_k} \leq \bar{\alpha}$}{
		Compute $\xx_{k+1}$ and $\sigma_{k+1}$ using some updating strategy \label{lin:prerejection_sigma_update}
		\;
	}
	\Else(\tcp*[f]{directionally transient}){
		Set $\xx_{k+1} = \xx_k$ and $\sigma_{k+1} = \gamma_2 \sigma_k$\;
	}
\end{procedure}

Note that there is a slight modification to the procedure derived in the previous section: the computation of $\xi$ in \cref{lin:relaxation_xi} of \cref{alg:prerejection}, which aims to resolve an issue arising when using \cref{thm:one_dimensional_persistent_minimizers} in practice.
When implementing the inequalities in \cref{eqn:alpha_arp} strictly ($\xi = 0$), the method would reject points even if they are extremely close to persistent minimizers as long as they are just outside the interval of persistent minimizers.
Since all subproblem solvers return approximate stationary points we must relax the requirements on $\ss_k$ slightly to avoid this situation.
Instead of requiring $\alpha$ to be a stationary point of $m_{\rightarrow, \sigma}$ in \cref{eqn:alpha_first_order_m}, we require $m_{\rightarrow, \sigma}'(\alpha) \leq \xi$.
Given that $\alpha > 0$ is naturally ensured by \cref{lin:reject_nondescent_direction} of \cref{alg:prerejection}, we can follow the derivation to obtain
\begin{align}
	\label{eqn:persistent_sigma}
	\sigma \leq \frac{\xi - t_{\rightarrow}'(\alpha)}{\alpha^{p}}
\end{align}
which means \cref{eqn:alpha_first_order_t_poly,eqn:alpha_second_order_t_poly} need to be adapted as
\begin{equation}\label{eqn:alpha_relaxed_optimality}
	\xi - t_{\rightarrow}'(\alpha) \geq 0 \quad \text{and} \quad t_{\rightarrow}''(\alpha)\alpha + p(\xi - t_{\rightarrow}'(\alpha)) \geq 0.
\end{equation}
By choosing $\xi = \max(0, m_{\rightarrow, \sigma_k}'(\norm{\ss_k}))$, we make sure that $\xi$ is never negative and that the relaxed stationarity condition $m_{\rightarrow, \sigma}'(\alpha) \leq \xi$ is satisfied at $\alpha = \norm{\ss_k}$ for the current regularization parameter $\sigma_k$.
In other words, since $\ss_k$ is considered to be an approximate stationary point by the subproblem solver, we relax \cref{eqn:alpha_first_order_m} just enough such that $\alpha = \norm{\ss_k}$ is also considered \enquote*{stationary enough} by \cref{alg:prerejection}.

\begin{sidewaysfigure}
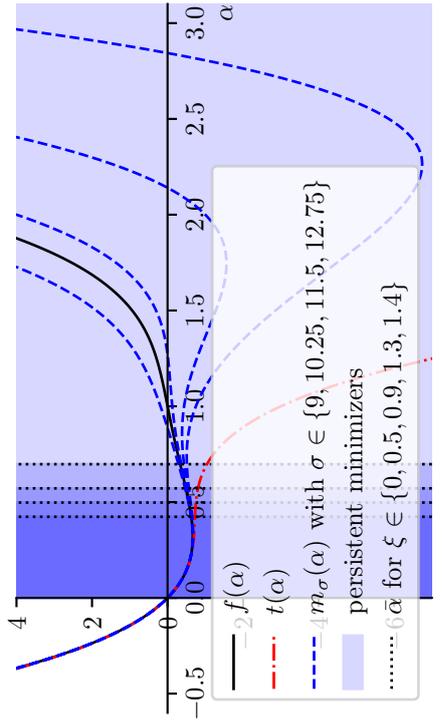
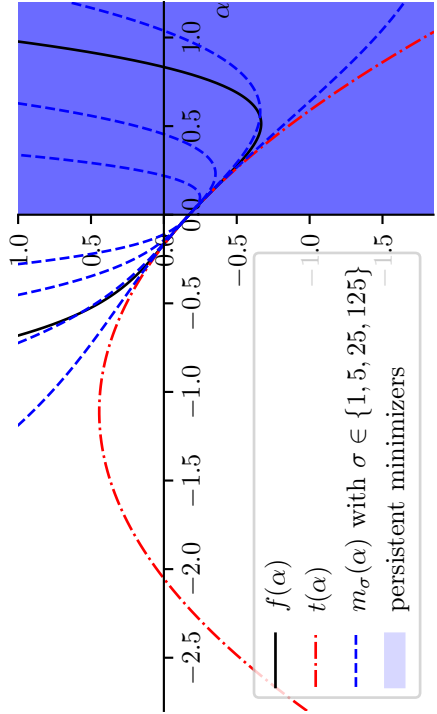
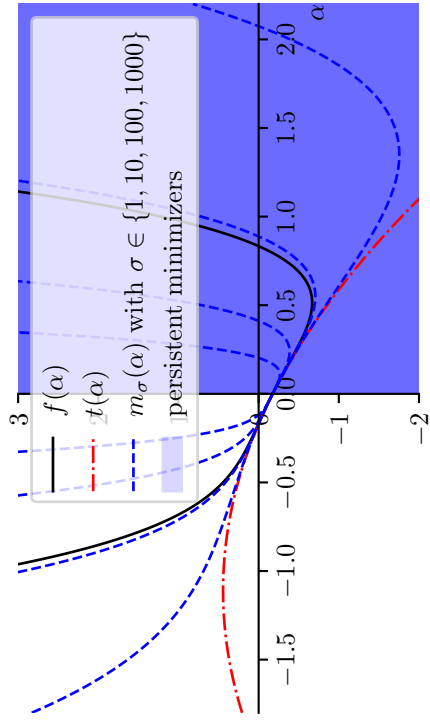

	\begin{subfigure}{0.49\textwidth}
		\centering
		\subimport{plots/illustration}{AR2-func0-relaxed.pgf}
		\caption{AR$2$ models for $f(x) = 3x^4 - 10x^3 + 12x^2 - 5x$, $\sigma \in \{1, 10, 100, 1000\}$ and $\xi \in \{0, 1.5, 10, 20\}$}
		\label{fig:prerejection_illustration_AR2_func0_relaxed}
	\end{subfigure}
	\hfill
	\begin{subfigure}{0.49\textwidth}
		\centering
		\subimport{plots/illustration}{AR3-func0-relaxed.pgf}
		\caption{AR$3$ models for $f(x) = 3x^4 - 10x^3 + 12x^2 - 5x$, $\sigma \in \{9, 10.25, 11.5, 12.75\}$ and $\xi \in \{0, 0.5, 0.9, 1.3, 1.4\}$}
		\label{fig:prerejection_illustration_AR3_func0_relaxed}
	\end{subfigure}

	\begin{subfigure}{0.49\textwidth}
		\centering
		\subimport{plots/illustration}{AR2-func1-relaxed.pgf}
		\caption{AR$2$ models for $f(x) = 3x^4 - \frac{1}{2}x^2 - \frac{10}{9} x - \frac{25}{144}$, $\sigma \in \{1, 5, 25, 125\}$}
		\label{fig:prerejection_illustration_AR2_func1_relaxed}
	\end{subfigure}
	\hfill
	\begin{subfigure}{0.49\textwidth}
		\centering
		\subimport{plots/illustration}{AR3-func1-relaxed.pgf}
		\caption{AR$3$ models for $f(x) = 3x^4 - \frac{1}{2}x^2 - \frac{10}{9} x - \frac{25}{144}$, $\sigma \in \{1, 10, 100, 1000\}$}
		\label{fig:prerejection_illustration_AR3_func1_relaxed}
	\end{subfigure}
	\caption{
		Illustration of AR$2$ and AR$3$ models and the impact of the relaxation parameter $\xi$ in \cref{alg:prerejection}.
		The persistent minimizer regions $(0, \bar{\alpha})$ are shaded in blue with regions overlapping for different values of $\xi$.
		The graphs clearly demonstrate how increasing $\xi$ relaxes the requirements on $\alpha$.
	}
	\label{fig:prerejection_illustration_relaxed}
\end{sidewaysfigure}

\Cref{fig:prerejection_illustration_relaxed} shows the same cases as \cref{fig:prerejection_illustration} but focuses on the effect of $\xi$ on the value of $\bar{\alpha}$.
Consider \cref{fig:prerejection_illustration_AR3_func0_relaxed} as an example.
As $\xi$ increases, the blue region, representing the interval $(0, \bar{\alpha}]$ expands until at $\xi=1.4$, $\bar{\alpha}$ becomes $+\infty$.
The depth of the colored region indicates the shared areas resulting from different choices of $\xi$; for example, the darkest region corresponds to $\xi=0$, as it is mutually contained by all provided $\xi$ values.
The same expansion of the blue region can also be observed in \cref{fig:prerejection_illustration_AR2_func0_relaxed}.
For \cref{fig:prerejection_illustration_AR2_func1_relaxed,fig:prerejection_illustration_AR3_func1_relaxed} the persistent minimizer region is already unbounded for $\xi = 0$, so any further relaxation does not change anything.

Note that, as explained above, $\xi$ is chosen depending on
\begin{equation}
	m_{\rightarrow, \sigma_k}'(\norm{\ss_k}) = \nabla m_k(\ss_k)^\T (\ss_k / \norm{\ss_k}) \leq \norm{\nabla m_k{(\ss_k)}},
\end{equation}
which in turn depends on the termination criterion for the subproblem solver.
In the following experiments, the subproblem solver is terminated using \cref{eqn:subproblem_absolute_tc} with $\eps_{\textrm{sub}} = 10^{-9}$, so that the conditions on $\alpha$ are only relaxed very little.
The values of $\xi$ in \cref{fig:prerejection_illustration_relaxed} are chosen much larger for illustration purposes.

\subsection{Numerical study of the pre-rejection technique}

\begin{sidewaysfigure}
	\centering
	\subimport{plots/prerejection}{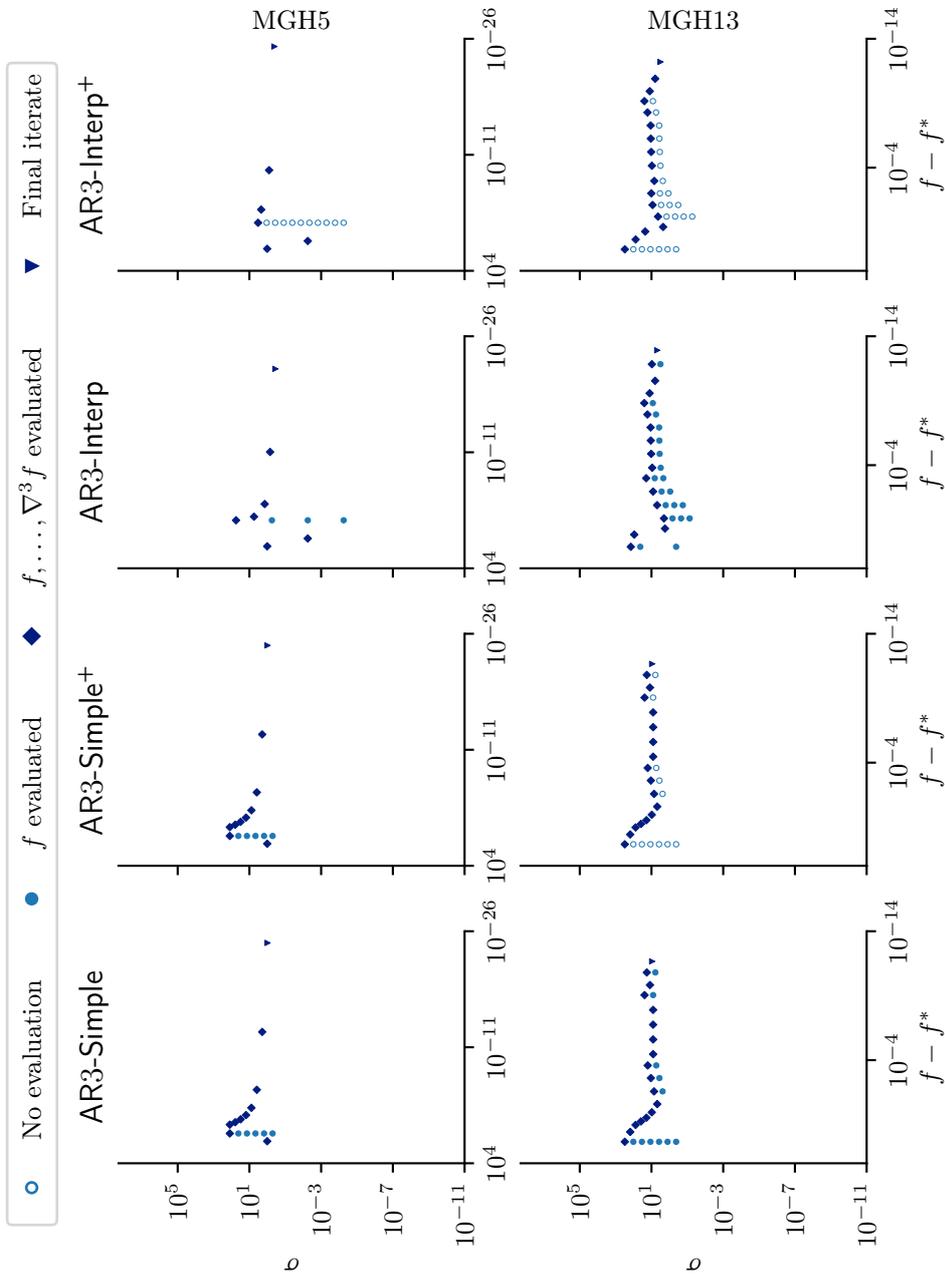}
	\caption{
    The convergence dot plot illustrates the impact of including pre-rejection.
    The pre-rejection technique is shown to effectively \enquote*{predict} unsuccessful iterations and avoid the corresponding function evaluations in both \textsf{AR$3$-Simple\textsuperscript{+}} and \textsf{AR$3$-Interp\textsuperscript{+}}.
    Additionally, we do not find any cases where pre-rejection incorrectly rejects a successful iteration, which would have resulted in a large $\sigma$ or more iterations.
    See \cref{sec:convergence_dot} for details on the convergence dot plots.
    }
    \label{fig:prerejection-convergence}
\end{sidewaysfigure}

\begin{sidewaysfigure}
	\centering
	\subimport{plots/prerejection}{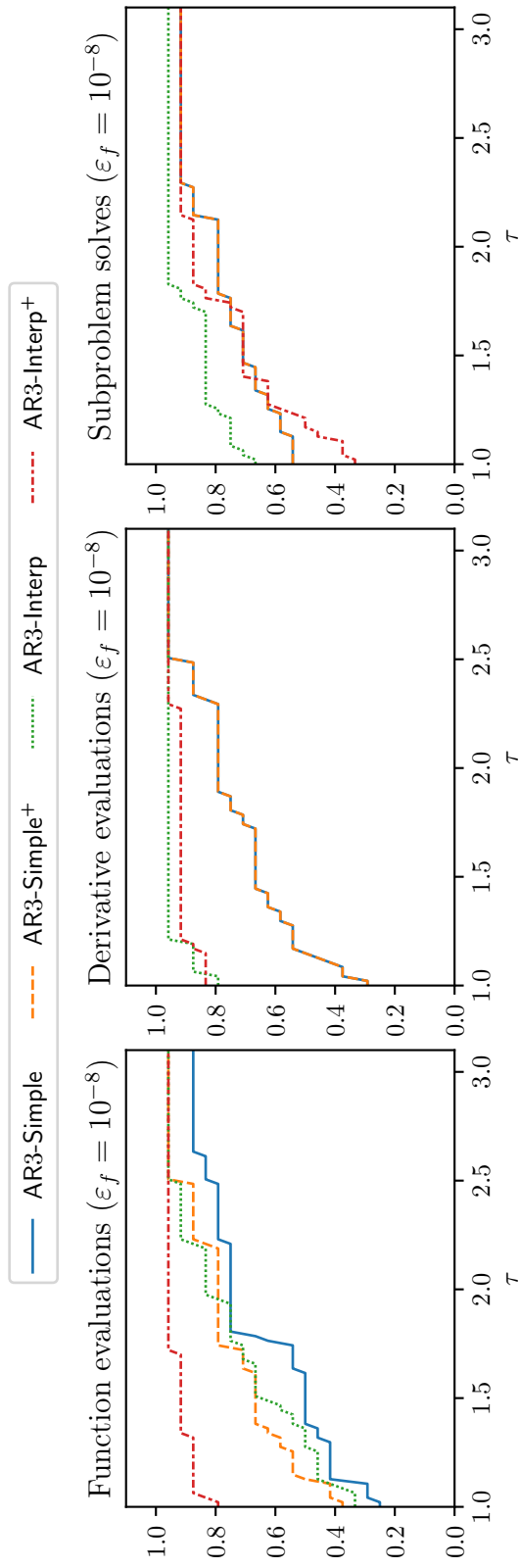}
	\caption{
    The performance profile plot illustrates the impact of including pre-rejection.
    \textsf{AR$3$-Interp\textsuperscript{+}} shows exceptional efficiency in avoiding unnecessary function evaluations, although it increases the number of subproblem solves.
    However, the enhancement in saving function evaluations persists without causing any additional solves for \textsf{AR$3$-Simple\textsuperscript{+}}.
    See \cref{sec:performance_profile} for details on the performance profile plots.
    }
    \label{fig:prerejection-performance-profile}
\end{sidewaysfigure}

In \cref{fig:prerejection-convergence}, we can observe in detail how the pre-rejection mechanism enhances performance.
For MGH5, the total number of function evaluations for \textsf{AR$3$-Interp} is nearly double compared to \textsf{AR$3$-Interp\textsuperscript{+}}.
In the case of MGH13, it was previously shown in \cref{fig:interpolation-convergence} that \textsf{AR$3$-Interp} used more function evaluations than \textsf{AR$3$-Simple}.
However, when using the pre-rejection framework, \textsf{AR$3$-Interp\textsuperscript{+}} avoids additional function evaluations, resulting in the same number of function evaluations as \textsf{AR$3$-Simple\textsuperscript{+}}.
The pre-rejection framework effectively rejects unnecessary function evaluations in unsuccessful cases, as shown for MGH13 in \cref{fig:interpolation-convergence}.

In \cref{fig:prerejection-performance-profile}, we observe that the performance of algorithms with our pre-rejection \cref{alg:prerejection} update is promising in terms of reducing the number of function evaluations.
Both \textsf{AR$3$-Simple} and \textsf{AR$3$-Interp} are enhanced by incorporating \cref{alg:prerejection}. The number of derivative evaluations remains identical for \textsf{AR$3$-Simple} and is similar for \textsf{AR$3$-Interp}, indicating that the additional pre-rejection framework does not increase the number of successful iterations while effectively reducing the number of function evaluations in unsuccessful iterations.
This suggests that the pre-rejection mechanism optimizes the overall process by cutting down unnecessary function evaluations without negatively impacting the convergence behaviour.

\begin{remark}
\label{rem:eta_1}
In AR$3$, the choice of the acceptance threshold $\eta_1$ plays a more delicate role than in AR$2$.
If no pre-rejection is used, the performance can vary substantially.
As illustrated in \cref{fig:discontinuity,fig:prerejection_illustration}, transient minimizers can have large norms regardless of the definiteness of the Hessian, and thus $\deltapred_k$ is expected to be also large since $\deltapred_k = t_k(\vek{0}) - t_k(\ss_k) > \sigma_k || \ss_k ||^{4} / 4$, which in turn implies that $\rho_k$ is likely to be very small. In this sense, a small $\eta_1$ increases the chance that transient minimizers (that still yield decrease) are accepted as successful.
However, when the pre-rejection is enabled and transient minimizers are rejected, this effect is expected to be limited and the behavior closer to that of AR$2$. The dependence of $\rho_k$ on $\sigma_k$ is then effectively continuous as shown in \cref{fig:discontinuity}, and the overall performance becomes less sensitive to moderate changes in $\eta_1$.
\end{remark}

\begin{remark}
\label{rem:prerejection}
The use of pre-rejection is not always the optimal choice. When $\nabla^p f$ is Lipschitz continuous with constant $L_p$ and $\sigma \ge \frac{p+1}{p!} L_p$, the AR$p$ model overestimates the objective function everywhere \cite[Lemma 2.1]{cartis2020concise}. In this case, choosing the global minimizer of the model, whether transient or persistent, is the best option as it provides the largest predicted decrease.
Moreover, selecting any local minimizer is sufficient to retain the optimal global complexity of the AR$p$ algorithm \cite{cartis2022evaluation}. Therefore, it is not recommended to apply pre-rejection when $\sigma_k$ is known to be sufficiently large such that $m_k(\ss_k) \ge f(\xx_k + \ss_k)$ holds.
In practice, however, the adaptively chosen $\sigma_k$ is often small enough that the global minimizer of the AR$p$ subproblem may lie far from the expansion point, where the Taylor model does not reflect the objective function correctly. In such cases, the pre-rejection mechanism (described below) typically classifies this minimizer as transient and rejects it. We will see later that the unsuccessful steps are correctly rejected by the pre-rejection mechanism, without computing additional function evaluations, as illustrated in \cref{fig:prerejection-convergence} and \cref{fig:prerejection-performance-profile}.
\end{remark}

\subsection{Comparison with the BGMS step control strategy}

\begin{procedure}[ht]
	\KwParameters{$\hat{\eta}_1, \hat{\eta}_2 > 0$, $J \in \N$, $\alpha > 0$, $0 < \gamma_1 < 1 < \gamma_2$, $\sigma_{\min} > 0$}
	\KwDefaults{$\hat{\eta}_1 = 10^3$, $\hat{\eta}_2 = 3$, $J = 20$, $\alpha = 10^{-8}$, $\gamma_1=0.5$, $\gamma_2=10$, $\sigma_{\min}=10^{-8}$}
	\KwNote{The procedure requires two additional variables: $j_k$ and $\sigma_k^{\textrm{ini}}$. Before the first iteration, the initial regularization parameter computed in \cref{eqn:sigma0_taylor} is stored in $\sigma_0^{\textrm{ini}}$, and \(\sigma_0\) and \(j_0\) are set zero.}
	\uIf{$\sigma_k = 0$ and no $\ss_k$ could be found}{
		Set $j_{k+1} = j_k + 1$, $\xx_{k+1} = \xx_k$, and $\sigma_{k+1} = \sigma_k^{\textrm{ini}}$
	}
	\Else{
		\uIf{$j_k \geq J$ or $\big( \frac{t_k(\mathbf{0}) - t_k(\ss_k)}{\max \{1, \abs{t_k(\mathbf{0})}\}} \leq \hat{\eta}_1$ and $\frac{\norm{\ss_k}_\infty}{\max \{1, \norm{\xx_k}_\infty\}} \leq \hat{\eta}_2 \big)$ \label{lin:bgms_safeguarded_prerejection}}{
			\uIf{$f(\xx_k + \ss_k) \leq f(\xx_k) - \alpha \norm{\ss_k}^{p+1}$}{
				Set $j_{k+1} = 0$, $\xx_{k+1} = \xx_k + \ss_k$, and $\sigma_{k+1} = 0$
			}
			\Else{
				Set $j_{k+1} = j_k + 1$, $\xx_{k+1} = \xx_k$, and $\sigma_{k+1} = \begin{cases}
					\sigma_k^{\textrm{ini}} & \text{if } \sigma_k = 0 \\
					\gamma_2 \sigma_k & \text{otherwise}
				\end{cases}$
			}
		}
		\Else{
			Set $j_{k+1} = j_k + 1$, $\xx_{k+1} = \xx_k$, and $\sigma_{k+1} = \begin{cases}
				\sigma_k^{\textrm{ini}} & \text{if } \sigma_k = 0 \\
				\gamma_2 \sigma_k & \text{otherwise}
			\end{cases}$
		}
	}
	Set $\sigma_{k+1}^{\textrm{ini}} = \max\{\gamma_1 \sigma_k, \, \gamma_1 \sigma_k^{\textrm{ini}}, \, \sigma_{\min}\}$\;
	\caption{BGMS update (Algorithm 4.1 in \cite{birgin2020use})}
	\label[procedure]{alg:bgms_sigma_update}
\end{procedure}

BGMS is a regularized third-order tensor method that incorporates both an updating strategy and a heuristic pre-rejection, as detailed in \cite[Algorithm 4.1]{birgin2020use}.
BGMS introduces several novel features, including:
\begin{itemize}
	\item \emph{Trial step with no regularization}: After every successful iteration, the new Taylor expansion is constructed and minimized without regularization ($\sigma = 0$). When no minimizer can be found or the step is not successful, the method switches to a nonzero regularization parameter ($\sigma > 0$).
	\item \emph{Pre-rejection of long and overly optimistic steps}: Whenever one of the following conditions,
	\begin{align}
		\label{eqn:BGMS_prerejection}
		\frac{t_k(\mathbf{0}) - t_k(\ss_k)}{\max \{1, \abs{t_k(\mathbf{0})}\}} \leq \hat{\eta}_1 \quad \text{and} \quad \frac{\norm{\ss_k}_\infty}{\max \{1, \norm{\xx_k}_\infty\}} \leq \hat{\eta}_2
	\end{align}
	is not satisfied, the step is rejected without evaluating the function at $\xx_k + \ss_k$.
	\item \emph{Safeguarded pre-rejection}: After $J$ consecutive unsuccessful iterations, the pre-rejection mechanism is deactivated until the next successful step, which means $f(\xx_k + \ss_k)$ is evaluated even when \cref{eqn:BGMS_prerejection} is violated.
	The authors report best results for $J = 20$, in which case this safeguard is almost never used, and can be understood as a device to ensure optimal theoretical complexity despite the heuristic nature of \cref{eqn:BGMS_prerejection}.
\end{itemize}
These features allow the BGMS algorithm to significantly reduce the total number of function evaluations.
Specifically, the first inequality in \cref{eqn:BGMS_prerejection} rejects steps with large predicted decrease, while the second inequality rejects steps with large entries.
Therefore, this safeguarded pre-rejection approach serves to \enquote*{discard unuseful steps} \cite{birgin2020use} without compromising the theoretical optimal complexity results derived in the same paper.

\begin{figure}
	\centering
	\subimport{plots/updates}{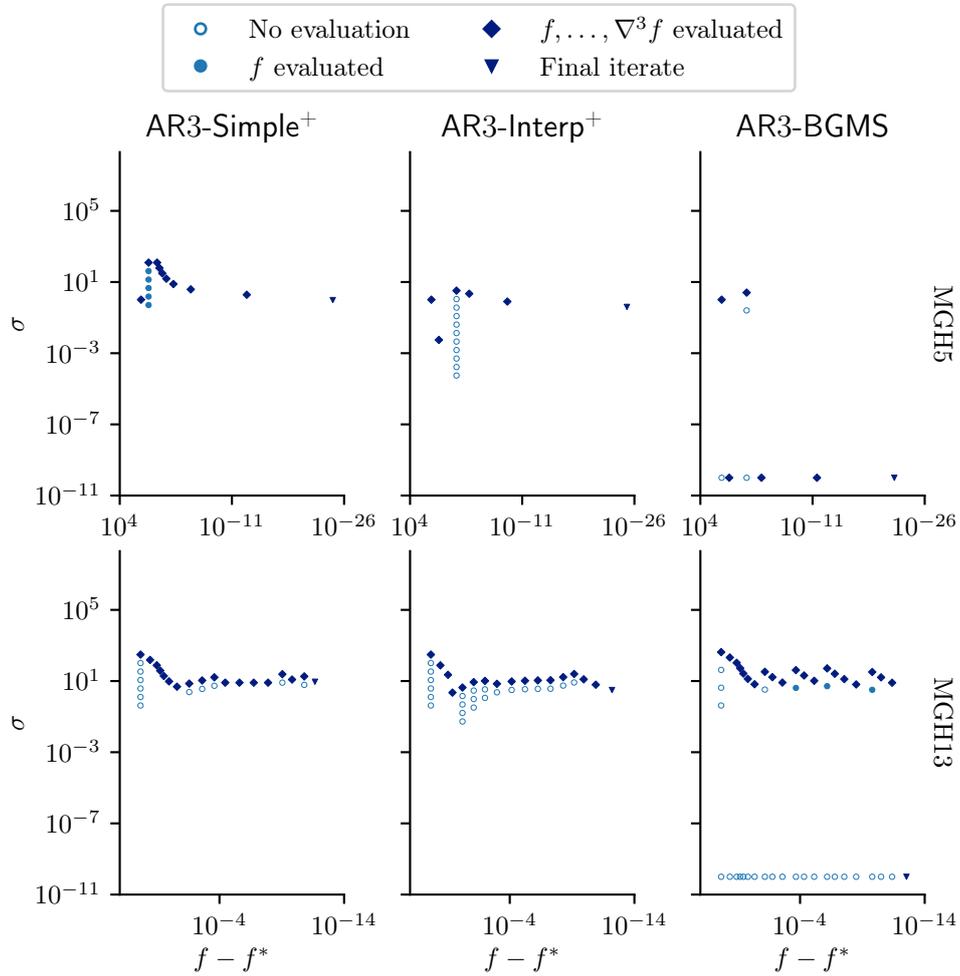}
	\caption{
    The convergence dot plot highlights the differences between the three update strategies. Note that pre-rejection is enabled for all three methods.
    \textsf{AR$3$-BGMS} demonstrates its unique ability to leverage large step-size trials with $\sigma_k = 0$ on MGH5.
    However, these attempts fail for MGH13, though fortunately, the objective function evaluation is avoided in each case.
    See \cref{sec:convergence_dot} for details on the convergence dot plots.
    }
    \label{fig:updates-convergence}
\end{figure}

\begin{sidewaysfigure}
	\centering
	\subimport{plots/updates}{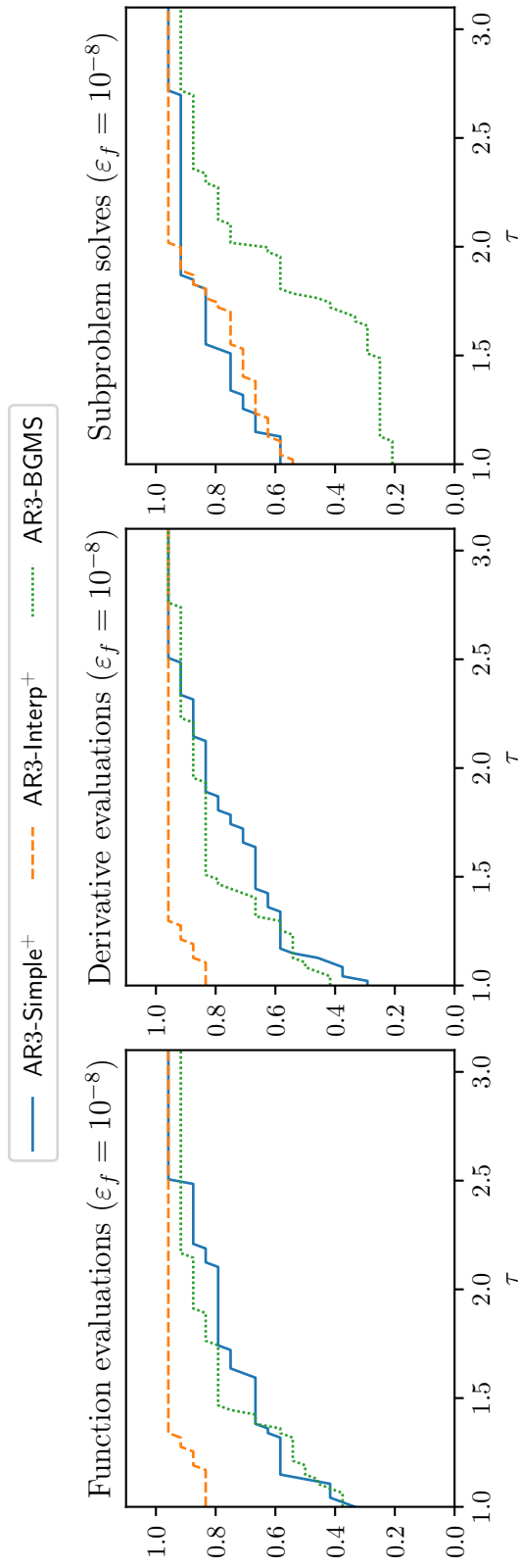}
	\caption{
    The performance profile plot highlights the differences between the three update strategies.
    Note that pre-rejection is enabled for all three methods.
    \textsf{AR$3$-Interp\textsuperscript{+}} outperforms the other two methods in terms of function evaluations and derivative evaluations.
    \textsf{AR$3$-BGMS} is comparable to \textsf{AR$3$-Simple\textsuperscript{+}} but is less efficient in subproblem solves.
    See \cref{sec:performance_profile} for details on the performance profile plots.
    }
    \label{fig:updates-performance-profile}
\end{sidewaysfigure}

\Cref{fig:updates-convergence,fig:updates-performance-profile} compare the simple and interpolation-based updating strategies enhanced by \cref{alg:prerejection} with the updating strategy and pre-rejection framework in \cref{alg:bgms_sigma_update}.
\textsf{AR$3$-BGMS} represents our implementation of the BGMS update.
In particular, it uses the same subproblem solver as \textsf{AR$3$-Simple} and \textsf{AR$3$-Interp} (see \cref{sec:implementation_details}) to ensure a fair comparison.
For the parameter settings of \textsf{AR$3$-BGMS}, we adhere to the parameters recommended in \cite{birgin2020use}.

The convergence dot plots in \cref{fig:updates-convergence} clearly show the attempts at calculating a model minimizer of the unregularized Taylor expansion with dots at the bottom of the graph.
Note that iterations where \(\sigma_k = 0\) are plotted as if \(\sigma_k = 10^{-10}\) to make them visible on the log-scaled axis.\footnote{As \(10^{-10} < \sigma_{\min} = 10^{-8}\) there is no risk of confusion.}
For MGH5,  the zero regularization attempt is highly effective. Three iterations successfully make use of the minimizer of the unregularized Taylor expansion and \textsf{AR$3$-BGMS} terminates after six iterations, which is the same number required by \textsf{AR$3$-Interp\textsuperscript{+}}.
In contrast, for MGH13, the zero regularization attempts are never successful.
Additionally, there are unsuccessful iterations where the function is evaluated, unlike in \textsf{AR$3$-Interp\textsuperscript{+}}, indicating that the pre-rejection mechanism in \textsf{AR$3$-BGMS} may be less efficient compared to the one discussed in this section.

The performance profile plots in \cref{fig:updates-performance-profile} demonstrate that \textsf{AR$3$-Interp\textsuperscript{+}} is more efficient than the other two methods in terms of both function evaluations and derivative evaluations.
While \textsf{AR$3$-BGMS} generally outperforms \textsf{AR$3$-Simple\textsuperscript{+}} across most problems, it lags in the subproblem solving plot due to its attempts to not use regularization after every successful iteration.

\section{Subproblem termination conditions}
\label{sec:theta}

So far all experiments used \cref{eqn:subproblem_absolute_tc} with $\eps_{\textrm{sub}} = 10^{-9}$ to decide when to terminate the search for a subproblem minimizer, which we consider as a high accuracy condition.
In order to investigate the impact of the subproblem termination condition on the performance of the algorithmic variants we are investigating/proposing, we compare this condition with the relative condition \cref{eqn:subproblem_relative_tc} for $\theta \in \{10^{-2}, 10^{0}, 10^{2}, 10^{4}\}$.
For larger steps, the latter condition allows for significantly more inexactness in the model solution, which can also lead to additional relaxation in the pre-rejection procedure \cref{alg:prerejection}.
Additional experiments, conducted to compare variants using the \enquote*{generalized norm} termination condition \cite{gratton2023adaptive}, are discussed in \cref{sec:generalized_norm}.

\subsection{Numerical study of subproblem termination conditions}

\begin{sidewaysfigure}
	\centering
	\subimport{plots/theta}{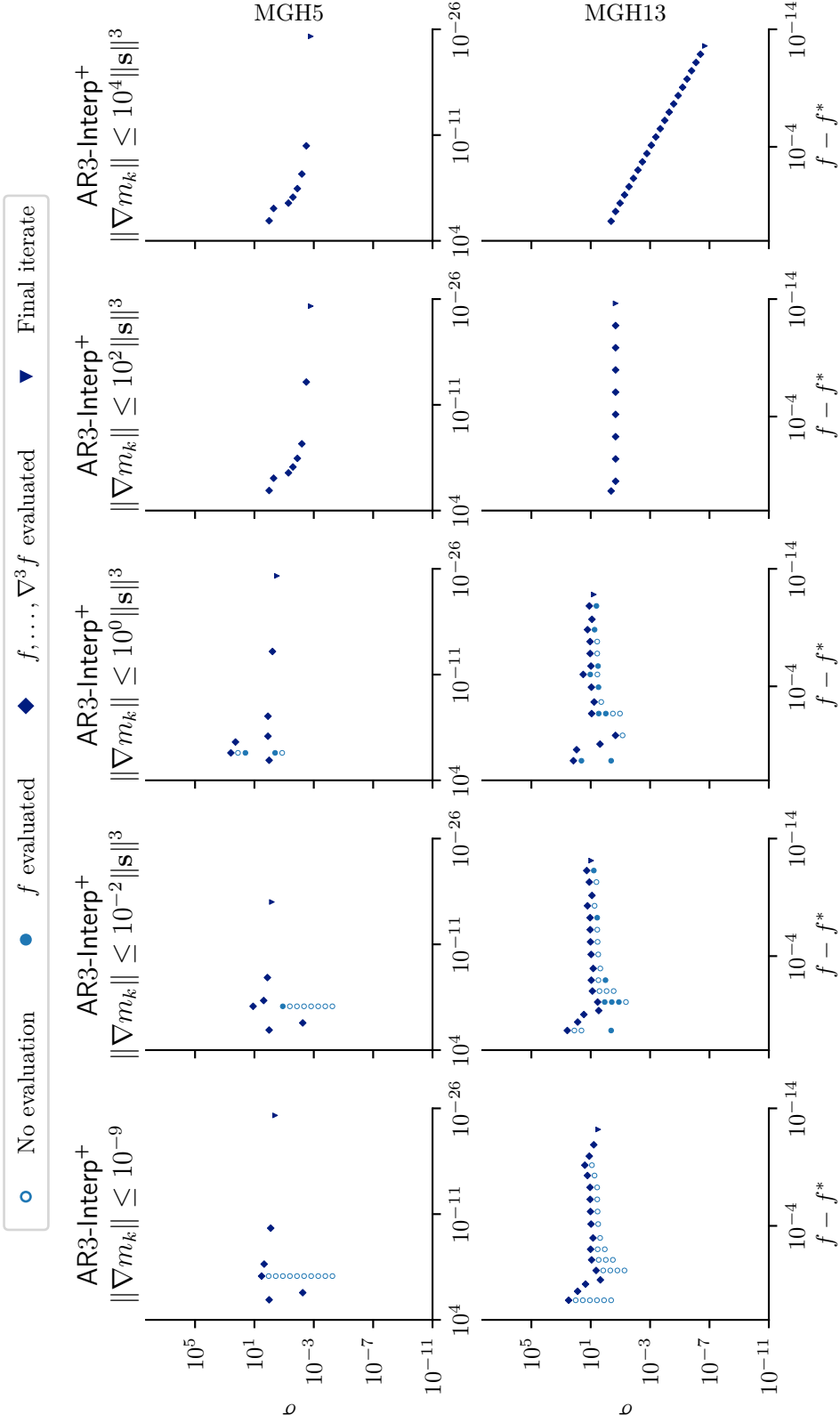}
	\caption{
    The convergence dot plot shows the differences between different subproblem termination conditions, namely \cref{eqn:subproblem_absolute_tc} with $\eps_{\textrm{sub}} = 10^{-9}$ and \cref{eqn:subproblem_relative_tc} with $\theta \in \{10^{-2}, 10^{0}, 10^{2}, 10^{4}\}$.
    Iterations with small $\sigma_k$ are more frequently successful for inexact solves.
    In the case of MGH13, choosing $\theta$ too large, however, leads to slow convergence.
    See \cref{sec:convergence_dot} for details on the convergence dot plots.
    }
    \label{fig:theta-convergence}
\end{sidewaysfigure}

\begin{sidewaysfigure}
	\centering
	\subimport{plots/theta}{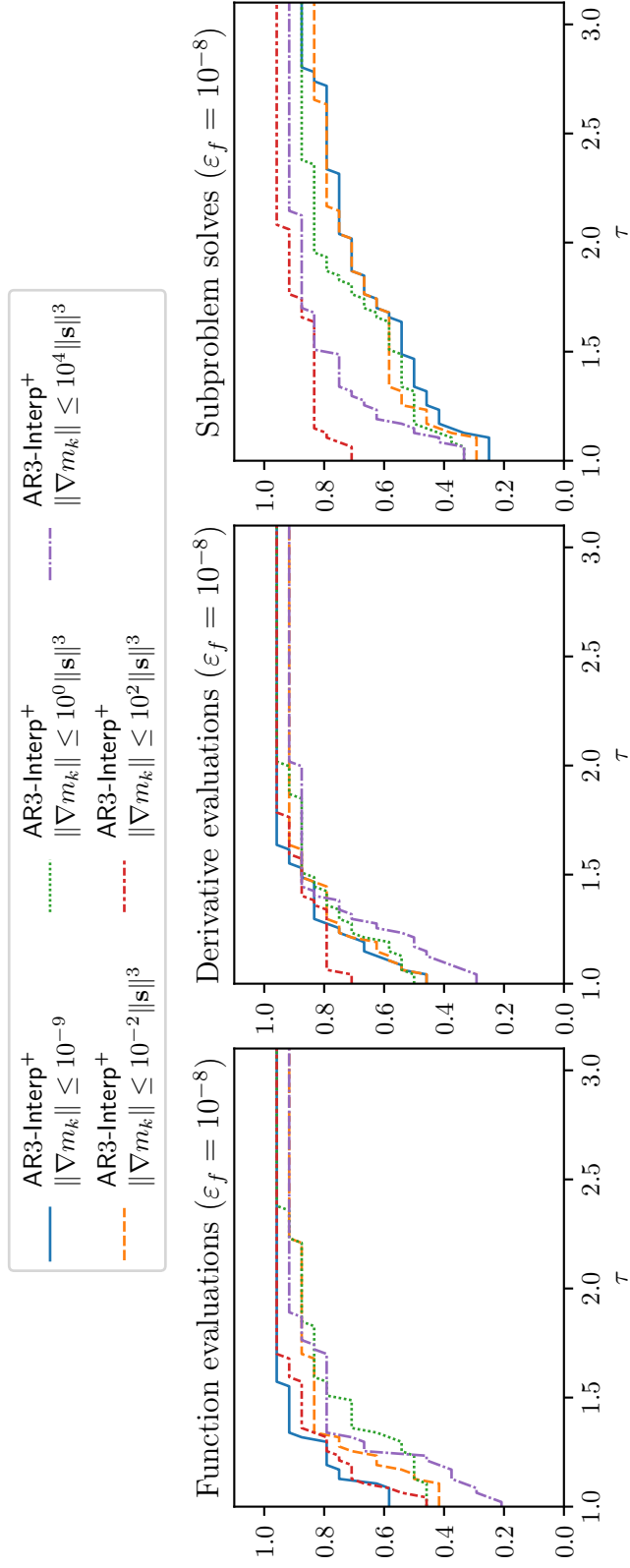}
	\caption{
    The performance profile plot shows the differences between different subproblem termination conditions, namely \cref{eqn:subproblem_absolute_tc} with $\eps_{\textrm{sub}} = 10^{-9}$ and \cref{eqn:subproblem_relative_tc} with $\theta \in \{10^{-2}, 10^{0}, 10^{2}, 10^{4}\}$.
    For all metrics, $\theta = 100$ is among the best performing variants.
    See \cref{sec:performance_profile} for details on the performance profile plots.
    }
    \label{fig:theta-performance-profile}
\end{sidewaysfigure}

As shown in \cref{fig:theta-convergence}, when $\theta$ becomes large enough, such as $\theta = 100$, \textsf{AR$3$-Interp\textsuperscript{+}} converges without any unsuccessful iterations for MGH5 and MGH13.
However, for $\theta = 10^4$, the early termination of the subproblem solver leads to slower progress and additional iterations for MGH13.
This suggests that the optimal parameter value is around $\theta=100$ in these tests.
In general, the optimal choice of $\theta$ will depend on both the test set and the subproblem solver and its parameters, since the solvers' path towards the model minimizer becomes relevant for less stringent termination conditions.

We observed that for larger $\theta$ values, our pre-rejection mechanism was rarely activated and did neither reduce nor increase the number of function evaluations significantly.
However, since the optimal $\theta$ may not be known in advance, it is still advisable to keep pre-rejection enabled to avoid unnecessary unsuccessful function evaluations where possible, as illustrated in the left three columns of \cref{fig:theta-convergence}.

The conclusions from the convergence dot plots also hold for the other functions in our test set.
Comparing \cref{eqn:subproblem_relative_tc} and \cref{eqn:subproblem_absolute_tc}, \cref{fig:theta-performance-profile} shows that the number of derivative evaluations does not change significantly, while the number of subproblem solves decreases, indicating that a higher percentage of iterations are successful.
The added inexactness allows the subproblem solver to terminate earlier and select a point which may not be close to a stationary point of the model but nonetheless leads to a successful iteration.
Again, the value of $\theta = 100$ consistently performs  best for \cref{eqn:subproblem_relative_tc} on function evaluations, derivative evaluations, and subproblem solves and is competitive or outperforms \cref{eqn:subproblem_absolute_tc} with $\eps_{\textrm{sub}} = 10^{-9}$.
Therefore, we view \textsf{AR$3$-Interp\textsuperscript{+}} using \cref{eqn:subproblem_relative_tc} and $\theta = 100$ as the most successful of all variants considered so far and will benchmark it in the following section.

\section{Benchmarking}
\label{sec:benchmark}

After determining the best-performing AR$3$ variant overall, we need to compare it to state-of-the-art AR$p$ variants.
Firstly, we benchmark against the second-order variant  \textsf{AR$2$-Interp}.
For \(p = 2\), \cref{alg:interpolation_sigma_update_full} matches the updating strategy for the regularization parameter proposed in \cite{gould2012updating}, which was found to outperform simpler updating rules.
We do not use \textsf{AR$2$-Interp\textsuperscript{+}} because the addition of pre-rejection would not make a difference.
As explained in \cref{sec:implementation_details}, we use a factorization-based solver which finds a high-accuracy global minimizer of the AR$2$ subproblems, and this global minimizer is always persistent by \cref{thm:ar2_global_min_persistent} (and so
the step \(\ss_k\) computed in this way would not be pre-rejected).
The best-performing termination condition for \textsf{AR$2$-Interp} was determined to be \cref{eqn:subproblem_relative_tc} with $\theta = 0.01$ (see \cref{sec:inexactness_AR2}).

Secondly, we compare our best-performing AR$3$ variant to the AR$3$ algorithm proposed in \cite{birgin2020use}, which we will refer to as \textsf{AR$3$-BGMS/Gencan} as it uses \cref{alg:bgms_sigma_update} to update the regularization parameter and Gencan \cite{birgin2001box,birgin2002large,andretta2005practical} as the subproblem solver.\footnote{Their implementation can be found at \url{https://github.com/johngardenghi/ar4} and \url{https://www.ime.usp.br/~egbirgin/sources/bgms/code}.}
Lastly, we also include \textsf{AR$3$-BGMS}, which is our implementation of the same updating strategy paired with our standard subproblem solver, to give some context on the differences caused by the subproblem solver and other implementation details.
For both BGMS variants we report their results using \cref{eqn:subproblem_relative_tc} with $\theta = 100$ as this is the subproblem termination condition used in \cite{birgin2020use}.

We included all $35$ problems from the MGH test set in the following performance profile plot, while our earlier numerical studies  were performed only on a strict subset of these test functions (including only half of the MGH test set and a few other problems) in order to avoid tuning. As mentioned above, our best performing variant has been selected as \textsf{AR$3$-Interp\textsuperscript{+}} with \cref{eqn:subproblem_relative_tc} and $\theta = 100$.

\begin{sidewaysfigure}
	\centering
	\subimport{plots/benchmark}{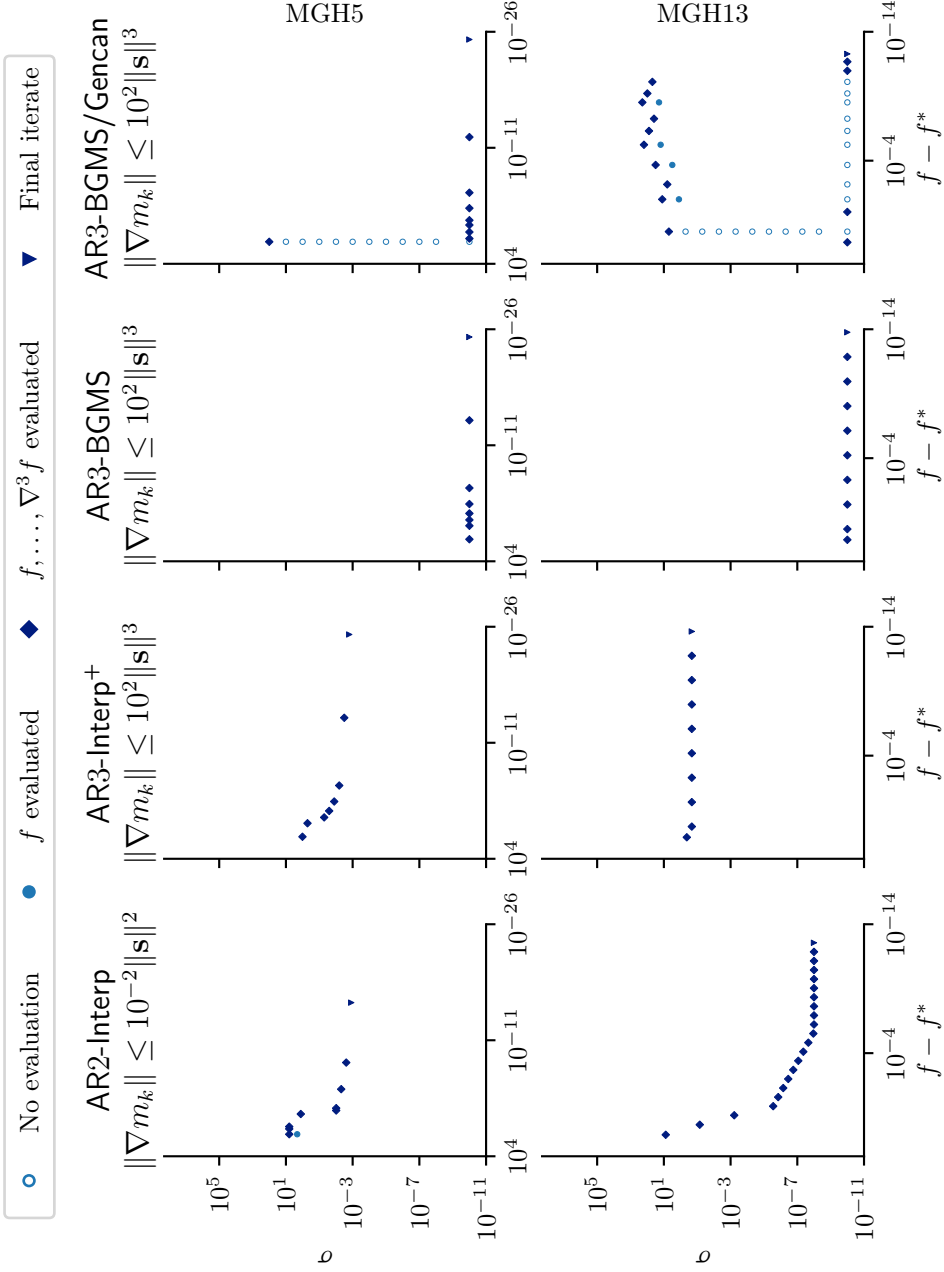}
	\caption{
    The convergence dot plot compares \textsf{AR$2$-Interp}, \textsf{AR$3$-Interp\textsuperscript{+}}, \textsf{AR$3$-BGMS} and \textsf{AR$3$-BGMS/Gencan} using \cref{eqn:subproblem_relative_tc}.
    The second-order method shows slower asymptotic convergence for both MGH5 and MGH13. There are clear differences between \textsf{AR$3$-BGMS} and \textsf{AR$3$-BGMS/Gencan} caused by differing solvers for the AR$3$ subproblems.
    See \cref{sec:convergence_dot} for details on the convergence dot plots.
    }
    \label{fig:benchmark-convergence}
\end{sidewaysfigure}

\begin{sidewaysfigure}
	\centering
	\subimport{plots/benchmark}{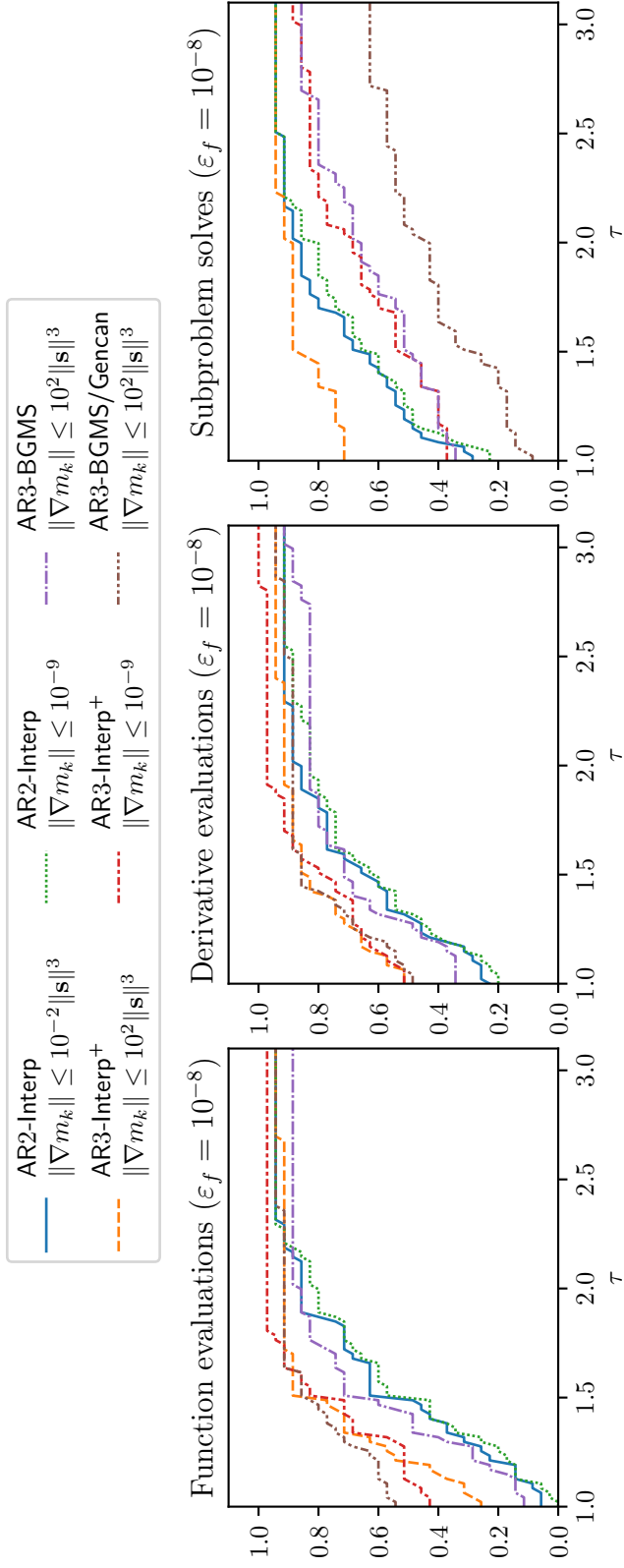}
	\caption{
    The performance profile plot compares \textsf{AR$2$-Interp}, \textsf{AR$3$-Interp\textsuperscript{+}}, \textsf{AR$3$-BGMS} and \textsf{AR$3$-BGMS/Gencan} using \cref{eqn:subproblem_absolute_tc,eqn:subproblem_relative_tc} on the full $35$ MGH test problems.
    \textsf{AR$3$-Interp\textsuperscript{+}} using the relative subproblem termination condition clearly outperforms the \textsf{AR$2$-Interp} variants on all metrics and is competitive with \textsf{AR$3$-BGMS/Gencan} in terms of function and derivative evaluations, while using far fewer subproblem solves.
    See \cref{sec:performance_profile} for details on the performance profile plots.}
	\label{fig:benchmark-performance-profile}
\end{sidewaysfigure}

As an illustration, the convergence dot plot in \cref{fig:benchmark-convergence} clearly shows the difference in asymptotic convergence between AR$2$ and AR$3$.
In the case of MGH5, both \textsf{AR$2$-Interp} and \textsf{AR$3$-Interp\textsuperscript{+}} exhibit superlinear local convergence, but the acceleration of convergence from one step to the next is greater for AR$3$.
Similarly, for MGH13 both methods have linear convergence of the function value to its optimum, but the rate is better for AR$3$, even when comparing to the slightly slower convergence of \textsf{AR$3$-Interp\textsuperscript{+}} with \cref{eqn:subproblem_absolute_tc} shown in \cref{fig:theta-convergence}.
The graph in \cref{fig:benchmark-convergence} also shows how the solver used for the AR$3$ subproblem impacts  the number of iterations and the convergence speed, given that the left two columns differ quite significantly, especially for MGH13.

Turning to the performance profile plots in \cref{fig:benchmark-performance-profile}, we see that in terms of function and derivative evaluations both \textsf{AR$2$-Interp} variants (using \cref{eqn:subproblem_absolute_tc} and \cref{eqn:subproblem_relative_tc}) are outperformed by the AR$3$ methods.
This agrees with the numerical findings in \cite{birgin2020use}, as well as the global complexity results for AR$p$ \cite{birgin2017worst,cartis2020concise,cartis2020sharp} that third-order methods can exploit  access to third derivatives to speed up the optimization calculations.
Within the AR$3$ methods, \textsf{AR$3$-BGMS/Gencan} generally uses less function and derivative evaluations but more subproblem solves than \textsf{AR$3$-BGMS}, which highlights the need to consider the subproblem solver as a central component of AR$p$ algorithms, especially when allowing for approximate solutions.
Overall, the combination of modifications to the basic AR$3$ algorithm suggested in this paper (\textsf{AR$3$-Interp\textsuperscript{+}}) is among the best-performing variants for function evaluations, derivative evaluations, and subproblem solves and therefore an efficient AR$3$ implementation.

\section{Hessian- and tensor-free implementation}
\label{sec:tensor_free}

As explained in the introduction, in the previous experiments we use \textsf{MCM} \cite[Algorithm 6.1]{cartis2011adaptiveI} as the subproblem solver for \textsf{AR$2$-Simple}, which in turn is used to solve the AR$3$ subproblems.
This choice restricts us to explicit Hessian evaluations and evaluations of third derivatives (provided either explicitly or as tensor-vector product function handles).
In the case where the user provides the third derivative through tensor-vector products, our algorithm is tensor-free in the sense that it is not necessary to store a possibly dense tensor, bringing down the memory requirements from $\mathcal{O}(d^3)$ to $\mathcal{O}(d^2)$, but at the cost of requiring more than one oracle call per successful iteration.
As the dimension $d$ grows, even needing $\mathcal{O}(d^2)$ memory can be prohibitive.
Moreover, direct factorizations of dense Hessians used by \textsf{MCM} become impractical, while higher-order methods such as AR$3$ remain attractive in principle.
To address this, we implemented \textsf{GLRT} \cite[Section 10]{cartis2022evaluation} within our package, a Krylov subspace-based iterative method for solving cubic regularization subproblems.
It only requires matrix-vector products of the matrix involved in the AR$2$ subproblem, which means it suffices to provide the Hessian through matrix-vector products and the third derivative through tensor-vector-vector products, making the overall algorithm both Hessian- and tensor-free.
In this setting, the memory requirement is dominated by storing a small number of Krylov basis vectors and auxiliary work arrays, that is, by $\mathcal{O}(d)$ rather than by $\mathcal{O}(d^2)$ or $\mathcal{O}(d^3)$.\footnote{In \textsf{GLRT}, the Krylov subspace basis must be stored explicitly as a matrix of size $d \times t$, where $t \leq d$ is the Krylov subspace iteration counter. Empirical evidence reported in \cite[Appendix A]{carmon2018analysis} suggests that, in practice, one typically has $t \leq 20$ for GLRT. Therefore, we omit $t$ in the memory cost.}
\Cref{tab:derivate_oracle_options} summarizes the different formats in which second- and third-order information can be provided to the algorithm, and which of the inner-inner solvers we implemented are compatible with each of those choices.

\begin{table}
    \caption{Types of derivative oracles supported by our implementation}
    \centering
    \begin{tabular}{lcc}
        \toprule
         & Evaluations & Compatible solvers \\
        \midrule
        Explicit & $\nabla^2 f(\xx)$, $\nabla^3 f(\xx)$ & \textsf{MCM}, \textsf{GLRT} \\
        Tensor-free & $\nabla^2 f(\xx)$, $\vv \mapsto \nabla^3 f(\xx)[\vv]$ & \textsf{MCM}, \textsf{GLRT} \\
        Hessian- \& tensor-free & $\vv \mapsto \nabla^2 f(\xx) \vv$, $(\vv, \ww) \mapsto \nabla^3 f(\xx)[\vv, \ww]$ & \textsf{GLRT} \\
        \bottomrule
    \end{tabular}
    \label{tab:derivate_oracle_options}
\end{table}

\begin{sidewaysfigure}
	\centering
	\subimport{plots/GLRT}{performance_profile-p23-tau-f-one-line.pgf}
	\caption{
		The performance profile plot shows the impact of using the subproblem solvers \textsf{MCM} and \textsf{GLRT} for solving the AR$2$ subproblems on the full 35 MGH test problems. The AR$2$ subproblem solvers are employed directly for AR$2$ and within \textsf{AR$2$-Simple} for AR$3$. For \textsf{AR$3$-Interp\textsuperscript{+}}, $\tilde{m}_k$ refers to the AR$2$ model constructed while minimizing the AR$3$ model. See \cref{sec:performance_profile} for details on the performance profile plots.
    }
	\label{fig:GLRT-performance-profile}
\end{sidewaysfigure}

To illustrate the correctness of our \textsf{GLRT} implementation and the impact on the usual metrics of function and derivative evaluations as well subproblem solves, we compare \textsf{MCM} and \textsf{GLRT} as subproblem solvers for AR$2$ and as subsubproblem solvers for AR$3$ in \cref{fig:GLRT-performance-profile}.
The AR$3$ experiments use \textsf{AR$3$-Interp\textsuperscript{+}} with the default subproblem solver \textsf{AR$2$-Simple} and \cref{eqn:subproblem_relative_tc} with $\theta = 100$.
Only the subproblem solver inside \textsf{AR$2$-Simple} and the accuracy to which its subproblem is solved varies.
The termination condition for the AR$2$ subproblem depends on the gradient of the AR$2$ model, which is denoted as $\tilde{m}_k$ in the legend of \cref{fig:GLRT-performance-profile}.
Note that for these experiments the algorithm does not operate in a Hessian- or tensor-free way even when using \textsf{GLRT}.
Otherwise, the different derivative oracles would make the comparison of the number of derivative evaluations significantly more complicated.
One can see that the overall performance stays essentially the same when solving the AR$2$ subproblems to the same absolute tolerance.
Only when allowing more inaccurate solves the subproblem solutions differ significantly and more iterations are necessary when using \textsf{GLRT}.
Therefore, we show that \textsf{GLRT} is able to provide similarly accurate solutions to the subsubproblems inside AR$3$ and can be used to apply third-order algorithms to large-scale problems.

\begin{sidewaysfigure}
    \centering
    \subimport{plots/GLRT}{rosenbrock_time_hvp_chol_vs_dim_log2_innertheta10.pgf}

    \caption{
    The leftmost plot shows the ``wall-clock'' time of using AR$1$, AR$2$ and AR$3$ on the multidimensional Rosenbrock problem as the dimension grows. For AR$2$ and AR$3$ the subproblem solvers \textsf{MCM} (with explicit Hessians and third derivatives) and \textsf{GLRT} (Hessian- and tensor-free) are used for solving the AR$2$ subproblems. They are employed directly for AR$2$ and within \textsf{AR$2$-Simple} for AR$3$. The plots in the center and on the right show the corresponding total number of Hessian–vector products for \textsf{GLRT} and the total number of Cholesky factorizations for \textsf{MCM}. Note that the total number of tensor–vector–vector products is the same as the total number of Hessian–vector products for AR$3$. A special marker denotes that the running time exceeded $4$ hours and the corresponding algorithm failed to solve the problem. These points are connect with a partially transparent line, since their slope is not indicative of the real growth.
    }
    \label{fig:GLRT-time}
\end{sidewaysfigure}

For readers who are interested in the performance of our algorithm in the Hessian- and tensor-free setting, we provide an additional experiment using the multidimensional Rosenbrock function (see \cref{sec:multidimensional_rosenbrock} for its definition) as an example.
We investigate the scaling of wall-clock time and of metrics related to the computational effort of solving the subproblems as the dimension grows from $2^2$ to $2^{14}$ in \cref{fig:GLRT-time}.
The left plot of this figure presents a running time comparison between AR$1$, AR$2$ and AR$3$, where AR$2$ and AR$3$ are combined with \textsf{MCM} (explicit Hessian and tensor) and \textsf{GLRT} (Hessian- and tensor-free).
AR$1$ does not require an inner solver.
Since the dimension can be large, we adapt the parameter settings for this experiment as follows.
We terminate the algorithm once $||\nabla f(\xx_k) || \leq 10^{-6}$ is achieved. We replaced the $1,000$ iteration limit with a time limit of $4$ hours for the main algorithm, but kept it for the inner and inner-inner solvers.
All the best function values for the solved problems in both \cref{fig:GLRT-time} are less than $10^{-11}$, indicating convergence to the global minimizer.
For both AR$2$ and AR$3$, we use \cref{eqn:subproblem_relative_tc} as the inner termination condition with $\theta=10$.
At the inner-inner level of AR$3$, we allow for a larger inexactness.
For the AR$2$ model $\tilde{m}_k$ constructed from the derivatives of the AR$3$ model, the AR$2$ subproblem solver terminates when $|| \nabla \tilde{m}_k(\tilde{\ss}_k) || \leq \tilde{\theta} || \tilde{\ss_k} ||^2$ for $\tilde{\theta} = 10^{\log_2(d)/2}$.\footnote{Note that these parameter values were chosen through trial and error. We observed that a large Krylov subspace was occasionally necessary inside \textsf{GLRT} to solve the problems with $\theta=10$, but at the same time for too large $\theta$ the number of outer iterations became very high. For an additional experiment with a $d$-dependent $\theta$, see \cref{fig:GLRT-time-d-theta} in the appendix. These empirically chosen values were sufficient for robust performance in our tests and were not further fine-tuned, as parameter optimization is not the focus of this work.}

For experiments with dimensions $2^{12}$ and $2^{14}$, AR$3$ with \textsf{MCM} and explicit third derivatives exceeded the available memory on our machine, since third-order tensors of corresponding dimensions need roughly $512$GB and $32$TB of memory in double precision.
At the same time, AR$3$ with \textsf{GLRT} is able to solve these two large dimensional problems within the given time limit, unlike the AR$2$ variants.
As running time results can vary based on different hardware or implementations, we also provide plots of the total number of Hessian–vector products (which is equal to the total number of tensor–vector–vector products for AR$3$) for methods using \textsf{GLRT}, and the total number of Cholesky factorizations for methods using \textsf{MCM} in \cref{fig:GLRT-time}.
Overall, the results indicate that for both AR$2$ and AR$3$ using \textsf{MCM} the number of Cholesky decompositions grow roughly linear with dimension and the corresponding timings grow roughly like $d^{3}$.
For AR$3$ using \textsf{GLRT} we see that the increase in Hessian-vector products (or equivalently tensor-vector-vector products) is linear with respect to the dimension, whereas the line for AR$2$ combined with \textsf{GLRT} has a steeper slope corresponding to a growth of $d^2$, which explains why AR$3$ is able to converge faster than AR$2$ for larger dimensions.
For the multidimensional Rosenbrock function, even with the additional cost of the tensor-vector-vector products, the Hessian- and tensor-free third-order algorithm outperforms the second-order one on large-scale problems.


We still emphasize, however, that the detailed choice and tuning of subproblem solvers is not the main focus of this work, and the comparison involving \textsf{GLRT} is mainly included to address concerns about storage and runtime costs for third-order methods.
A more systematic runtime-oriented study of inner and inner-inner solvers is left for future research.

\section{Conclusions}

In this paper, we introduced several practical variants within the general AR$p$ algorithmic framework for nonconvex unconstrained optimization problems. Our approach incorporates an adaptive interpolation-based regularization parameter updating strategy and a pre-rejection step control strategy. Specifically, we proposed an inexpensive heuristic for choosing a reasonable $\sigma_0$, and extended the interpolation-based updating strategy in \cite{gould2012updating} to cases where $p \geq 3$. Numerical results demonstrate the effectiveness and robustness of the proposed methods.
Next, we investigated the fundamental differences in AR$p$ subproblem solutions when $p \geq 3$ (compared to second-order variants) and introduced a novel pre-rejection module that efficiently rejects directionally transient minimizers of the subproblems. 
Furthermore, we examined the performance differences between absolute and relative termination conditions in the subproblems.
Numerical results show that the proposed pre-rejection framework significantly improves efficiency and provides a reliable guide on predicting steps leading to unsuccessful iterations before any function evaluations, with both absolute and relative termination conditions.
We then combined the best of our proposed strategies and conducted a benchmark experiment, which confirmed that our AR$3$ variants are more efficient and competitive than the AR$2$ variants in both function and derivative evaluations, while also offering improvements in implementation efficiency.
Lastly, addressing concerns about runtime and memory requirements for third-order algorithms, we detail how AR$3$ can be used in a Hessian- and tensor-free way and investigate the scaling of computational costs with dimension on the example of the multidimensional Rosenbrock function.
Additional detailed experimental results for each section are provided in the appendix.

Avoiding the computational burden of computing the exact third-order tensor within AR$3$ methods is an active field of investigation with the potential to improve the efficiency of these implementations. One such approach is to approximate the tensor term by simpler terms, such as in \cite{zhu2023cubic}, which uses a scalar approximation while solving the AR$3$ subproblem. Alternatively, in a similar vein with quasi-Newton methods, \cite{welzel2024approximating} uses $(p-1)$th-order information to approximate the $p$th-order tensor using secant-type formulas and minimal Frobenius norm updates. We delegate the numerical implementation of such extensions within our AR$3$ framework to future work.

\section*{Acknowledgements}
We thank the authors of \cite{birgin2020use}, for making their implementation of their AR$3$ code and of the MGH test set publicly available. We are also grateful to Nick Gould for his valuable comments and assistance with the MATLAB---Fortran interface. Finally, we extend our thanks to the three anonymous referees, including the technical referee, for their insightful questions and suggestions, which helped us to improve the paper.

\printbibliography

\appendix

\section{Test functions}
\label{sec:functions}

In the following, We describe the test functions we used to enhance the MGH test set, as well as the slalom and hairpin turn functions used in \cref{fig:slalom_hairpin}.
For all but the slalom and hairpin turn function we provide two implementations, one returning an explicit matrix for the Hessian and either an explicit tensor (explicit) or a tensor-vector function handle (tensor-free) for the third derivative and one returning matrix-vector and tensor-vector-vector function handles (Hessian-free).
As described in \cref{sec:implementation_details}, the latter only works in conjunction with \textsf{GLRT} as the inner-inner solver for AR$3$.
A summary of the test functions and their memory requirements is given in \cref{tab:storage_formats}.

\begin{table}
\centering
\caption{Memory requirements for evaluating Hessians and third derivatives (or their actions) for our implementations of the test functions used in this paper. A dash represents that we do not provide the corresponding implementation in our code.}
\label{tab:storage_formats}
\begin{tabular}{lccc}
\toprule
 & Explicit & Tensor-free & Hessian- \& tensor-free \\
\midrule
MGH test problems \cite{more1981testing}
  & $\mathcal{O}(d^2)$ / $\mathcal{O}(d^3)$
  & ---
  & --- \\
Reg. third-order polynomial (\ref{sec:reg_third_order})
  & ---
  & $\mathcal{O}(d^2)$ / $\mathcal{O}(d^2)$
  & $\mathcal{O}(d^2)$ / $\mathcal{O}(d^2)$\footnotemark \\
Multidimensional Rosenbrock (\ref{sec:multidimensional_rosenbrock})
  & $\mathcal{O}(d^2)$ / $\mathcal{O}(d^3)$
  & ---
  & $\mathcal{O}(d)$ / $\mathcal{O}(d)$ \\
Nonlinear least-squares (\ref{sec:nonlinear_least_squares})
  & $\mathcal{O}(d^2)$ / $\mathcal{O}(d^3)$
  & ---
  & $\mathcal{O}(d)$ / $\mathcal{O}(d)$\\
Slalom and hairpin turn (\ref{sec:slalom_hairpin})
  & $\mathcal{O}(d^2)$ / $\mathcal{O}(d^3)$
  & ---
  & --- \\
\bottomrule
\end{tabular}
\end{table}

\footnotetext{Unitary matrices of size $d \times d$ are stored according to our construction in \cref{sec:reg_third_order}.}

\subsection{Regularized third-order polynomials}
\label{sec:reg_third_order}
To evaluate the AR3 method on functions with a significant third-order component, we consider functions of the following form:
\begin{align}
	\label{eqn:regularized_cubic}
	f(\xx) = \bb^\T \xx + \frac{1}{2} \xx^\T \HH \xx + \frac{1}{6} \tns{T} [\xx]^3 + \norm{\xx}^c.
\end{align}
This formulation allows us to control the conditioning of both the Hessian and the third-order tensor term.
In our experiment, we set the regularization exponent to $ c=8 $ and the dimension to $d=100$.
The vector $ \bb \in \reals^d $ is randomly generated with entries drawn from a uniform distribution $\mathcal{U}[0,1]$.
For the construction of $\HH \in \reals^{d \times d}$ and $\tns{T} \in \reals^{d \times d \times d}$, we first generate unitary matrices $\UU=[\uu_1, \ldots \uu_d] \in \reals^{d \times d}$ and $\VV=[\vv_1, \ldots \vv_d] \in \reals^{d \times d}$, which are obtained from SVD decompositions of random matrices with entries independently drawn from the standard normal distribution $\mathcal{N}[0,1]$.
We then construct $\HH = \sum_{i=1}^{d} \lambda_i^{\HH} \uu_{i} \otimes \uu_{i} $, and $\tns{T} =\sum_{i=1}^{n} \lambda_i^{\tns{T}} \vv_{i} \otimes \vv_{i} \otimes \vv_{i} $ using the following settings for $1 \leq i \leq d$,
\begin{itemize}
	\item Both $\lambda_i^{\HH}$ and $\lambda_i^{\tns{T}}$ take $d$ values from \texttt{linspace(1,10,d)}, starting from $1$ and ending at $10$, with equal spacing.
	\item $\lambda_i^{\HH}$ values are taken from \texttt{logspace(-6,6,d)}, i.e., from $10^{-6}$ to $10^{6}$, equally spaced on a logarithmic scale.
	The order is then randomly permuted, and the values are multiplied by a random sign. $\lambda_i^{\tns{T}}$ values are taken from \texttt{linspace(1,10,d)} as before.
	\item $\lambda_i^{\HH}$ values are taken from \texttt{linspace(1,10,d)}. $\lambda_i^{\tns{T}}$ values are taken from \texttt{logspace(-6,2,d)}, after a random permutation and random sign changes.
	\item $\lambda_i^{\HH}$ and $\lambda_i^{\tns{T}}$ values are taken from \texttt{logspace(-6,6,d)} and \texttt{logspace(-6,2,d)} respectively, with random permutations and random sign changes applied to both.
\end{itemize}
With these settings, we obtain four test problems, including three that are ill-conditioned.
The default starting point for all four test problems is $\xx_0 = (1, \dots, 1)^\T$.

\subsection{Multidimensional Rosenbrock}\label{sec:multidimensional_rosenbrock}
The multidimensional Rosenbrock function, also known as (Rosenbrock's) valley or banana function, is defined as
\begin{align}
	\label{eqn:rosenbrock}
	f(\xx) = \sum_{i=1}^{d-1} \left[100(x_{i}^2 - x_{i+1})^2 + (x_{i}-1)^2\right]
\end{align}
for any dimension $d \geq 2$.
The global minimum is attained at $\xx^* = (1, \dots, 1)^\T$, where $f(\xx^*) = 0$.
The default starting point is $\xx_0 = (0, \dots, 0)^\T$.

\subsection{Nonlinear least-squares}\label{sec:nonlinear_least_squares}
We consider a binary classification problem using a nonlinear least-squares function~\cite{liu2022newtonmr} of the form
\begin{align}
	\label{eqn:NLS}
	f(\xx) = \frac{1}{n} \sum_{i=1}^{n}  \left(\psi(\aa_i, \xx) - b_i \right)^2, \quad \psi(\aa_i, \xx) = \frac{1}{1+e^{-\dotprod{\aa_i, \xx}}}
\end{align}
where the dataset $ \{\aa_i, b_i\}_{i=1}^n $ consists of feature vectors $ \aa_i \in \mathbb{R}^d, \; i=1,\ldots,n $ and the corresponding labels $ b_i \in \{0, 1\}, \; i=1,\ldots,n $.
We randomly generate $\aa_i$ from a uniform distribution $\mathcal{U}[0,1]$, and $b_i$ are set to zero or one with equal probability.
The default starting point is $\xx_0 = (0, \dots, 0)^\T$.

\subsection{Slalom and hairpin turn functions}
\label{sec:slalom_hairpin}
These two functions were constructed such that knowledge of the third derivative is crucial for fast convergence.
The idea is that in the direction of rapid decrease, the function behaves like a saddle point, where the first and second derivative are zero, but the third derivative is not.
In the other direction, the function only decreases very slowly.
Moreover, we want first- and second-order methods that follow the slow decrease direction to eventually converge to the same point as the fast third-order method. To achieve this, the function must compel the lower-order methods to take a large detour instead of the direct path that the third-order method \enquote*{sees}.
To accomplish these objectives, we define a piecewise function using a few helper functions.

Consider polynomials $q \colon \reals \to \reals$ that satisfy the following conditions:
\begin{equation}\label{eqn:slalom_helper_poly}
	q(0) = 0, \ q(1) = 1, \ q'(0) = q'(1) = 0, \ q''(0) = q''(1) = 0, \ q'''(0) = q'''(1) = z.
\end{equation}
The lowest-order family of such polynomials is given by
\begin{equation}
	q(x) = -20 x^7 + 70x^6 - 84x^5 + 35x^4 + \frac{z}{6} \paren{2x^7 - 7x^6 + 9x^5 - 5x^4 + x^3}.
\end{equation}
We define $h_1(x) = 6x^5 - 15x^4 + 10x^3$, which is the solution for $z=60$.
By \cref{eqn:slalom_helper_poly}, we can define a piecewise function that is three times continuously differentiable as
\begin{equation}
	h_2(x) = h_1(x + 0.5 - \floor{x + 0.5}) + \floor{x + 0.5} - 0.5,
\end{equation}
which is the component in the direction of fast decrease.
Whenever $x = n + 0.5$ for some integer $n$, the function has a saddle point with a nonzero third derivative.

Next we need to address the other direction.
We want to scale $h_2$ by another function such that first- and second-order methods remain at the saddle point in the $x$ direction and instead move along the $y$ direction.
We define $h_3(x) = -20 x^7 + 70x^6 - 84x^5 + 35x^4$, which is the solution of \cref{eqn:slalom_helper_poly} for $z = 0$.
Again, by \cref{eqn:slalom_helper_poly}, the following piecewise function is three times continuously differentiable:
\begin{equation}
	h_4(x, y) = \begin{cases}
		-h_3(2x + 2) (1 - y) + 1 & -1 \leq x < -0.5,     \\
		y                        & -0.5 \leq x \leq 0.5, \\
		h_3(2x - 1) (1 - y) + y  & 0.5 < x \leq 1.
	\end{cases}
\end{equation}
This function continuously interpolates between $h_4(x, y) = y$ for $x \in [-0.5, 0.5]$ and $h_4(x, y) = 1$ for $x \in \{-1, 1\}$.

The two components can then be combined into
\begin{equation}\label{eqn:slalom_step_function}
	g(x, y) = h_2(x) h_4 \paren*{x, \frac{2}{1 + e^y}},
\end{equation}
which serves as the main building block for this construction.
The region where $(x, y) \in [-1, 1] \times [0, \infty)$ resembles a hairpin turn in the mountains, except that the actual turn is stretched out to infinity.
As $y$ increases, the \enquote*{streets} at $x = 0.5$ and $x = -0.5$ level out at zero, with one decreasing from $0.5$ and the other one increasing from $-0.5$.

To ensure that the minimizer of this function is close to $(-0.5, 0)$, we bend the function upwards with a fourth-order barrier term.
Additionally, we introduce a small slope $r > 0$ in the $x$ direction so that instead of terminating, our algorithm follows the slight downward slope on the flat part of the function and finds the minimizer. In the plot shown in \cref{fig:slalom_hairpin}, the value of $r$ is $3 * 10^{-4}$.
This results in what we refer to as the \enquote*{hairpin turn} function.
\begin{equation}
	b(x, x_{\min}, x_{\max}) = \begin{cases}
		(x - x_{\min})^4 & x \leq x_{\min},         \\
		0                & x_{\min} < x < x_{\max}, \\
		(x - x_{\max})^4 & x \geq x_{\max},
	\end{cases}
\end{equation}
\begin{equation}
	\label{eqn:hairpin_function}
	f_{\text{hairpin}}(x, y) = g(x, y) + r x + 50 b(x, -0.4, 0.5) + 50 b(y, 0, 5).
\end{equation}

Next, let us describe the \enquote*{slalom} function.
It is constructed by stitching together copies of the function in \cref{eqn:slalom_step_function} together at $y = 0$ and periodically in the $x$ direction, ensuring that the new function is continuous and that the downhill slopes on one side match with those on the other. The slalom function is defined as:
\begin{equation}\label{eqn:slalom_function}
	f_{\text{slalom}}(x, y) = rx + \begin{cases}
		g(x + 1 - 2\floor{\frac{x+1}{2}}, y) + 2\floor{\frac{x+1}{2}} & y \leq 0 \\
		g(x - 2\floor{\frac{x}{2}} - 1, y) + 2\floor{\frac{x}{2}} + 1 & y \geq 0
	\end{cases}
\end{equation}
Again, a slight slope in the $x$ direction is added to prevent convergence to saddle points.
Note that this function is three times continuously differentiable everywhere except along the $y = 0$ line, where there is a discontinuity in the first derivative.

\section{Theorem Proofs}
\label{sec:theorem_proofs}

\subsection{\texorpdfstring{Proof of \cref{thm:persistent_minimizers_local}}{Proof of Theorem 1}}
\label{sec:thm_prf_persistent_minimizers_local}
\begin{proof}
	Recall that the model has the form $m_{\sigma}(\ss) = t(\ss) + \frac{\sigma}{p+1} \norm{\ss}^{p+1}$, which means any stationary point $\ss$ satisfies
	\begin{equation}\label{eqn:persistent_minimizers_stationary}
		\nabla m_{\sigma}(\ss) = \nabla t(\ss) + \sigma \norm{\ss}^{p-1} \ss = \vek{0}.
	\end{equation}
	Rearranging and taking norm gives
	\begin{equation}
		\sigma = \frac{\norm{\nabla t(\ss)}}{\norm{\ss}^p} \eqqcolon \tilde{\sigma}(\ss)
	\end{equation}
	for $\sigma \geq 0$.
	This defines a function $\tilde{\sigma} \colon \R^d \setminus \{\vek{0}\} \to \R_{\geq 0}$ which satisfies $\tilde{\sigma}(\vek{\psi}(\sigma)) = \sigma$ for all $\sigma \in I$ for any minimizer curve $\vek{\psi}$.

	To show that any persistent minimizer curve converges to $\vek{0}$, we can bound $\tilde{\sigma}$ by
	\begin{equation}
		\tilde{\sigma}(\ss) = \frac{\norm{\nabla t(\ss)}}{\norm{\ss}^p} \leq \max\{C_1 \norm{\ss}^{-1}, C_2 \norm{\ss}^{-p}\}
	\end{equation}
	using the fact that $t$ is a $p$th-order polynomial, which implies that $\norm{\nabla t(\ss)} / \norm{\ss}^{p-1}$ is uniformly upper bounded by some constant $C_1 < \infty$ when $\ss$ is bounded away from zero and $\norm{\nabla t(\ss)}$ is uniformly upper bounded by some constant $C_2 < \infty$ when $\ss$ is close to zero.
	Therefore,
	\begin{equation}
		\norm{\vek{\psi}(\sigma)} \leq \max\{ C_1 \tilde{\sigma}(\vek{\psi}(\sigma))^{-1}, C_2^{1/p} \tilde{\sigma}(\vek{\psi}(\sigma))^{-1/p} \} = \max\{ C_1 \sigma^{-1}, C_2^{1/p} \sigma^{-1/p} \} \to 0
	\end{equation}
	as $\sigma \to \infty$ for a persistent minimizer curve $\vek{\psi}$.

	Clearly, if $\vek{\psi}$ is persistent and converges to $\vek{0}$, then it also satisfies $\inf_{\sigma \in I} \norm{\vek{\psi}(\sigma)} = 0$.
	We only need to show that this latter condition in turn implies persistency.
	The fact that $\norm{\vek{\psi}(\sigma)}$ is not bounded away from zero means there is a sequence $\sigma_1, \sigma_2, \ldots \in I$ such that $\norm{\vek{\psi}(\sigma_k)} \to 0$ as $k \to \infty$.
	We have
	\begin{equation}
		\sigma_k = \tilde{\sigma}(\vek{\psi}(\sigma_k)) = \frac{\norm{\nabla t(\vek{\psi}(\sigma_k))}}{\norm{\vek{\psi}(\sigma_k)}^p} \to \infty
	\end{equation}
	since the numerator converges to $\norm{\nabla t(\vek{0})} > 0$ and the denominator converges to zero.
	Therefore, $\vek{\psi}$ must be persistent.
\end{proof}

\subsection{\texorpdfstring{Proof of \cref{thm:ar2_global_min_persistent}}{Proof of Theorem 2}}
\label{sec:thm_prf_ar2_global_min_persistent}
\begin{proof}
	Without loss of generality, we can assume that
	\begin{equation}
		m_{\sigma}(\ss) = \gg^\T \ss + \frac{1}{2} \ss^\T \HH \ss + \frac{\sigma}{3} \norm{\ss}^3
	\end{equation}
	for some $\gg \in \R^n \setminus \{\vek{0}\}$ and a symmetric $\HH \in \R^{n \times n}$.
	According to \cite[Theorem 8.2.8]{cartis2022evaluation}, any global minimizer $\ss_*$ of this model satisfies
	\begin{equation}\label{eqn:ar2_global_min}
		(\HH + \lambda_* \eye) \ss_* = - \gg
	\end{equation}
	where $\HH + \lambda_* \eye \succeq \mat{0}$ and $\lambda_* = \sigma \norm{\ss_*}$.
	By the positive semidefiniteness of the matrix and the nonnegativity of $\sigma$, we know that $\lambda_* \geq \max\{0, -\lambda_{\min}(\HH)\} \eqqcolon \lambda_s$.

	First, consider the case where $\lambda_* > -\lambda_{\min}(\HH)$.
	Using the fact that the matrix in \cref{eqn:ar2_global_min} is nonsingular and that $\gg \neq \vek{0}$, we can define
	\begin{equation}
		\ss_1(\lambda) = -(\HH + \lambda \eye)^{-1} \gg \quad \text{and} \quad \sigma_1(\lambda) = \lambda / \norm{\ss_1(\lambda)}
	\end{equation}
	for all $\lambda > -\lambda_{\min}(\HH)$.
	Note that according to \cite[Theorem 8.2.3]{cartis2022evaluation}, we know that $\pi(\lambda) = \norm{\ss_1(\lambda)}^{-1}$ has a nonnegative derivative whenever $\lambda > \lambda_s$.
	Therefore,
	\begin{equation}
		\sigma_1'(\lambda) = \pi(\lambda) + \lambda \pi'(\lambda) > 0
	\end{equation}
	which means that $\sigma_1(\lambda)$ is strictly monotonically increasing on $[\lambda_*, \infty)$.
	Moreover, $\sigma_1(\lambda)$ is continuous and converges to infinity as $\lambda$ approaches infinity.
	This implies that there exists a $c \geq 0$ such that $\sigma_1 \colon [\lambda_*, \infty) \to [c, \infty)$ is one-to-one, and that $\vek{\psi}_1$, defined by
	\begin{equation}\label{eqn:ar2_minimizer_curve_not_hard_case}
		\vek{\psi}_1(\sigma_1(\lambda)) = \ss_1(\lambda) \text{ for all } \lambda \geq \lambda_*,
	\end{equation}
	is part of a persistent minimizer curve which contains $\ss_* = \ss_1(\lambda_*)$.

	Next, consider the case where $\lambda_* = -\lambda_{\min}(\HH)$.
	In this case $\HH + \lambda_* \eye$ is singular.
	Let $V$ be the nullspace of $\HH + \lambda_* \eye$ and $V_{\perp}$ its orthogonal complement.
	We can decompose $\ss_*$ into a sum $\vv + \ww$ where $\vv \in V$ and $\ww \in V_{\perp}$.
	Assume for now that $\vv \neq \vek{0}$, the other case will be handled later.
	Since $\gg \neq \vek{0}$, we also have $\ww \neq \vek{0}$, and so
	\begin{equation}
		\ss_2(\alpha) = \alpha \vv + \ww \quad \text{and} \quad \sigma_2(\alpha) = \lambda_* / \norm{\ss_2(\alpha)}
	\end{equation}
	are well-defined.
	Clearly, $\norm{\ss_2(\alpha)}$ is strictly monotonically increasing, and $\sigma_2(\alpha)$ is strictly monotonically decreasing.
	Therefore, there exists a $c \geq 0$ such that $\sigma_2 \colon [0, \infty) \to (0, c]$ is one-to-one, and $\vek{\psi}_2$, defined by
	\begin{equation}
		\vek{\psi}_2(\sigma_2(\alpha)) = \ss_2(\alpha) \text{ for all } \alpha \geq 0,
	\end{equation}
	forms part of a minimizer curve which contains $\ss_* = \ss_2(1)$.

	With the given construction, $\vek{\psi}_2$ is not persistent, because it is only defined on $\sigma \in (0, c]$.
	However, it is possible to extend $\vek{\psi}_2$ by composing it with $\vek{\psi}_1$.
	Note that since $-\lambda_{\min}(\HH) = \lambda_* \geq 0$, the construction of $\vek{\psi}_1$ in \cref{eqn:ar2_minimizer_curve_not_hard_case} extends to all $\lambda > \lambda_s = \lambda_*$, and $\ss_1(\lambda)$ can be equivalently defined as $-(\HH + \lambda \eye)^\dagger \gg$, where $(\cdot)^\dagger$ denotes the pseudoinverse of a matrix.
	Therefore,
	\begin{equation}\label{eqn:ar2_minimizer_curve_limit}
		\lim_{\lambda \to \lambda_*} \ss_1(\lambda) = -(\HH - \lambda_* \eye)^\dagger \gg = \ss_2(0) \quad \text{and} \quad \lim_{\lambda \to \lambda_*} \sigma_1(\lambda) = \lambda_* / \norm{\ss_2(0)} = \sigma_2(0).
	\end{equation}
	This implies that the curve defined by following $\vek{\psi}_2$ for $\sigma \in (0, \sigma_2(0)]$ and $\vek{\psi}_1$ for $(\sigma_2(0), \infty)$ forms (part of) a persistent minimizer curve that contains $\ss_*$.

	To finish the proof, we need to consider the case $\lambda_* = -\lambda_{\min}(\HH)$ and $\vv = \vek{0}$ that was skipped before.
	This case corresponds exactly to $\ss_* = -(\HH - \lambda_* \eye)^\dagger \gg$.
	Therefore, in this scenario, it suffices to extend $\vek{\psi}_1$ by using the limit points of $\sigma_1$ and $\ss_1$ as determined in \cref{eqn:ar2_minimizer_curve_limit}, namely map $\lambda_* / \norm{\ss_*}$ to $\ss_*$.
	The extended curve is continuous as shown above and (part of) a persistent minimizer curve.
\end{proof}

\subsection{\texorpdfstring{Proof of \cref{thm:0_persistent}}{Proof of Corollary 1}}
\label{sec:thm_prf_0_persistent}
\begin{proof}
	In addition to the assumptions in the statement, assume there exists a bounded and $0$-persistent minimizer curve $\vek{\psi}$.
	This implies $\vek{\psi} \colon [0, \infty) \to (0, \bar{\alpha}]$ or $\vek{\psi} \colon (0, \infty) \to (0, \bar{\alpha})$.
	In both cases, we know from \cref{eqn:regularization_first_order} that
	\begin{equation}
		\frac{t'(\bar{\alpha})}{\bar{\alpha}^p} = \lim_{\sigma \to 0} \frac{-t'(\vek{\psi}(\sigma))}{\vek{\psi}(\sigma)^p} = \lim_{\sigma \to 0} \sigma = 0.
	\end{equation}

	Now consider the case where $t'(\bar{\alpha}) = 0$.
	Following the same construction as in the proof of \cref{thm:one_dimensional_persistent_minimizers}, we see that $\lim_{\alpha \to \bar{\alpha}} \sigma(\alpha) = 0$, and therefore, $\vek{\psi}$ is defined for all $\alpha \in (0, \infty)$, making it $0$-persistent.
	The boundedness of $\vek{\psi}$ follows from the existence of $\bar{\alpha}$.
\end{proof}

\section{Detailed performance profile plots}
\label{sec:additional_experiments}
We provide additional performance profile plots for the experiments from each section of the main paper.
For a description of the performance profile plots and the meaning of $\eps_f$ and $\tau$, see \cref{sec:performance_profile}.

\subsection{\texorpdfstring{Choosing the initial regularization parameter $\sigma_0$}{Choosing the initial regularization parameter sigma0}}
\begin{figure}[H]
	\centering
	\subimport{plots/sigma0}{performance_profile-First_Order-tau-f.pgf}
	\caption{Additional performance profile plots are provided to show the impact of using different $\sigma_0$ values with respect to different $\eps_f$.
    }
	\label{fig:sigma0-convergence-full-tau}
\end{figure}
\begin{figure}[H]
	\centering
	\subimport{plots/sigma0}{performance_profile-First_Order-eps-f.pgf}
	\caption{Additional performance profile plots are provided to show the impact of using different $\sigma_0$ values with respect to different $\tau$.
    }
	\label{fig:sigma0-convergence-full-eps}
\end{figure}

\subsection{\texorpdfstring{An interpolation-based update for AR$p$}{An interpolation-based update for ARp}}
\begin{figure}[H]
	\centering
	\subimport{plots/interpolation}{performance_profile-First_Order-tau-f.pgf}
	\caption{The performance profile plot illustrates the effectiveness of the interpolation-based strategy with respect to different $\eps_f$.}
	\label{fig:interpolation-convergence-full-tau}
\end{figure}
\begin{figure}[H]
	\centering
	\subimport{plots/interpolation}{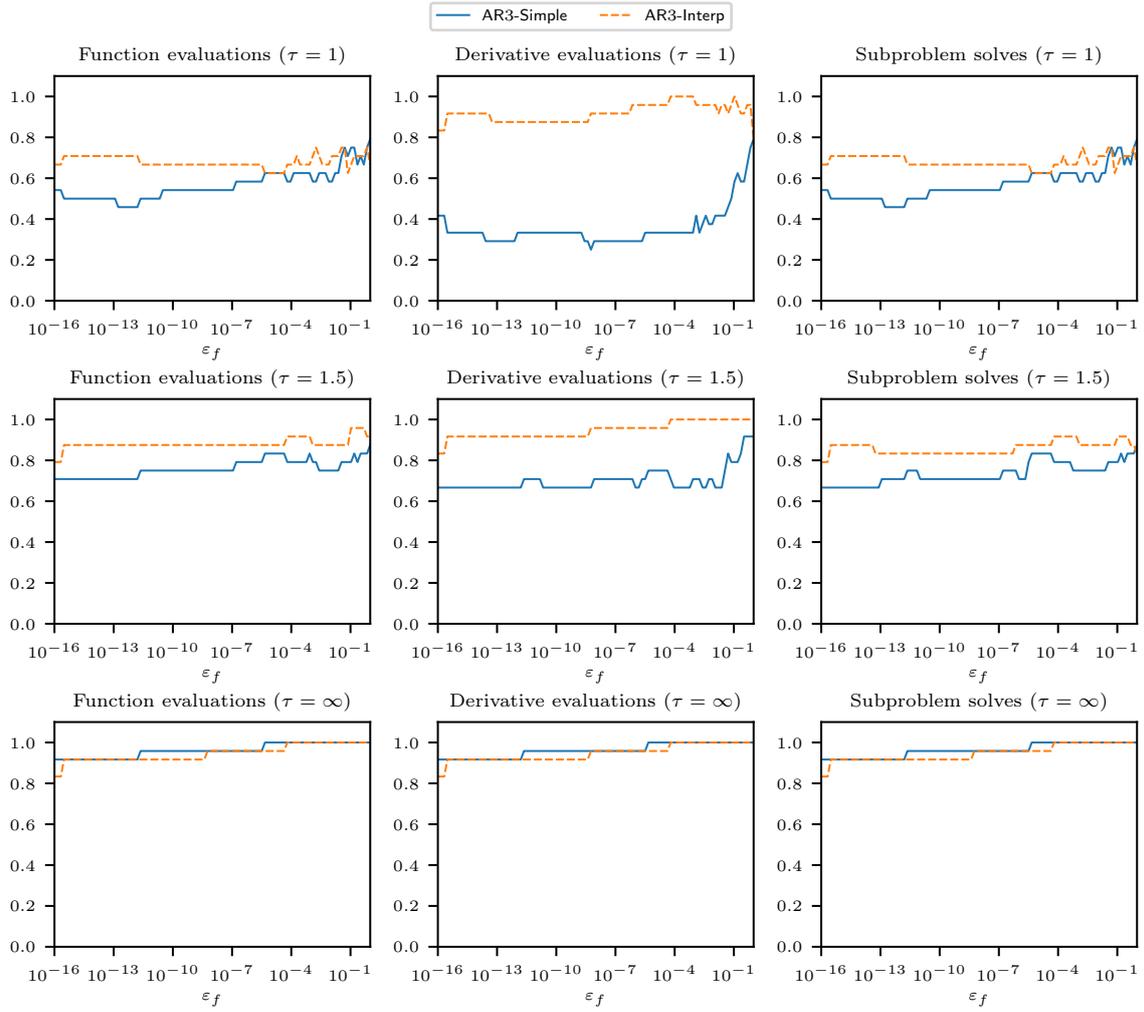}
	\caption{The performance profile plot illustrates the effectiveness of the interpolation-based strategy with respect to different $\tau$.}
	\label{fig:interpolation-convergence-full-eps}
\end{figure}

\subsection{A pre-rejection framework}
\begin{figure}[H]
	\centering
	\subimport{plots/prerejection}{performance_profile-First_Order-tau-f.pgf}
	\caption{The performance profile plot illustrates the impact of including pre-rejection with respect to different $\eps_f$.}
	\label{fig:prerejection-convergence-full-tau}
\end{figure}
\begin{figure}[H]
	\centering
	\subimport{plots/prerejection}{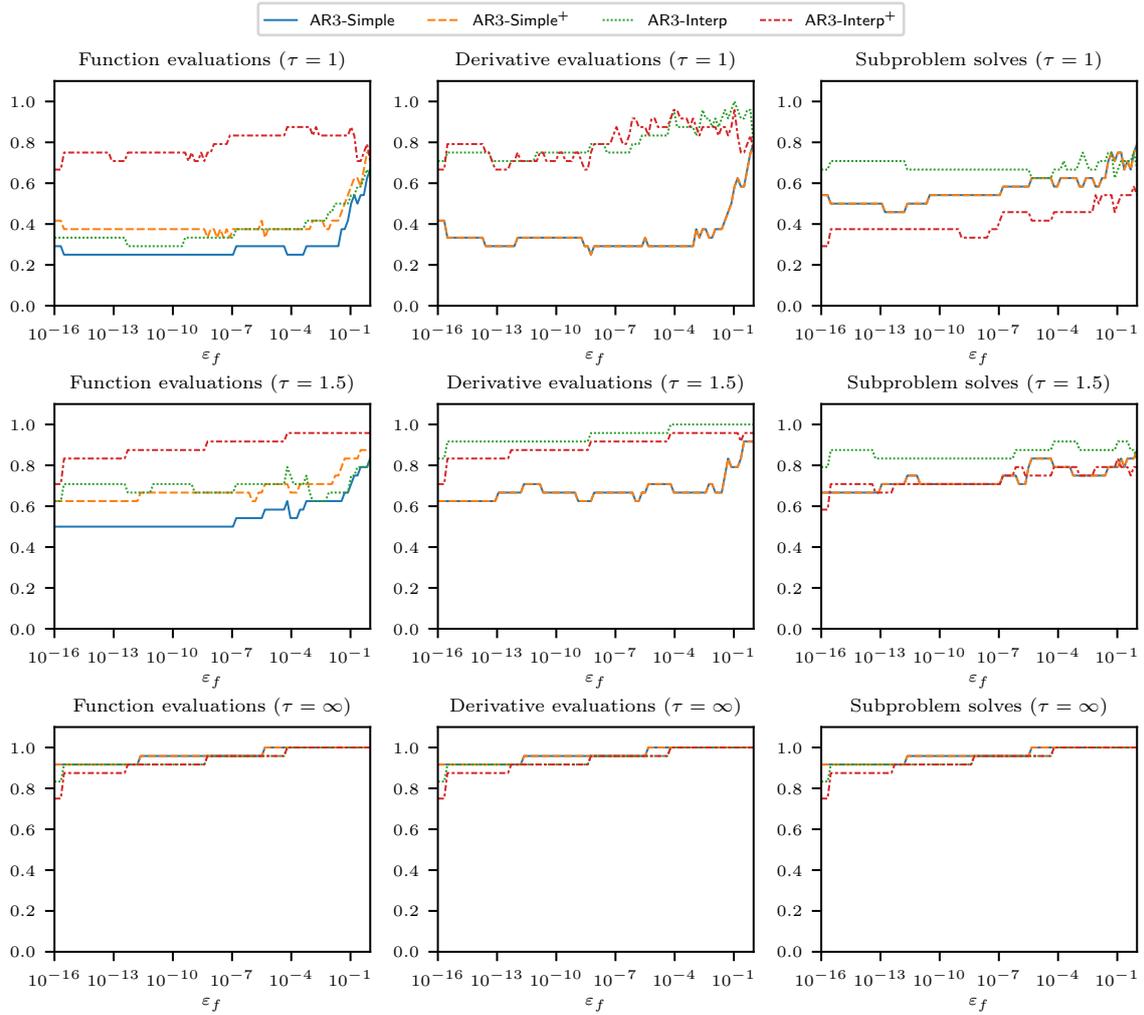}
	\caption{The performance profile plot illustrates the impact of including pre-rejection with respect to different $\tau$.}
	\label{fig:prerejection-convergence-full-eps}
\end{figure}

\subsection{Comparison with BGMS step control strategy}
\begin{figure}[H]
	\centering
	\subimport{plots/updates}{performance_profile-First_Order-tau-f.pgf}
	\caption{The performance profile plot highlights the differences between the three update strategies with respect to different $\eps_f$.
		Note that pre-rejection is enabled for all three methods.}
	\label{fig:updates-convergence-full-tau}
\end{figure}
\begin{figure}[H]
	\centering
	\subimport{plots/updates}{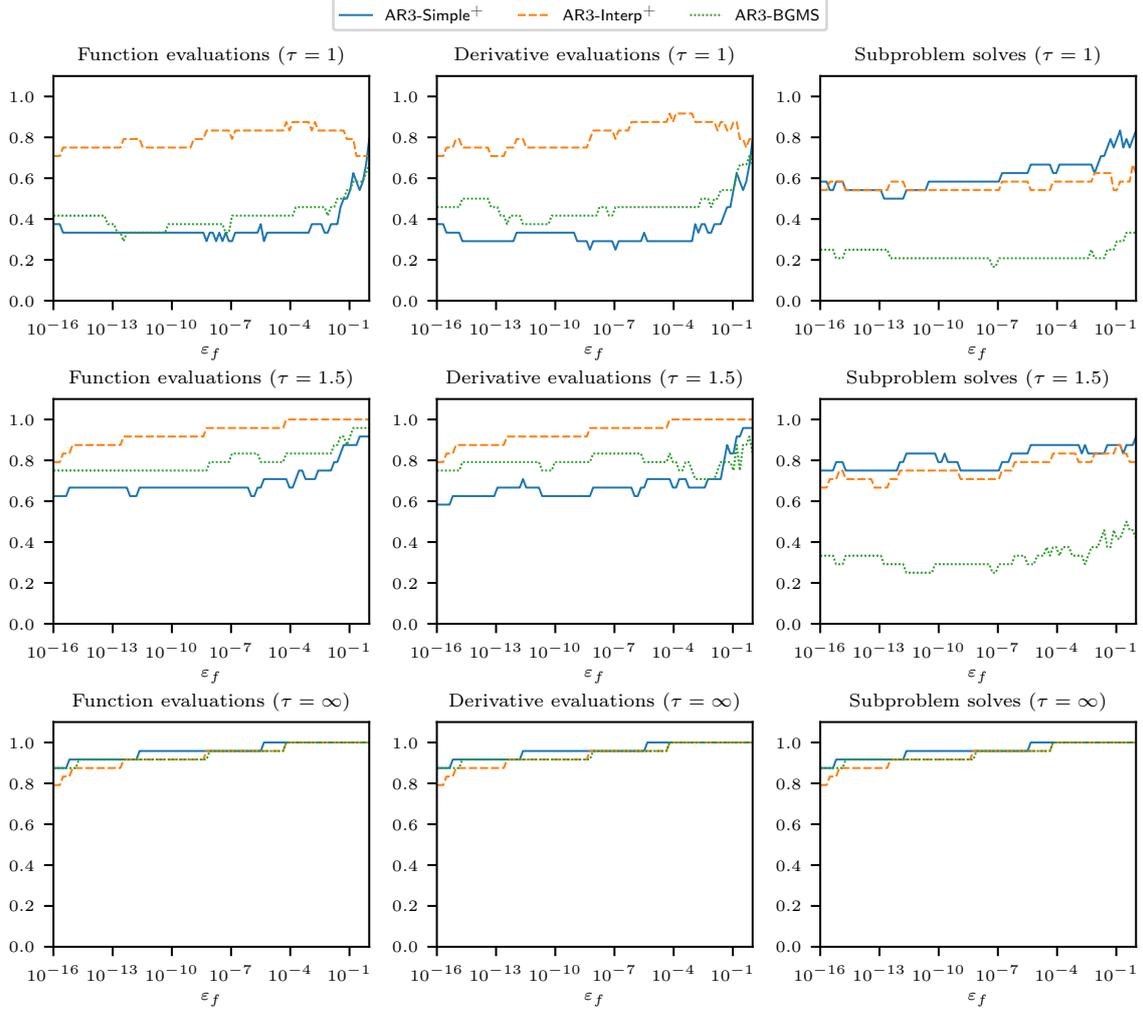}
	\caption{The performance profile plot highlights the differences between the three update strategies with respect to different $\tau$.
		Note that pre-rejection is enabled for all three methods.}
	\label{fig:updates-convergence-full-eps}
\end{figure}

\subsection{Subproblem termination condition}
\begin{figure}[H]
	\centering
	\subimport{plots/theta}{performance_profile-Interpolation_m-3-tau-f.pgf}
	\caption{The performance profile plot showing the differences between $\theta \in \{10^{-2}, 10^{0}, 10^{2}, 10^{4}\}$ in \cref{eqn:subproblem_relative_tc} with $p = 3$ with respect to different $\eps_f$.}
	\label{fig:theta-convergence-full-tau}
\end{figure}
\begin{figure}[H]
	\centering
	\subimport{plots/theta}{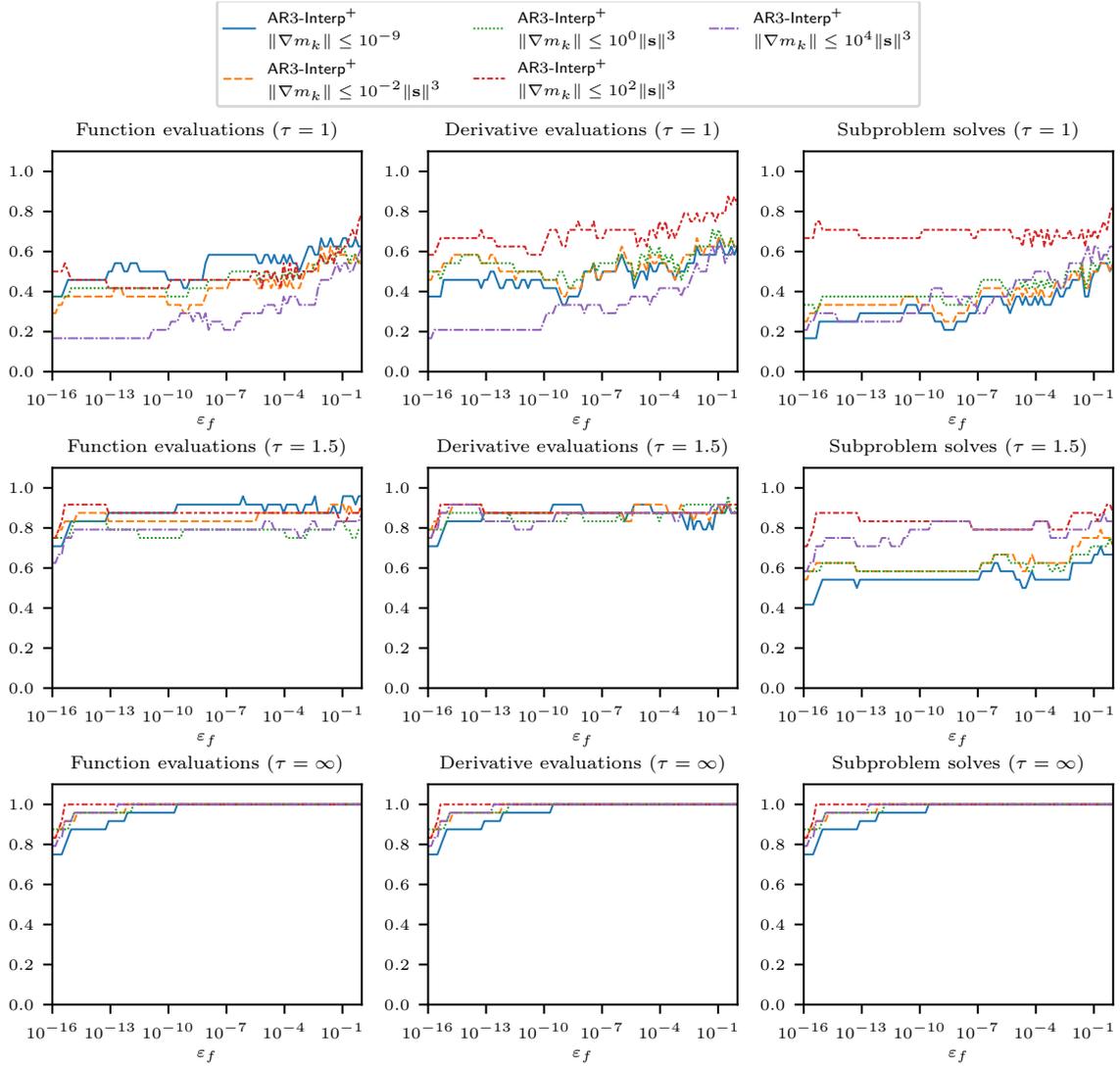}
	\caption{The performance profile plot showing the differences between $\theta \in \{10^{-2}, 10^{0}, 10^{2}, 10^{4}\}$ in \cref{eqn:subproblem_relative_tc} with $p = 3$ with respect to different $\tau$.}
	\label{fig:theta-convergence-full-eps}
\end{figure}

\subsection{Benchmarking}
\begin{figure}[H]
	\centering
	\subimport{plots/benchmark}{performance_profile-tau-f.pgf}
	\caption{The performance profile plot compares \textsf{AR$2$-Interp}, \textsf{AR$3$-Interp\textsuperscript{+}}, \textsf{AR$3$-BGMS} and \textsf{AR$3$-BGMS/Gencan} using both \cref{eqn:subproblem_absolute_tc} and \cref{eqn:subproblem_relative_tc} on full $35$ MGH test set with respect to different $\eps_f$.}
	\label{fig:benchmark-convergence-full-tau}
\end{figure}
\begin{figure}[H]
	\centering
	\subimport{plots/benchmark}{performance_profile-eps-f.pgf}
	\caption{The performance profile plot compares \textsf{AR$2$-Interp}, \textsf{AR$3$-Interp\textsuperscript{+}}, \textsf{AR$3$-BGMS} and \textsf{AR$3$-BGMS/Gencan} using both \cref{eqn:subproblem_absolute_tc} and \cref{eqn:subproblem_relative_tc} on full $35$ MGH test set with respect to different $\tau$.}
	\label{fig:benchmark-convergence-full-eps}
\end{figure}

\subsection{Hessian and Tensor Free Implementation}
\begin{figure}[H]
	\centering
	\subimport{plots/GLRT}{performance_profile-p23-tau-f.pgf}
	\caption{The performance profile plot compares all combinations of \textsf{AR$2$-Interp} and \textsf{AR$3$-Interp\textsuperscript{+}} with \textsf{MCM} and \textsf{GLRT} as AR$2$ subproblem solvers using both \cref{eqn:subproblem_absolute_tc} and \cref{eqn:subproblem_relative_tc} conditions for terminating the AR$2$ subproblem on the full $35$ MGH test set with respect to different $\eps_f$.
    For \textsf{AR$3$-Interp\textsuperscript{+}}, $\tilde{m}_k$ refers to the AR2 model constructed while minimizing the AR3 model.}
	\label{fig:glrt-convergence-full-tau}
\end{figure}
\begin{figure}[H]
	\centering
	\subimport{plots/GLRT}{performance_profile-p23-eps-f.pgf}
	\caption{The performance profile plot compares all combinations of \textsf{AR$2$-Interp} and \textsf{AR$3$-Interp\textsuperscript{+}} with \textsf{MCM} and \textsf{GLRT} as AR$2$ subproblem solvers using both \cref{eqn:subproblem_absolute_tc} and \cref{eqn:subproblem_relative_tc} conditions for terminating the AR$2$ subproblem on the full $35$ MGH test set with respect to different $\tau$.
    For \textsf{AR$3$-Interp\textsuperscript{+}}, $\tilde{m}_k$ refers to the AR2 model constructed while minimizing the AR3 model.}
	\label{fig:glrt-convergence-full-eps}
\end{figure}

\begin{sidewaysfigure}
    \centering
    \subimport{plots/GLRT}{rosenbrock_time_hvp_chol_vs_dim_log2.pgf}

    \caption{
    The leftmost plot shows the ``wall-clock'' time of using AR$1$, AR$2$ and AR$3$ on the multidimensional Rosenbrock problem as the dimension grows. For AR$2$ and AR$3$ the subproblem solvers \textsf{MCM} (with explicit Hessians and third derivatives) and \textsf{GLRT} (Hessian- and tensor-free) are used for solving the AR$2$ subproblems. They are employed directly for AR$2$ and within \textsf{AR$2$-Simple} for AR$3$. The plots in the center and on the right show the corresponding total number of Hessian–vector products for \textsf{GLRT} and the total number of Cholesky factorizations for \textsf{MCM}. Note that the total number of tensor–vector–vector products is the same as the total number of Hessian–vector products for AR$3$. A special marker denotes that the running time exceeded $4$ hours and the corresponding algorithm failed to solve the problem. These points are connect with a partially transparent line, since their slope is not indicative of the real growth. The ``outlier'' for \textsf{AR$2$-Interp + GLRT} at dimension $2^{8}$ needs a surprising amount of time, but aligns very well with the other data points in terms of total Hessian–vector products. For both AR$2$ and AR$3$, we use \cref{eqn:subproblem_relative_tc} as the inner termination condition with $\theta = 10^{\log_2(d)/4}$ where $d$ is the dimension of the problem.
    }
    \label{fig:GLRT-time-d-theta}
\end{sidewaysfigure}

\section{\texorpdfstring{Relative termination condition for AR$2$}{Relative termination condition for AR2}}
\label{sec:inexactness_AR2}
In this section, we provide the numerical illustrations for \cref{sec:theta} for AR$2$.
From \cref{fig:theta-convergence-p-2,fig:theta-performance-profile-p-2} we can conclude that the impact of the termination condition on AR$2$ is limited, with \cref{eqn:subproblem_relative_tc} and $\theta=10^{-2}$ being among the best with respect to function evaluations, derivative evaluations and subproblem solves.

\begin{sidewaysfigure}
	\centering
	\subimport{plots/theta}{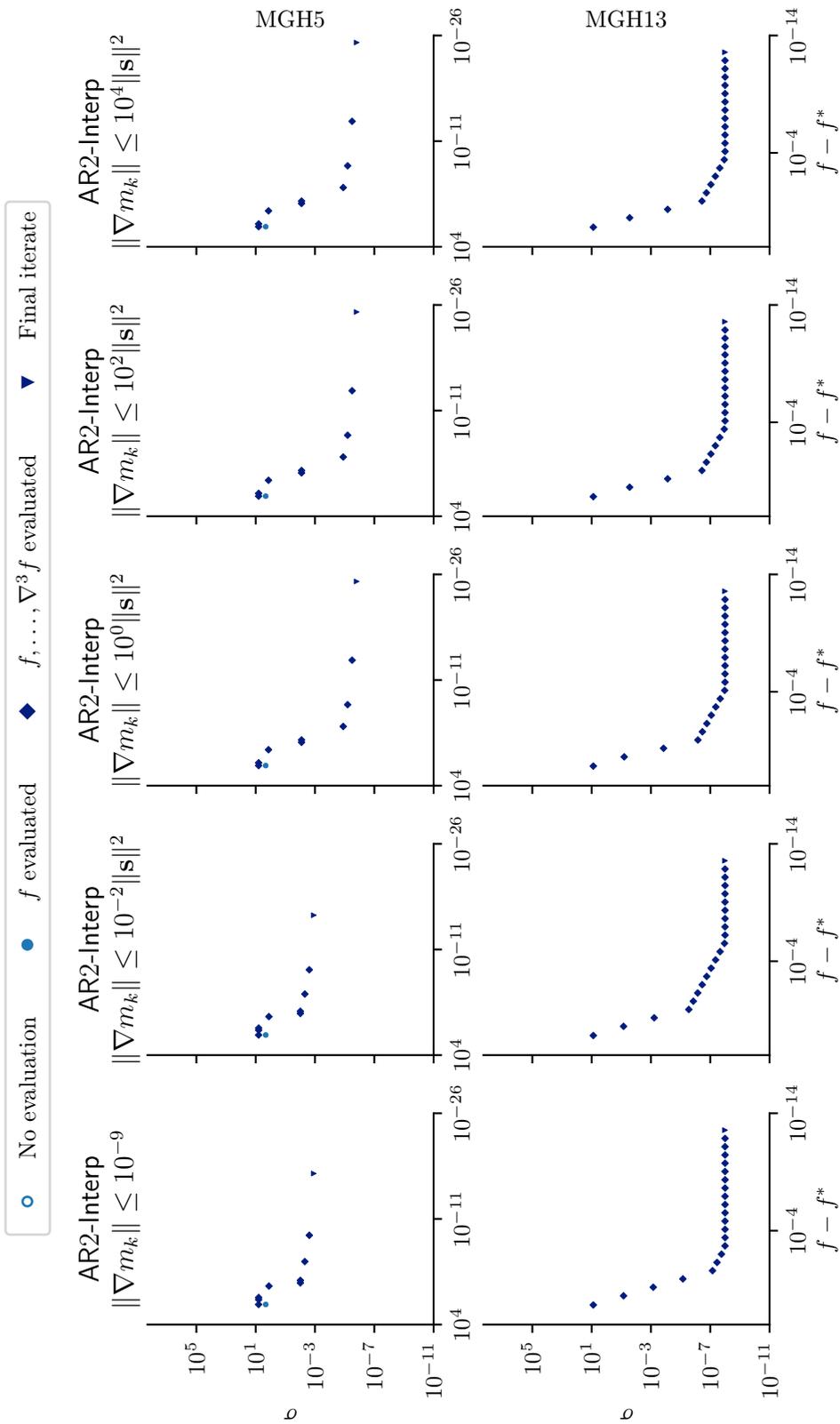}
	\caption{The convergence dot plot showing the differences between $\theta \in \{10^{-2}, 10^{0}, 10^{2}, 10^{4}\}$ in \cref{eqn:subproblem_relative_tc} for \textsf{AR$2$-Interp}. An extra condition \cref{eqn:subproblem_absolute_tc} is provided as an approximation of an exact solve.}
	\label{fig:theta-convergence-p-2}
\end{sidewaysfigure}

\begin{sidewaysfigure}
	\centering
	\subimport{plots/theta}{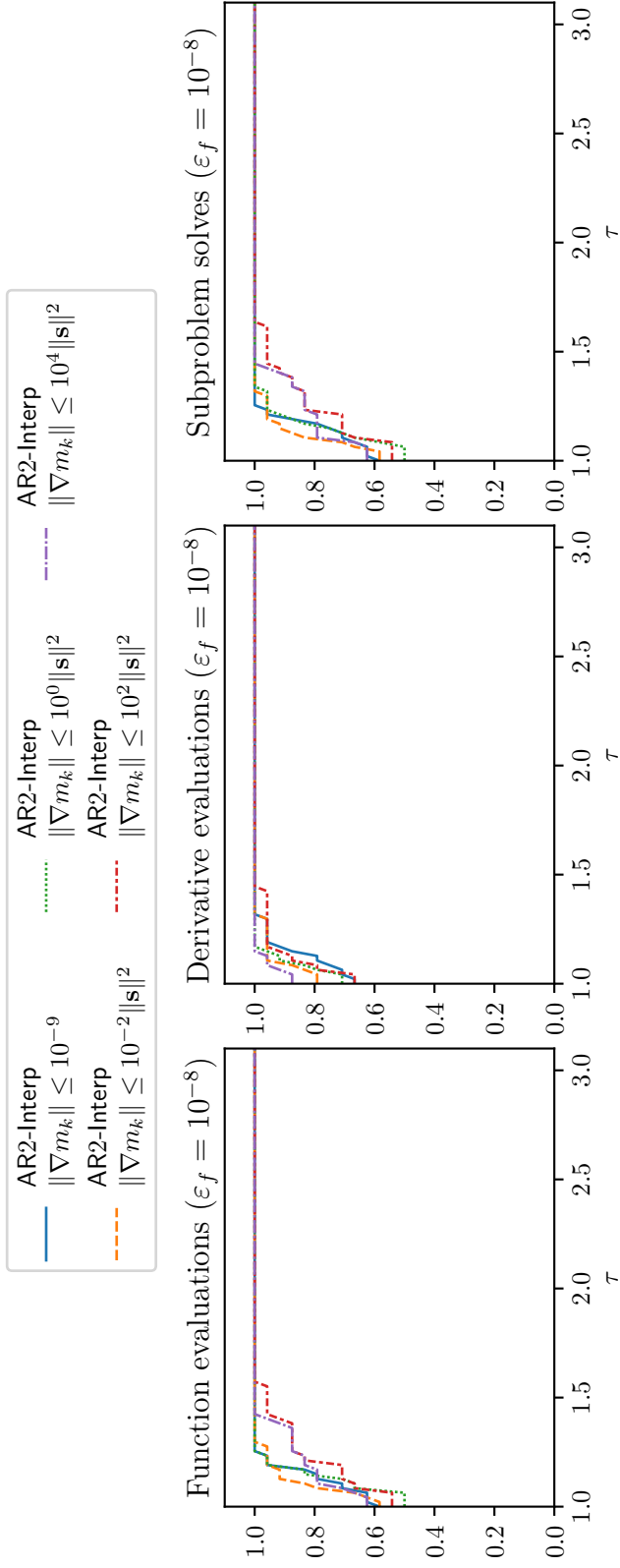}
	\caption{The performance profile plot showing the differences between $\theta \in \{10^{-2}, 10^{0}, 10^{2}, 10^{4}\}$ in \cref{eqn:subproblem_relative_tc} for \textsf{AR$2$-Interp}. An extra condition \cref{eqn:subproblem_absolute_tc} is provided as an approximation of an exact solve.}
	\label{fig:theta-performance-profile-p-2}
\end{sidewaysfigure}

\section{Generalized norm termination condition}
\label{sec:generalized_norm}
\begin{sidewaysfigure}
	\centering
	\subimport{plots/theta_GN}{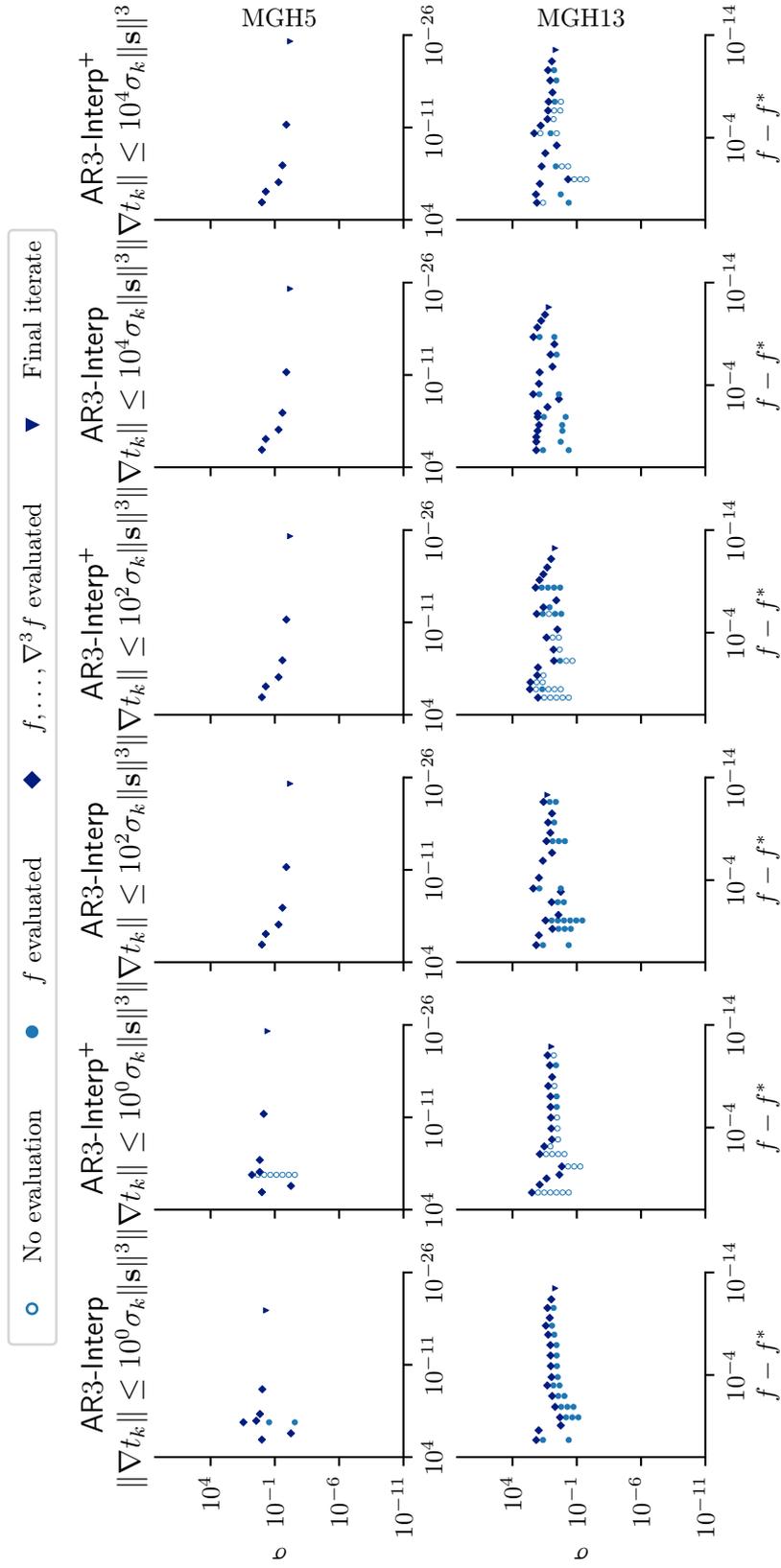}
	\caption{The convergence dot plot shows the differences between $\theta \in \{10^{0}, 10^{2}, 10^{4}\}$ in \cref{eqn:subproblem_general_norm_tc} for both \textsf{AR$3$-Interp} and \textsf{AR$3$-Interp\textsuperscript{+}}.}
	\label{fig:theta-gn-convergence}
\end{sidewaysfigure}

In this section, we consider the generalized norm termination condition \cite{gratton2023adaptive}, given by
\begin{align}
	\label{eqn:subproblem_general_norm_tc}
	\norm{\nabla t_k(\ss_k)} \leq \theta \sigma_k \norm{\ss_k}^p
	\tag{TC.g}
\end{align}
where $\theta \geq 1$.
We choose $\theta = \{1, 10^2, 10^4\}$.
It has been shown in \cite{gratton2023adaptive} that an exact solution to the subproblem satisfies \cref{eqn:subproblem_general_norm_tc} with $\theta=1$, however, the reverse is not necessarily true, which means that $\theta=1$ still represents an inexact condition.
From \cref{fig:theta-gn-convergence}, we observe similar performance compared to the results in \cref{fig:theta-convergence}.
The pre-rejection framework continues to work well with this termination condition.
When $\theta=1$, the performance of \cref{eqn:subproblem_general_norm_tc} is very similar to that shown in the leftmost subplot in \cref{fig:theta-convergence}.
However, when $\theta$ increases, no drastic improvements are observed for MGH13.

\begin{sidewaysfigure}
	\centering
	\subimport{plots/theta_GN}{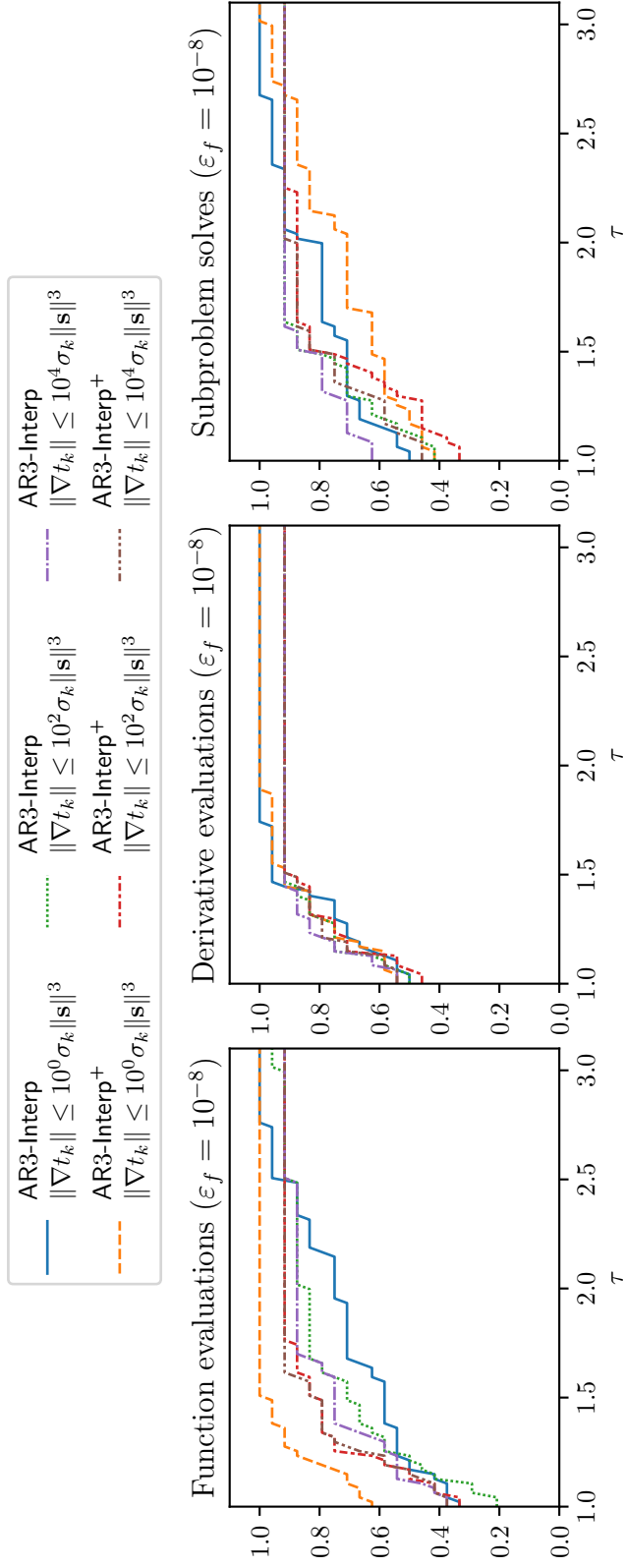}
	\caption{The performance profile plot shows the differences between $\theta \in \{10^{0}, 10^{2}, 10^{4}\}$ in \cref{eqn:subproblem_general_norm_tc} for both \textsf{AR$3$-Interp} and \textsf{AR$3$-Interp\textsuperscript{+}}.}
	\label{fig:theta-gn-performance-profile-interpolation}
\end{sidewaysfigure}

\cref{fig:theta-gn-performance-profile-interpolation} demonstrates the overall performance for different values of $\theta$. It is clear that the variant with $\theta = 1$ uses the least amount of function evaluations, whereas the variant with $\theta = 10^4$ requires the fewest subproblem solves.

\begin{sidewaysfigure}
	\centering
	\subimport{plots/theta_GN}{performance_profile-benchmark-3-tau-f-one-line.pgf}
	\caption{The performance profile benchmark plot shows the differences between \textsf{AR$3$-Simple\textsuperscript{+}}, \textsf{AR$3$-Interp\textsuperscript{+}}, and \textsf{AR$3$-BGMS} using three termination conditions: \cref{eqn:subproblem_absolute_tc} with $\eps_{\textrm{sub}} = 10^{-9}$, \cref{eqn:subproblem_relative_tc} with  $\theta=100$, and \cref{eqn:subproblem_general_norm_tc} with $\theta=1$. This experiment uses the full $35$ MGH test set.}
	\label{fig:theta-gn-performance-profile}
\end{sidewaysfigure}

Next, we maintain $\theta=1$ in \cref{eqn:subproblem_general_norm_tc} and provide the benchmark performance profile plot in \cref{fig:theta-gn-performance-profile} to determine the best combination of AR$3$ variant and termination condition.
All three \textsf{AR$3$-Interp} strategies outperform \textsf{AR$3$-Simple} and \textsf{AR$3$-BGMS} in both function and derivative evaluations. Furthermore, \textsf{AR$3$-Interp} using \cref{eqn:subproblem_relative_tc} with $\theta=100$ emerges as the most competitive candidate across all three measures.

\Cref{fig:theta-gn-convergence-p-2,fig:theta-gn-performance-profile-p-2-interpolation,fig:theta-gn-performance-profile-p-2} present the results of the same experiments as above for AR$2$.

\begin{figure}
	\centering
	\subimport{plots/theta_GN}{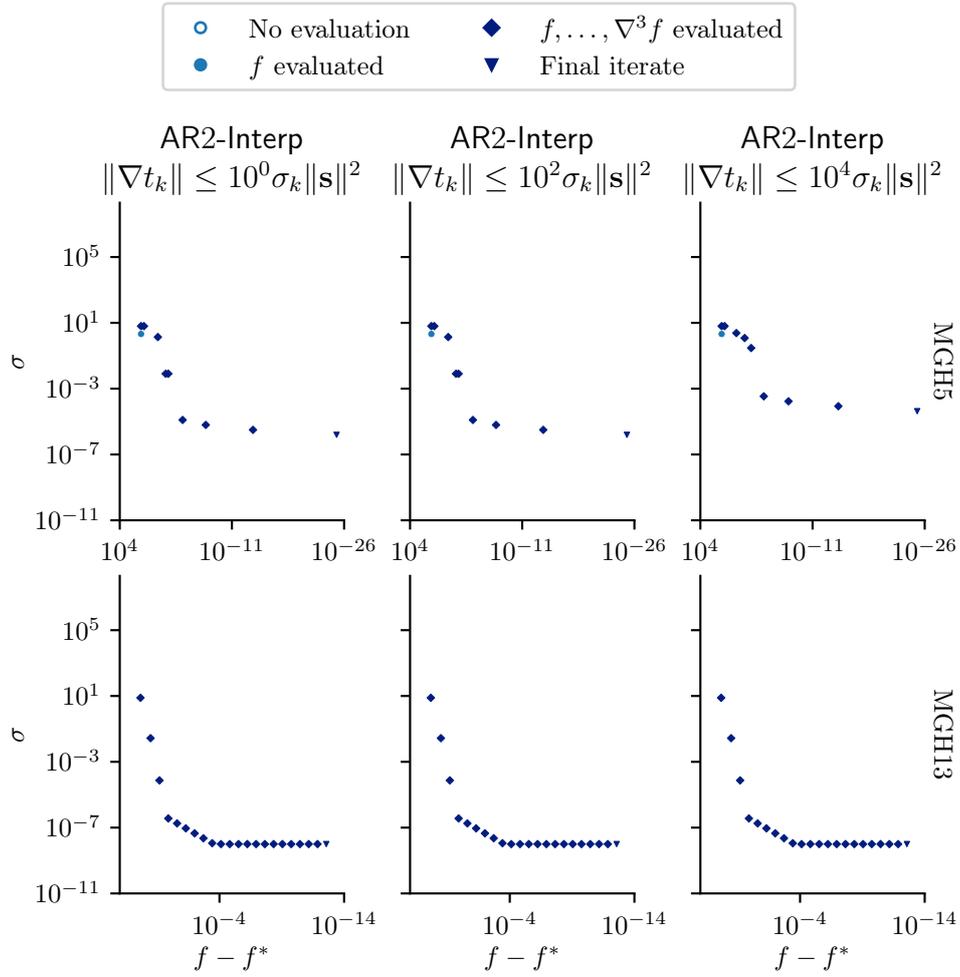}
	\caption{The convergence dot plot shows the differences between $\theta \in \{10^{0}, 10^{2}, 10^{4}\}$ in the subproblem termination condition \cref{eqn:subproblem_general_norm_tc} for \textsf{AR$2$-Interp}.}
	\label{fig:theta-gn-convergence-p-2}
\end{figure}

\begin{sidewaysfigure}
	\centering
	\subimport{plots/theta_GN}{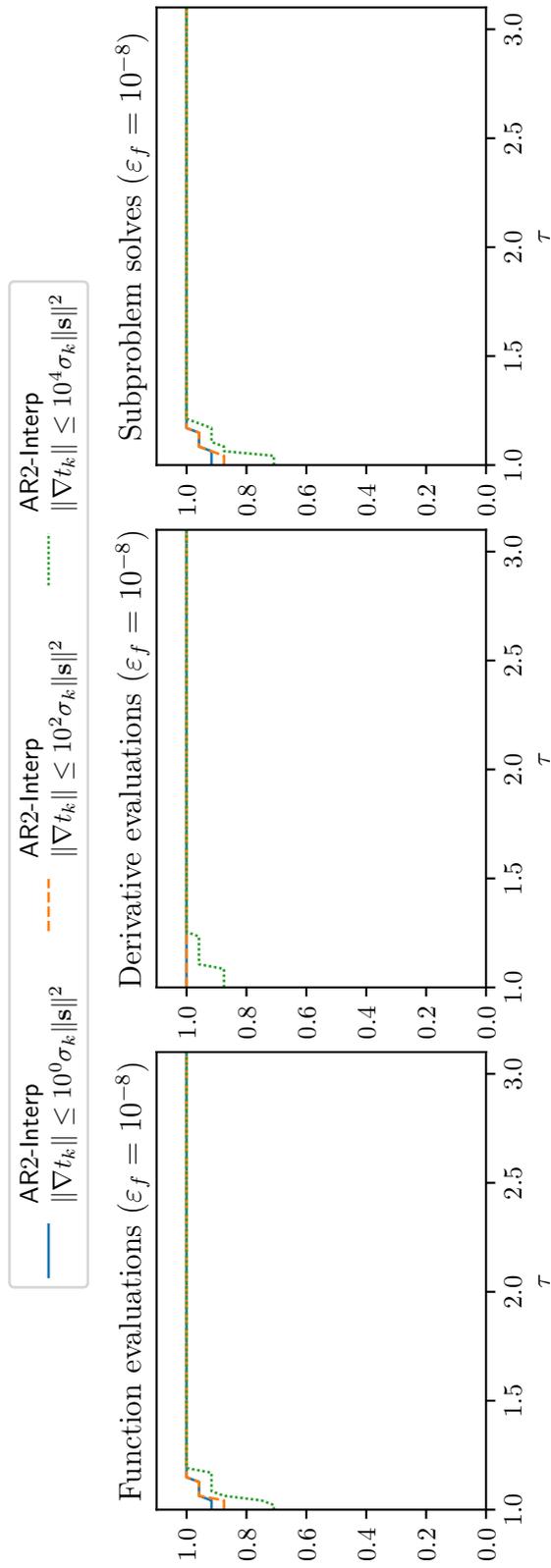}
	\caption{The performance profile plot shows the differences between $\theta \in \{10^{0}, 10^{2}, 10^{4}\}$ in the subproblem termination condition \cref{eqn:subproblem_general_norm_tc} for \textsf{AR$2$-Interp}.}
	\label{fig:theta-gn-performance-profile-p-2-interpolation}
\end{sidewaysfigure}

\begin{sidewaysfigure}
	\centering
	\subimport{plots/theta_GN}{performance_profile-benchmark-2-tau-f-one-line.pgf}
	\caption{The performance profile benchmark plot shows the differences between \textsf{AR$2$-Simple}, \textsf{AR$2$-Interp}, and \textsf{AR$2$-BGMS} using three termination conditions: \cref{eqn:subproblem_absolute_tc} with $\eps_{\textrm{sub}} = 10^{-9}$, \cref{eqn:subproblem_relative_tc} with $\theta=100$, and \cref{eqn:subproblem_general_norm_tc} with $\theta=1$. This experiment uses the full $35$ MGH test set.}
	\label{fig:theta-gn-performance-profile-p-2}
\end{sidewaysfigure}

\section{\texorpdfstring{AR$2$ versus AR$3$}{AR2 versus AR3}}
\label{sec:AR2vsAR3}
In this section, we focus on the comparison between \textsf{AR$2$-Interp} and \textsf{AR$3$-Interp\textsuperscript{+}} variants.
We provide performance profile plots for different test sets, namely:
\begin{itemize}
	\item For $18$ odd MGH test problems and $6$ additional problems in \cref{fig:2vs3-performance-profile-training},
	\item For the full $35$ problems in the MGH test set in \cref{fig:2vs3-performance-profile-benchmark},
	\item For the six additional problems in \cref{fig:2vs3-performance-profile-extra},
	\item For the four regularized third-order polynomial problems \cref{fig:2vs3-performance-profile-regularized-cubic}.
\end{itemize}
As shown in \cref{fig:2vs3-performance-profile-extra,fig:2vs3-performance-profile-regularized-cubic}, \textsf{AR$3$-Interp\textsuperscript{+}} with \cref{eqn:subproblem_absolute_tc} performed exceptionally well across all measures on the four regularized third-order polynomial problems, particularly in function evaluations. However, on a larger test set as in \cref{fig:2vs3-performance-profile-training,fig:2vs3-performance-profile-benchmark}, these advantages are less pronounced compared to \textsf{AR$3$-Interp\textsuperscript{+}} with $\theta=100$. Nonetheless, it is evident by \cref{fig:2vs3-performance-profile-training,fig:2vs3-performance-profile-benchmark,fig:2vs3-performance-profile-extra,fig:2vs3-performance-profile-regularized-cubic} that \textsf{AR$2$-Interp} is less efficient compared to \textsf{AR$3$-Interp\textsuperscript{+}}, which aligns with the observations in \cite{birgin2020use} and our example in \cref{exp:harpin_slalom}.

\begin{sidewaysfigure}
	\centering
	\subimport{plots/2vs3}{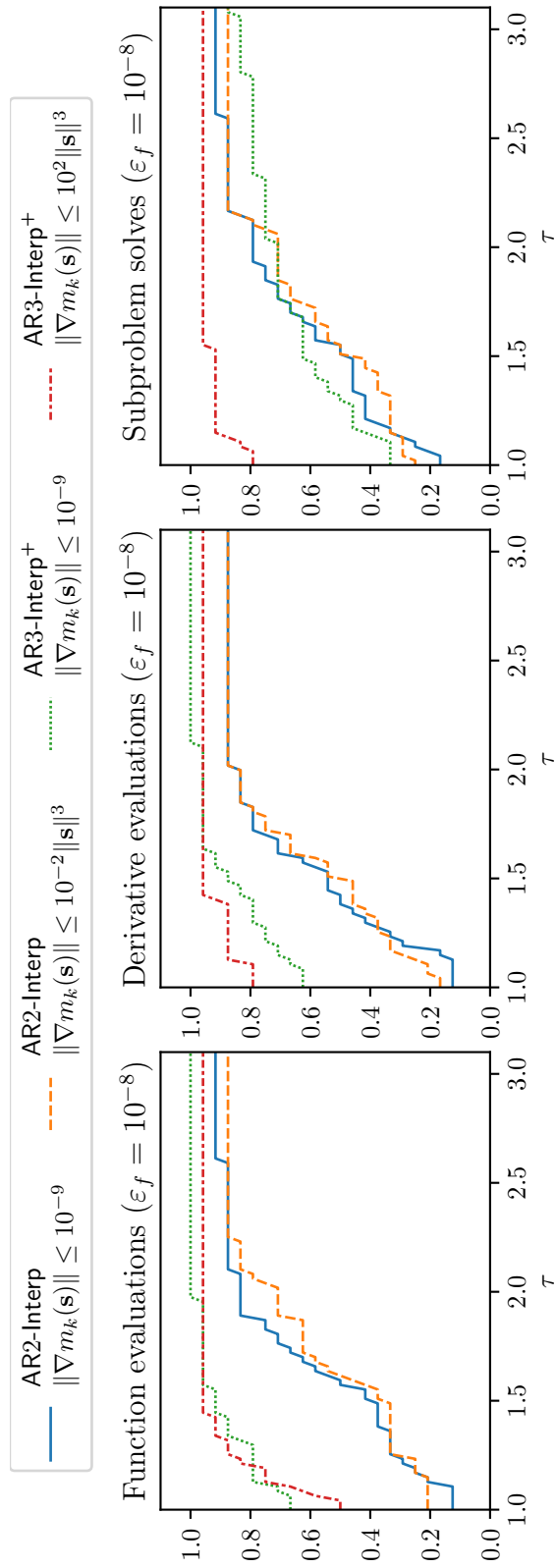}
	\caption{The performance profile plot compares the \textsf{AR$2$-Interp} and \textsf{AR$3$-Interp\textsuperscript{+}} variants using $18$ odd MGH test problems and $6$ additional test problems.}
	\label{fig:2vs3-performance-profile-training}
\end{sidewaysfigure}

\begin{sidewaysfigure}
	\centering
	\subimport{plots/2vs3}{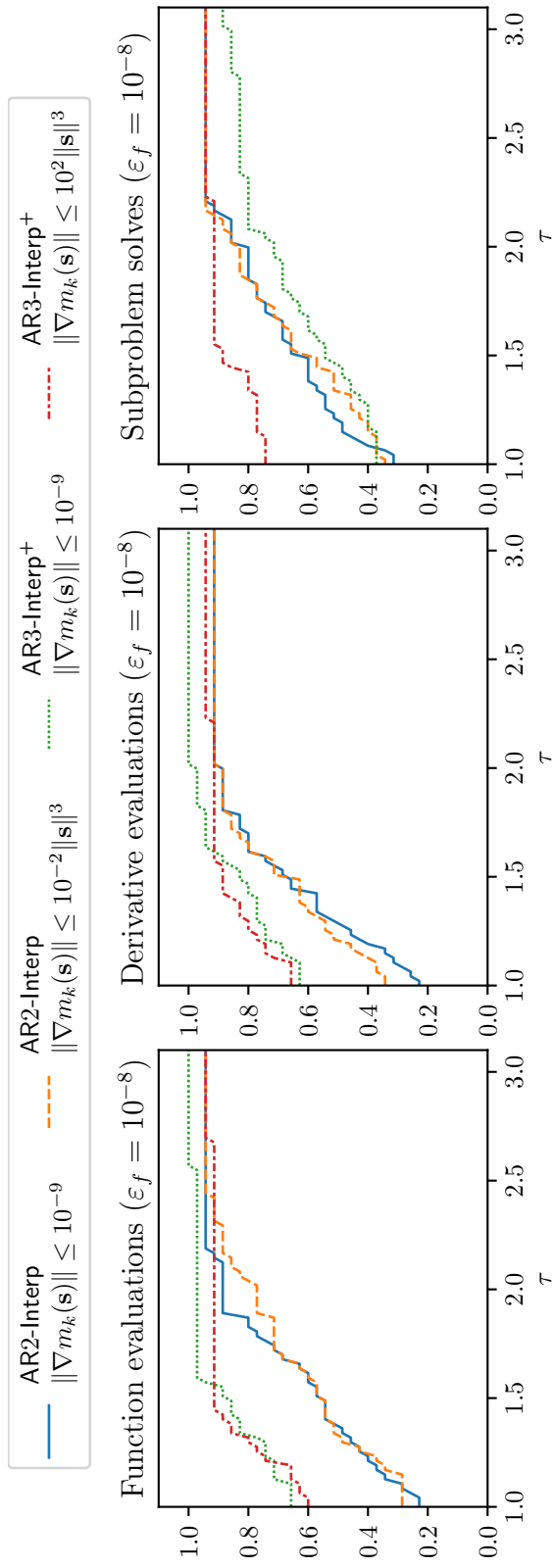}
	\caption{The performance profile plot compares the \textsf{AR$2$-Interp} and \textsf{AR$3$-Interp\textsuperscript{+}} variants using full $35$ MGH test set.}
	\label{fig:2vs3-performance-profile-benchmark}
\end{sidewaysfigure}

\begin{sidewaysfigure}
	\centering
	\subimport{plots/2vs3}{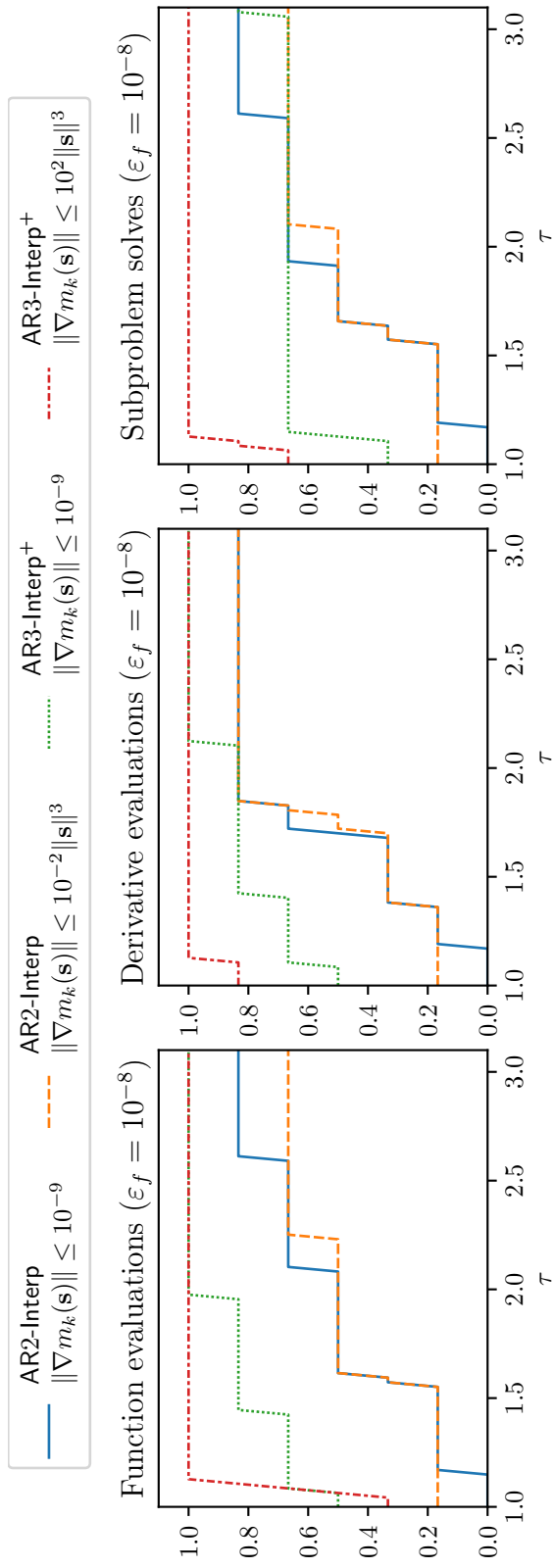}
	\caption{The performance profile plot compares the \textsf{AR$2$-Interp} and \textsf{AR$3$-Interp\textsuperscript{+}} variants using $6$ additional test problems.}
	\label{fig:2vs3-performance-profile-extra}
\end{sidewaysfigure}

\begin{sidewaysfigure}
	\centering
	\subimport{plots/2vs3}{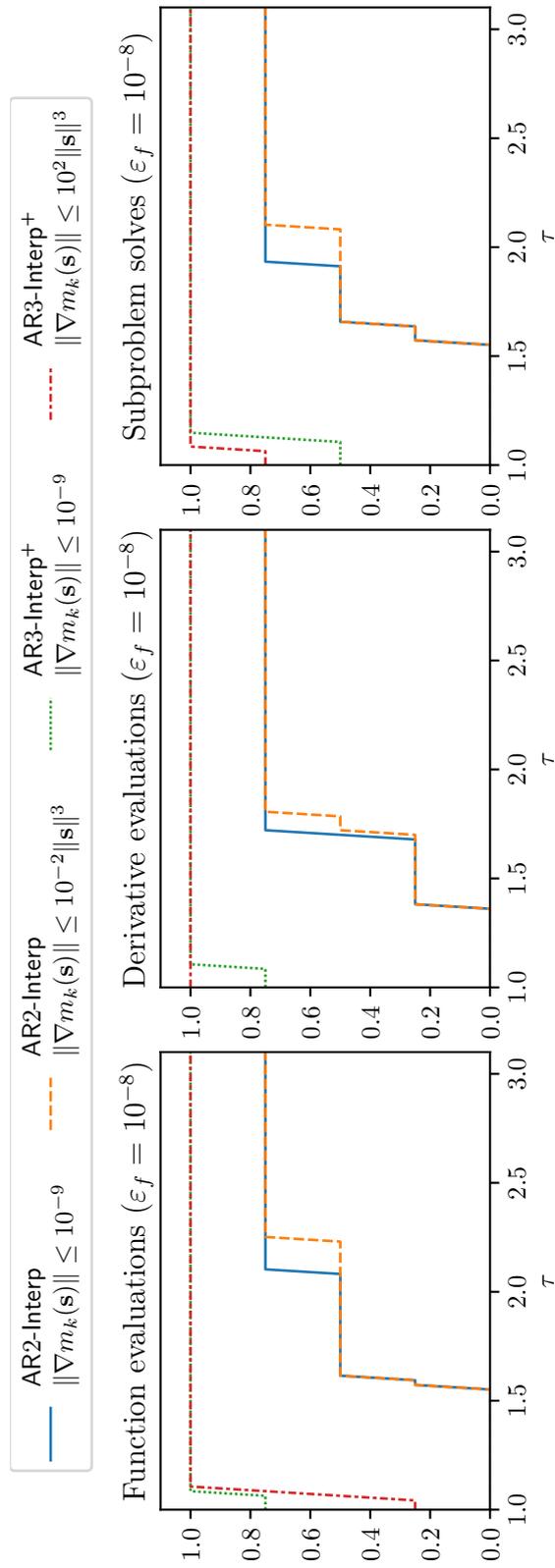}
	\caption{The performance profile plot compares the \textsf{AR$2$-Interp} and \textsf{AR$3$-Interp\textsuperscript{+}} variants using $4$ regularized third-order polynomials test problems.
    }
	\label{fig:2vs3-performance-profile-regularized-cubic}
\end{sidewaysfigure}

\end{document}